\documentclass[10pt,a4paper]{amsart}
\usepackage[utf8]{inputenc}
\usepackage{amsmath}
\usepackage{amsthm}
\usepackage[all]{xy}
\usepackage{amsfonts}
\usepackage[usenames]{xcolor}
\usepackage{amssymb}
\usepackage{float}
\usepackage{xspace}
\usepackage{pdflscape}

\definecolor{darkgreen}{rgb}{0,0.5,0} %
\usepackage[
        colorlinks, citecolor=red,
        backref,
        pdfauthor={Roberto Dvornicich, Davide Lombardo, Francesco Veneziano, Umberto Zannier},
        pdftitle={Classification of rational angles in plane lattices II},
]{hyperref}
\usepackage[alphabetic, backrefs]{amsrefs} %

\usepackage[normalem]{ulem}

\usepackage{multirow}

\usepackage{tkz-euclide}
\usetikzlibrary{shapes.geometric, intersections, calc}

\usepackage[a4paper, top=3cm, bottom=3cm, left=2.5cm, right=2.5cm]{geometry}

\newcommand{\pro}{\mathbb{P}}

\newcommand{\zee}{\mathbb{Z}}

\newcommand{\se}{symmetry-equivalent\xspace}
\newcommand{\bif}{belongs to an unbounded family\xspace}

\newcommand{\Q}{\mathbb{Q}}
\newcommand{\C}{\mathbb{C}}
\newcommand{\R}{\mathbb{R}}
\newcommand{\roots}{U}
\newcommand*\circled[1]{\tikz[baseline=(char.base)]{
            \node[shape=circle,draw,inner sep=2pt] (char) {#1};}}

\newcommand{\abs}[1]{\left| #1 \right|}
\newcommand{\generated}[1]{\left\langle #1 \right\rangle}
\newcommand{\generatedQ}[1]{\left\langle #1 \right\rangle_\mathbb{Q}}

\theoremstyle{definition}
\newtheorem{lemma}{Lemma}
\newtheorem{proposition}[lemma]{Proposition}

\newtheorem{theorem}[lemma]{Theorem}
\newtheorem{remark}[lemma]{Remark}
\newtheorem{corollary}[lemma]{Corollary}
\newtheorem{definition}[lemma]{Definition}
\numberwithin{lemma}{section}

\newcommand{\xp}{x'}
\newcommand{\yp}{y'}
\newcommand{\zp}{z'}
\newcommand{\Gal}{\operatorname{Gal}}

\title{Classification of rational angles in plane lattices II}

\author{Roberto Dvornicich}
\address{Department of mathematics, University of Pisa, Largo Bruno Pon\-te\-cor\-vo~5, 56127 Pisa, Italy}
\email{roberto.dvornicich@unipi.it}

\author{Davide Lombardo}
\address{Department of mathematics, University of Pisa, Largo Bruno Pon\-te\-cor\-vo~5, 56127 Pisa, Italy}
\email{davide.lombardo@unipi.it}

\author{Francesco Veneziano}
\address{Department of mathematics, University of Genova, Via Dodecaneso~35, 16146 Ge\-no\-va, Italy}
\email{veneziano@dima.unige.it}

\author{Umberto Zannier}
\address{Scuola Normale Superiore, Piazza dei Cavalieri 7, 56126 Pisa, Italy}
\email{umberto.zannier@sns.it}


\begin{document}

\begin{abstract}
This paper is a continuation of an earlier one, and completes a classification of the configurations of points in a plane lattice that determine angles that are rational multiples of $\pi$.
We give a complete and explicit description of lattices according to which of these configurations can be found among their points.
\end{abstract}

\maketitle

In the recent paper \cite{DvoVenZan} we considered the problem of classifying lattices (of rank $2$) in the plane $\R^2$ and sets of ordered triples of points of the lattice defining  angles with amplitudes which are {\it rational } multiples of $\pi$.  We will call them `rational angles'.  (We notice at once that the middle point in any such triple can be assumed to be the origin.)

\medskip

This kind of problem appeared to us rather spontaneously, although it is treated only in a scattered form in the existing literature, and is actually almost absent from it (we found this fact somewhat surprising).  But at least the problem is a natural generalisation of some old investigations, some of which are recalled and summarised in \cite{DvoVenZan}; for example, those on regular polygons with vertices on a plane lattice. At the end of the introduction we shall give a further motivation: the description of closed polygons with vertices on a given lattice and all angles rational.

\medskip

In fact, we are prompted to point out at once that the problem itself has quite a number of different aspects: for instance, we can vary both the lattice and the points; and we can ask both for all possible configurations of points that determine rational angles in a given lattice, and  also for an exact description of the lattices for which a given configuration is possible (and is `maximal' for that lattice). Furthermore, when both the lattice and the configuration are given, we can ask for a description of all the systems of points in that lattice that give copies of the given configuration.

All of this is explained in detail in the first paper \cite{DvoVenZan}, where we give a sort of `list' of different questions that can be asked, and where we outline a programme for obtaining complete answers to them. In that paper we also carried out many of the steps in this programme; let us now recall the strategy and the main issues that arise.

\subsection{Trigonometric diophantine equations} It turns out that any kind of `complete' classification, whatever that may mean, leads in the first place to certain (rather laborious) arithmetical questions involving {\it algebraic equations to be solved in unknown roots of unity}. 

This is a rather well-known topic in itself, closely related to (and in a sense originating from) the so-called {\it Manin-Mumford conjecture} and certain analogous problems formulated by Lang, which led to extensive research in several directions (see for example the book by the fourth author \cite{Z.UI} and the survey paper \cite{Z.T}).
But in fact equations to be solved in roots of unity had also appeared independently, already in Gordan's 1877 investigations on finite subgroups of $\mathrm{PGL}_2(\C)$. They were then studied systematically by other authors, in particular by Mann \cite{Mann} and a decade later by Conway and Jones \cite{ConwayJones}. Further results and methods were then developed by other authors (we cite, for example, the paper \cite{DZred} by the first and last author of the present work).

In particular, the paper \cite{ConwayJones} by Conway and Jones improved on some of Mann's (crucial) estimates for the maximum order of the roots of unity appearing in a suitably normalised equation. More generally, their analysis also covered what they called {\it trigonometric diophantine equations}, meaning, roughly speaking, algebraic equations (or systems of algebraic equations) in variables $x_1,\ldots ,x_m, y_1,\ldots ,y_n$, to be solved with rational (or integral) $x_i$ and roots of unity $y_j$. The terminology is, of course, motivated by the fact that these latter constraints amount to the {\it rationality} of the {\it angles} $\theta_j\in\R/\zee$ when we write $y_j=\exp(2\pi \sqrt{-1}\theta_j)$.

In fact, the work of Conway and Jones provided a formal theory and a kind of algorithm for reducing any trigonometric diophantine equation (or system) to an ordinary diophantine system, namely by getting rid of the variables $y_i$ in roots of unity. In particular, this procedure allows one to `solve' any given algebraic equation (e.g., defined over $\Q$) purely in roots of unity, but is actually more general than this.\footnote{The possibility of such a reduction was already implicit in Mann's inequalities, but the paper by Conway and Jones emphasises this aspect and provides a formal treatment of it.}

\subsection{About our previous paper and the current one} 
In fact, the problem we are considering here turned out to be a {\it mixed problem} of exactly this sort, namely falling within the scope of Conway and Jones's theory. After reducing the problem to a trigonometric diophantine equation, the two main steps of our previous work on this subject were essentially as follows:

\begin{enumerate}
\item to bound the orders of the roots of unity that can appear corresponding to `nontrivial` configurations, and 
\item to classify the possible configurations according to these orders. 
\end{enumerate}

We fully described the configurations where the orders are unbounded, and showed that they can be suitably parametrised.
Let us comment on these steps.

\medskip

{\it First goal}. The bounds mentioned above were fully explicit. Nevertheless, the number of cases to be analysed individually was so large that part of the resulting classification, though possible in principle, was not carried out in full detail. Our first objective in the present paper is to address this aspect: we refine some of the previous estimates and complete the computational part of the proof.

\medskip

{\it Second goal}. More interestingly, already in the earlier paper it turned out that a few such configurations, corresponding to certain specific orders of the roots of unity involved, could be parametrised by rational points on certain curves of genus up to $5$. This fact truly gives our basic problem a doubly diophantine aspect.

It is well known that the rational points on (smooth) curves can be parametrised by $\pro_1(\Q)$ (genus $0$), or form a group of finite rank (genus $1$), or form a finite set (genus $>1$). In principle, however, only the first of these cases can be solved effectively with today's knowledge. Hence, even though after the results of the paper \cite{DvoVenZan} the curves in question were known (and thus the classification was in a sense complete), it was not entirely clear whether all solutions could be written down `explicitly'.

Our second purpose in this paper is to carry out this step and complete the classification in a fully explicit way. In fact, the particular structure of the curves in question allows us to (a) compute the Mordell-Weil groups of the genus-1 curves and (b) compute the finitely many remaining finite sets of rational points. This is done by a combination of methods, the last step using the Chabauty-Skolem methods, see \S 8. We remark that the chain of reductions which allowed us to reduce the problem from curves of genus up to 5 to curves of genus 2 for which known algorithmic methods apply was not guaranteed, \textit{a priori}, to exist.

\subsection{Final comments} Our discussion of the situation fits well within the scope of the theory of Conway and Jones mentioned above.

As far as we know, this is the first `practical' example where some naturally given trigonometric diophantine equations lead to {\it non-trivial} ordinary diophantine equations. 
Indeed, to the best of our knowledge, in all the previous applications of the theory of Mann and/or Conway and Jones, the main difficulties lay, so to speak, in the `trigonometric part' of the problem, while the purely diophantine questions turned out to be of a milder nature, usually yielding an overdetermined system even in complex numbers.\footnote{In fact, this also happens in some of the cases analysed here, but other cases lead to the high-genus curves already mentioned.}
See for instance the paper \cite{Poonenetal} for a remarkable recent example of a geometric problem (already raised by Conway and Jones) leading to equations in roots of unity. In \cite{Poonenetal} ordinary diophantine equations do not appear to a large extent, while the computational complexity in dealing with the trigonometric part is more substantial. The authors use several tools, also of a theoretical nature, to reduce this complexity, and many more subtle ideas and devices are needed to achieve the desired goal.

\subsection{A further curious motivation} As we hinted above, we recently realised that the subject of these two papers of ours is closely related to the subject of {\it polygons with rational angles} (which are relevant in other areas, such as the theory of billiards). 

We formally state this link, which is of a very elementary nature, in the following proposition, where by \emph{rational $n$-tuple} we mean a set of $n$ elements of a lattice of rank $2$ in $\R^2$ (which we identify with $\C$) such that any two of them form a rational angle with the origin.
In fact, as in \cite{DvoVenZan}, it is easy to see that one can equivalently work with the vector space $V$ generated by the lattice over $\Q$. We call such a vector space $V$ simply a \textit{space}. We have:

\begin{proposition}
 A space $V$ as above contains a rational $n$-tuple if and only if it contains an $n$-gon with rational angles and no pair of parallel sides.
\end{proposition}
\begin{proof}
 If $p_1,\dotsc,p_n\in V$ are the vertices of such an $n$-gon, then setting $v_i=p_{i+1}-p_i$ for $i=1,\dotsc, n$ (interpreting $p_{n+1}$ as $p_1$) gives a rational $n$-tuple. Indeed, each $v_i$ lies in $V$, the angle between $v_{i-1}$ and $v_i$ is equal to $p_{i-1}\widehat{p_i}p_{i+1}$, and the fact that there is no pair of parallel sides implies that each pair of $v_i$ is $\R$-linearly independent.
 
 Conversely, assume that $v_0,\dotsc,v_{n-1}$ is a rational $n$-tuple. Without loss of generality we can assume that the $v_i$ are contained in the same half-plane and are ordered by increasing argument.
 Set now $p_1=0$ and $p_{i+1}=p_i+v_{i}$ for $i=1,\dotsc, n-2$, so that $p_{n-1}=v_1+v_2+\dotsb +v_{n-2}$. The point $p_{n-1}\in V$ lies in the angle generated by $v_0,v_n$, so that $p_{n-1}=av_0+bv_n$ for some positive rational numbers $a,b$. Set $p_n=bv_n$. Then $p_1p_2\dotsm p_n$ is a rational $n$-gon. The angle between two consecutive sides is the same as that between two of the $v_i$, and one can see that the polygon does not cross itself by checking that the coordinates of the sequence $p_i$ $i=1,\dotsc, n-1$ with respect to the basis $v_0,v_n$ form two increasing sequences. 
\end{proof}

\subsection{Proofs and computations}
Once the theoretical method has been established, the proofs of the main results are carried out with the help of a computer. We describe our approach in Section~\ref{sec: notation and conventions}, after recalling some necessary preliminaries from \cite{DvoVenZan} and setting up our notation, and give details in the remaining sections. 

We immediately point out that we develop a rather general strategy to handle trigonometric-diophantine equations which greatly reduces the combinatorial difficulty of analysing all possible `vanishing subsums' of a given equation. In particular, Proposition~\ref{prop:Subsums} allows us to check certain properties not for every decomposition of the given equation into vanishing subsums, but rather for every subset of its monomials of a fixed size. For a sense of how the proposition can be applied, and of the enormous improvement that results from this approach, we refer the reader to Proposition~\ref{prop:LongSubsum3plus2}: the computation that proves this result involves checking the 2,002 subsets of cardinality 9 of a set of 14 elements. If we had to consider all possible decompositions into vanishing subsums, we would instead have to loop over all possible 190,899,322 partitions of a set of $14$ elements. Even more strikingly, Proposition~\ref{prop:LongSubsum222} is proved by considering 8,436,285 subsets, rather than the approximately $5.5 \cdot 10^{20}$ partitions of a set of 27 elements. In future work, we plan to apply this strategy to other trigonometric-diophantine equations: for example, the main conjecture of \cite{wang2023rational} should be amenable to our methods.

The main computations were carried out in MAGMA V2.28-4 \cite{MAGMA}; the corresponding source code and output files are available at 

\begin{center}
\url{https://github.com/DavideLombardoMath/rational-angles-lattices}.
\end{center}
All the computer calculations required for this paper took less than 1.5 hours in total on a 10-core  Intel Core i9-10900KF CPU @ 3.70GHz with 32GiB of RAM.

\section*{Acknowledgements}
We thank Alessandra Caraceni for her help with the euclidean proof presented in the appendix.
We thank Evan O'Dorney for his comments and his thorough Math Review of the first part of this work.
The second and third authors are members of the GNSAGA group of the Istituto Nazionale di Alta Matematica.

\section{Notation and conventions}\label{sec: notation and conventions}
 We recall here the main definitions and results from \cite{DvoVenZan}.

For each $n\geq1$ we write $\zeta_n$ for a primitive $n$-th root of unity, and we denote by $\roots$ the set of all complex roots of unity. 

We call \emph{space} the $\mathbb{Q}$-span of two $\mathbb{R}$-linearly independent elements of $\mathbb{C}$. 

We will say that two spaces $V_1,V_2$ are \textit{homothetic}
if one is sent to the other by a homothety of the form $z\mapsto \lambda z$ for some fixed $\lambda\in\C^*$. We will say that they are \textit{equivalent} %
if $V_1$ is homothetic to $V_2$ or its complex conjugate $\overline{V_2}$. The configurations of rational angles appearing in a lattice $\Lambda$ depend only on the corresponding space $V=\Lambda \otimes \mathbb{Q}$, and in fact only on its homothety class. For this reason, we can -- whenever desired -- assume that the lattice (or the corresponding space) is of the form $\langle 1, \tau \rangle$ for a suitable $\tau \in \mathbb{C}$.

If $V$ is a space, we call \emph{rational angle} in $V$ a pair $(v_1,v_2)$ of $\mathbb{R}$-linearly independent elements of $V$ such that the argument of $v_1/v_2$ is a rational multiple of $\pi$. Swapping $v_1$ and $v_2$ or multiplying them by non-zero rational numbers defines the same angle. More generally, we call \emph{rational $n$-tuple} an $n$-tuple of elements of $V$ such that any two of them form a rational angle.

We say that a space is of type \circled{$n$} if it contains a rational $n$-tuple, and we extend this notation additively by saying that it is of type \circled{$n$}+\circled{$m$} if it contains a rational $n$-tuple and a disjoint rational $m$-tuple, and so on. There is an obvious partial order on the possible types of a given space, and we say that $V$ is of a certain \textit{exact} type if that type is maximal for $V$.

In \cite{DvoVenZan} it was proved that the existence of two independent rational angles leads to algebraic equations for the generator $\tau$ of the normalised lattice $\langle 1, \tau \rangle$. The shape of these equations depends on whether the two angles are adjacent. Having three independent rational angles then leads to two equations for the generator of the lattice, from which one can recover a diophantine-trigonometric equation whose exact form depends on the relative position of the three angles. This means that spaces of types $\circled{4},\circled{3}+\circled{2},\text{ and }3\circled{2}$ arise from different equations which have been worked out in \cite{DvoVenZan}*{§3} and should be studied separately, although using the same techniques.

For a rational angle $(v_1,v_2)$ the root of unity involved in the diophantine-trigonometric equation has an argument which is double to the geometric amplitude of the angle between $v_1$ and $v_2$; that is, if $\mu$ is a root of unity such that $\mu v_1/v_2 \in \R$, then the equation only involves $\mu^2$. We will sometimes refer to the root of unity $\mu^2$ as the \emph{squared amplitude} of the angle.

Spaces of type \circled{4} have already been fully classified in \cite{DvoVenZan}, spaces of type \circled{3}+\circled{2} will be classified in Section~\ref{sect:32}, and type $3\circled{2}$, the most general, will take up most of the effort of this paper.
The corresponding algebraic equation is represented by a polynomial $P\in\zee[a,b,c,d,x,y,z]$ which is of degree 2 in each of the variables $x,y,z$ and homogeneous of degree 4 in the variables $a,b,c,d$. The geometric interpretation of the solutions to the equation $P=0$ is described in \S \ref{subsec:corresp}.

The 27 terms of the polynomial $P$ are listed in the following table, where each term is described as a monomial in $x,y,z$ multiplied by a coefficient in $a,b,c,d$.
\begin{equation}
\arraycolsep=1.5pt\def\arraystretch{2}
\label{table:BigPoly}
\begin{array}{|c|l|}
\hline
\text{Monomials}        &       \text{Coefficient}                                     \\
\hline
 x^2y^2z^2,1        &       -b (a - c) (b - d) d                                    \\
 x^2y^2z,z          &       b (a b c + a b d - 2 a c d - 2 b c d + c^2 d + c d^2)   \\
 x^2yz^2,y          &       d (a^2 b + a b^2 - 2 a b c - 2 a b d + a c d + b c d)   \\
 xy^2z^2,x          &       (a - c) (b - d) (a b + c d)                             \\
 x^2y^2,z^2         &       -b (b - c) c (a - d)                                    \\
 x^2yz,yz           &         -a^2 b c - a b^2 c - a^2 b d - a b^2 d + 8 a b c d - a c^2 d - 
 b c^2 d - a c d^2 - b c d^2\\
 x^2z^2,y^2         &       -a (b - c) (a - d) d                                    \\
 xy^2z,xz           &       \begin{aligned}-2 a^2 b^2 + a^2 b c + a b^2 c - 2 a b c^2 + a^2 b d + a b^2 d +\\+ 
 a c^2 d + b c^2 d - 2 a b d^2 + a c d^2 + b c d^2 - 2 c^2 d^2\end{aligned}\\
 xyz^2,xy           &       \begin{aligned}-2 a^2 b^2 + a^2 b c + a b^2 c + a^2 b d + a b^2 d - 2 a^2 c d +\\- 
 2 b^2 c d + a c^2 d + b c^2 d + a c d^2 + b c d^2 - 2 c^2 d^2\end{aligned}\\
 y^2z^2,x^2         &       -a (a - c) c (b - d)                                    \\
 x^2y,yz^2          &       c (a^2 b + a b^2 - 2 a b c - 2 a b d + a c d + b c d)   \\
 x^2z,y^2z          &       a (a b c + a b d - 2 a c d - 2 b c d + c^2 d + c d^2)   \\
 xy^2,xz^2          &       (b - c) (a - d) (a b + c d)                             \\
 xyz                &       \begin{aligned}2 (2 a^2 b^2 - a^2 b c - a b^2 c + 2 a b c^2 - a^2 b d - a b^2 d + 
   2 a^2 c d - 4 a b c d +\\+ 2 b^2 c d - a c^2 d - b c^2 d + 2 a b d^2 - a c d^2 - b c d^2 + 2 c^2 d^2)\end{aligned}\\
 \hline
\end{array}
\end{equation}

Classifying spaces of type 3\circled{2} is essentially equivalent to the determination of all solutions to the equation
\begin{equation}\label{eqn:main}
 P(a,b,c,d,x,y,z)=0
\end{equation}
with $x,y,z\in\roots$ and $(a:b:c:d)\in\pro_3(\Q)$ satisfying the non-degeneracy conditions
\begin{equation}\label{eq:Conditions}
 x,y,z\neq 1,\qquad a,b,c,d \text{ distinct and non-zero.}
\end{equation}
In the rest of the paper, whenever we speak of solutions to \eqref{eqn:main}, we implicitly mean that conditions \eqref{eq:Conditions} are also satisfied.

\subsection{Correspondence between algebraic solutions and geometric configurations}\label{subsec:corresp}
Given a solution to \eqref{eqn:main} such that $ab\neq cd$, we can set
\begin{equation}\label{eq:tauabcd}
 \tau=\frac{(cd-ab)x(y-1)(z-1)}{a(y-x)(z-1)+b(xy-1)(z-1)-c(z-x)(y-1)-d(xz-1)(y-1)};
\end{equation}
then $V=\generatedQ{1,\tau}$ gives a space with three rational angles $(1,\tau),(\tau+a,\tau+b),(\tau+c,\tau+d)$.

Solutions to $\eqref{eqn:main}$ such that $ab=cd$ do not necessarily determine a space; this case occurs when the two quadratic equations for $\tau$ \cite{DvoVenZan}*{3.11} arising from the two rational angles $(\tau+a,\tau+b),(\tau+c,\tau+d)$ are proportional. This case will be treated separately in Section~\ref{sect:abcd}.

\subsection{Symmetries of the equation}
If $(a,b,c,d,x,y,z)$ is a solution of \eqref{eqn:main} so are
\begin{align}\label{simmetrie}
\begin{split}
 &\left(b,a,c,d,x,y^{-1},z\right)\\
 &\left(a,b,d,c,x,y,z^{-1}\right)\\
 &\left(\frac{1}{a},\frac{1}{b},\frac{1}{c},\frac{1}{d},x^{-1},y,z\right)\\
 &\left(c,d,a,b,x,z,y\right)\\
 &\left(b(b-c)(b-d),a(b-c)(b-d),b(a-c)(d-b),b(a-d)(c-b),y,x,z\right).
 \end{split}
\end{align}
The condition $ab\neq cd$ is preserved by these substitutions. Moreover, if $ab \neq cd$ holds, the values of $\tau$ given by~\eqref{eq:tauabcd} before and after the substitution produce two homothetic spaces.

Let $G$ be the group of signed $3 \times 3$ permutation matrices. The formulas \eqref{simmetrie} clearly induce an action of $G$ on the set of solutions of \eqref{eqn:main} by considering the action on $x, y, z$.

We will call the least common multiple of the multiplicative orders of $x,y,z$ their \textit{common order}. Notice in particular that $x,y,z$ all belong to the cyclotomic field $\mathbb{Q}(\zeta_n)$, where $n$ is their common order.

Let now $(a,b,c,d,x,y,z)$ be a solution to Equation \eqref{eqn:main} for which $x,y,z$ have common order $n$. For any $\sigma \in \operatorname{Gal}\left( \mathbb{Q}(\zeta_n) / \mathbb{Q} \right)$, the 7-uple $(a,b,c,d,\sigma(x),\sigma(y),\sigma(z))$ is again a solution to Equation \eqref{eqn:main}. More generally, if $\sigma$ is an automorphism in $\operatorname{Gal}(\overline{\mathbb{Q}}/\mathbb{Q})$, then $(a,b,c,d,\sigma(x),\sigma(y),\sigma(z))$ is a solution to \eqref{eqn:main} whenever $(a,b,c,d,x,y,z)$ is.
This action of the Galois group on the set of solutions is non-trivial on the spaces $\generatedQ{1,\tau}$, even up to homothety. 

Finally, the actions of $G$ and of $\operatorname{Gal}\left( \overline{\mathbb{Q}} / \mathbb{Q} \right)$ on the set of solutions of \eqref{eqn:main} commute, hence we get an action of $G \times \operatorname{Gal}\left( \overline{\mathbb{Q}} / \mathbb{Q} \right)$. We will mostly study solutions of Equation \eqref{eqn:main} up to the action of this group, as made precise in the following definition.

\begin{definition}\label{def:se}
We say that two solutions $(a,b,c,d,x,y,z)$ and $(a',b',c',d',x',y',z')$ of Equation \eqref{eqn:main} are \textit{\se} if they lie in the same orbit under the action of $G \times \operatorname{Gal}\left( \overline{\mathbb{Q}} / \mathbb{Q} \right)$.
\end{definition}

\subsection{Outline of the steps leading to the full classification}
We summarise here how the classification will be broken down:
\begin{itemize}
\item Part 1: Solve Equation \eqref{eqn:main} in the cases when one of $a, b, c, d$ is zero or when two of them coincide.
These solutions correspond to spaces of type \circled{3}+\circled{2} (or of type \circled{4}, which have already been classified in \cite{DvoVenZan}); this is done is Section~\ref{sect:32}.

\item Part 2: Find all solutions with $ab=cd$. These correspond to spaces of type  \circled{2}+\circled{2}+\circled{2} in which the two quadratic equations \cite{DvoVenZan}*{3.11} coming from the angles $(\tau+a, \tau+b)$ and $(\tau+c,\tau+d)$ are proportional.
These spaces are easier to study because, under this assumption, the polynomial $P$ factors. After discarding trivial factors and using the symmetries mentioned above, the equation reduces to a sum of four cosines of rational angles, which was already studied by Conway and Jones. It is important to note, however, that in this case a solution to the equation does not necessarily lead to a valid geometric configuration. We will work out precisely when this happens; this is done in Section~\ref{sect:abcd}.

Taking the results of \cite{DvoVenZan} and of Parts 1 and 2 into account, we can now assume that $a,b,c,d$ are all distinct, different from zero, and satisfy $ab\neq cd$.

\item Part 3: Find all solutions to \eqref{eqn:main}. This is worked out in Section~\ref{sect:222}.
\end{itemize}

Although Parts 1 and 2 above are much easier than Part 3, all three of them involve solving a diophantine-trigonometric equation of the same nature. This is achieved by applying the same strategy, which we break down as follows:
\begin{itemize}
\item Step 1: 
Improve the bound given in \cite{DvoVenZan}*{Theorem 1.1} on the common order of $x, y, z$. This is achieved by systematically studying which subsets of the monomials in the equation $P(a, b, c, d, x, y, z)=0$ can add to zero (in the language of \cite{ConwayJones}, these are called \textit{vanishing subsums}). Our approach is fairly general and is described in Section~\ref{sect:Subsums}, see in particular Proposition~\ref{prop:Subsums}. This already leads to substantial improvements over the bound found in \cite{DvoVenZan}*{§~5}.

\item Step 2: Further limit which prime powers appear in the common order of $x,y,z$. We achieve this by considering a certain subgroup of the full Galois group of $\Q(x,y,z)/\Q$ and applying the results of Parts 1 and 2 to bypass several subcases. The details are given in Section~\ref{sect:GaloisConjugates}. This is a key step in making the computation feasible and leads to a substantial improvement over \cite{DvoVenZan}.  

\item Step 3: Once the bound on the common order is small enough, one can exhaustively check all possibilities for $x,y,z$, leading to finitely many algebraic varieties (most of them empty). For each of them, the set of rational points has to be computed; of course, this is not guaranteed to be possible \textit{a priori}, but by a piece of good luck this step can be done for all the varieties involved in our computations (for example, we find curves of genus 2, 3, 5 whose rational points can be determined either because they admit maps to elliptic curves of rank 0 or because they are amenable to the Chabauty-Coleman approach).

\end{itemize}

Putting together the results of \cite{DvoVenZan} and of this paper, the following table summarises all the non-trivial  exact types that can occur and where they are described.

\begin{center}
\begin{tabular}{|c|p{6.5cm}|c|}
\hline
Exact type                  & Spaces & Notes                  \\
\hline
$\infty\circled{6}$         & $\Q(\sqrt{-3})$.       &  CM and Superrectangular \cite{DvoVenZan}*{§ 4.2}                  \\
\hline
$\infty\circled{4}$                  & $\Q(\sqrt{-1})$.                      & CM and Superrectangular \cite{DvoVenZan}*{§ 4.2} \\ 
\hline
$\infty\circled{2}$                  & Imaginary quadratic fields different from $\Q(\sqrt{-1})$ and $\Q(\sqrt{-3})$                      & CM and Rectangular \cite{DvoVenZan}*{§ 4.2}   \\
\hline
$\circled{4}$                  & All non-CM Superrectangular spaces and the two dodecagonal spaces             &   \cite{DvoVenZan}*{§ 7}                \\
\hline
$2\circled{3}$                  & Four spaces                    &   Theorem~\ref{thm:3plus2}      \\
\hline
$6\circled{2}$                  & Four elliptic families                      & Section~\ref{sect:222}                  
\\
\hline
$5\circled{2}$                  & A rational family                      & Section~\ref{sect:abcd}                  
\\
\hline
$4\circled{2}$                  & A rational family and four elliptic families                      &              Section~\ref{sect:222}     \\ \hline
$3\circled{2}$                  & Three spaces                    &       Section~\ref{sect:abcd}        
\\
\hline
\end{tabular}
\end{center}

Of course, there are also spaces of exact types 0,\circled{2}, 2\circled{2} and \circled{3}, as we now show.

We recall from \cite{DvoVenZan}*{§ 3} that $(a_0\tau+a_1,b_0\tau+b_1)$ is a rational angle in $\generated{1,\tau}$ with $\frac{b_0\tau+b_1}{a_0\tau+a_1}\mu\in\R$ for a root of unity $\mu\neq \pm 1$ precisely when 
\begin{equation}\label{eqn:1angle:prima}
  a_0 b_0 \tau\overline{\tau}+a_0 b_1 \frac{\mu^2\tau-\overline{\tau}}{\mu^2-1} + a_1 b_0 \frac{\mu^2\overline{\tau}-\tau}{\mu^2-1} + a_1 b_1 =0.
\end{equation}
For fixed $a_0,b_0,a_1,b_1,\mu$ the set of $\tau$ which satisfy this equation is a conic in the plane (it is easily checked that if $(a_0,a_1)$ and $(b_0,b_1)$ are not proportional the equation is not identically satisfied in $\tau$).

As $\C$ is not the union of countably many conics, there are infinitely many $\tau$'s such that $\generated{1,\tau}$ does not have any rational angles at all. In fact, the set of $\tau$ which generate a space with at least one rational angle has Lebesgue measure zero.

Similarily, if we fix $\theta\in\roots$ and set $\tau=r\theta$, we see that for angles different from $(1,\tau)$ Equation~\eqref{eqn:1angle:prima} has at most 2 solutions in $r$. Therefore, for any given amplitude, there are uncountably many spaces with exactly one rational angle of that amplitude and no other rational angle.

To show that there exist spaces of exact type 2\circled{2}, we begin by fixing 
a root of unity $\theta$ with large order and consider the lattices generated by complex numbers $\tau$ of the form $\tau=r\theta$ with $r\in\R$.
According to \cite{DvoVenZan}*{§ 3.3}, the angle $(\tau+a,\tau+b)$ is a rational angle with squared amplitude $y$ if and only if
 \begin{align}\label{eqn:SecondoAngolo}
  \tau^2 (y-1)+\tau\left[a(y-x)+b(x y -1)\right]+a b x (y-1)&=0;
\end{align}
it is also useful to consider the corresponding equation for $r$
 \begin{align}\label{eqn:SecondoAngolo.r}
  r^2 +r\frac{a(y-x)+b(x y -1)}{\theta (y-1)}+a b &=0.
\end{align}

We can select $y\in \roots$ and $a,b\in\Q$ so that the degree of $y$ over $\Q(x)$ is large and the left-hand side of \eqref{eqn:SecondoAngolo} is irreducible (as a polynomial in $r$) over $\Q(x,y)$. This implies that the degree of $\tau$ over $\Q(x)$ is large, and therefore, again by \eqref{eqn:SecondoAngolo}, the space $\generated{1,\tau}$ cannot contain a right angle.
It was proved in \cite{DvoVenZan} that any space with three independent angles with common order large enough must contain a right angle, and this shows that $\generated{1,\tau}$ cannot contain a third rational angle.
The same argument applies, using \cite{DvoVenZan}*{§ 3.2}, for spaces of exact type \circled{3}.

\section{General strategy for subsums}\label{sect:Subsums}
Building upon the strategy in \cite{DvoVenZan} we explain here how to bound the common order of $x,y,z$ occurring in a solution of \eqref{eqn:main}.
Consider an equation of the form
\begin{equation}\label{eq:SubsumsStrategy}
f(x,y,z) := \sum_{\underline{i}=(i_1,i_2,i_3) \in \{0,1,2\}^3} p_{\underline{i}} x^{i_1} y^{i_2} z^{i_3} = 0
\end{equation}
where each $p_{\underline{i}}$ is a rational number (in our applications, $p_{\underline{i}}$ will depend polynomially on the rational variables $a,b,c,d$, but this is not important for the present discussion).
Let $(x,y,z) \in U^3$ be a solution of Equation \eqref{eq:SubsumsStrategy}. We fix a decomposition of $\{0,1,2\}^3$ as a disjoint union $\coprod_{j=1}^r I_j$ with the following properties:
\begin{enumerate}
\item for each $j=1,\ldots,r$ we have $\sum_{\underline{i} = (i_1,i_2,i_3) \in I_j} p_{\underline{i}} x^{i_1} y^{i_2} z^{i_3} = 0$;
\item if $I$ is a proper nonempty subset of $I_j$ some $j=1,\ldots,r$ we have $\sum_{\underline{i} = (i_1,i_2,i_3) \in I} p_{\underline{i}} x^{i_1} y^{i_2} z^{i_3} \neq 0$.
\end{enumerate}
In other words, we fixed a maximally refined decomposition of \eqref{eq:SubsumsStrategy} into vanishing subsums. 

\begin{remark}
This decomposition is not necessarily unique, as shown for example by $1+1+(-1)+(-1)$, where we can group together the first and third summands (resp.~second and fourth), or the first and the fourth (resp.~second and third).
\end{remark}

For each $j=1,\ldots,r$ we also fix an arbitrary element $\underline{i}_j$ of $I_j$ and define
\begin{equation}\label{eq:submoduleMj}
M_j := \langle \underline{i}-\underline{i}_j \bigm\vert \underline{i} \in I_j \rangle_{\mathbb{Z}} \subseteq \mathbb{Z}^3,
\end{equation}
where we interpret elements of $\{0,1,2\}^3$ as elements of the free $\mathbb{Z}$-module $\mathbb{Z}^3$. Notice that the free $\mathbb{Z}$-module $M_j$ is independent of the choice of $\underline{i}_j \in I_j$. Finally we set $M := M_1 + \cdots + M_r$.
\begin{proposition}\label{prop:Subsums}
The following hold:
\begin{enumerate}
\item If the quotient $\mathbb{Z}^3/M$ is infinite, there exist $\alpha_1, \alpha_2, \alpha_3 \in \mathbb{Z}$, not all equal to zero, such that for all $t \in \mathbb{C}^\times$ the triple $(x t^{\alpha_1}, y t^{\alpha_2}, z t^{\alpha_3})$ is a solution to Equation \eqref{eq:SubsumsStrategy}.
\item If the quotient $\mathbb{Z}^3/M$ is finite (so that it is a finite abelian group), let $e$ denote its exponent and let $\ell := \max \{ |I_j| : j =1,\ldots, r \}$ be the length of the longest vanishing subsum. The common order of $x, y, z$ divides $e n$, where $n=\prod_{\substack{p \text{ prime} \\ p \leq \ell} } p$.
\end{enumerate}
\end{proposition}
\begin{proof}
We will find it useful to employ the following notation:
\begin{enumerate}
\item we consider elements of $\mathbb{Z}^3$ as vectors with three coordinates, and a vector $\underline{\alpha}$ of $\mathbb{Z}^3$ will be understood to have coordinates $(\alpha_1, \alpha_2, \alpha_3)$;
\item given two vectors $\underline{\alpha}$ and $\underline{\beta}$ of $\mathbb{Z}^3$ we denote by $\underline{\alpha} \cdot \underline{\beta} = \alpha_1\beta_1 + \alpha_2\beta_2+\alpha_3\beta_3$ their scalar product;
\item given nonzero complex numbers $z_1, z_2, z_3$ and $\underline{\alpha} \in \mathbb{Z}^3$, we denote by $(z_1,z_2,z_3)^{\underline{\alpha}}$ the complex number $z_1^{\alpha_1} z_2^{\alpha_2} z_3^{\alpha_3}$.
\end{enumerate}

Suppose first that $\mathbb{Z}^3/M$ is infinite. In this case, the vector subspace $M \otimes_{\mathbb{Z}} \mathbb{Q}$ of $\mathbb{Q}^3$ has dimension strictly less than 3. In particular, the orthogonal subspace $(M \otimes \mathbb{Q})^\bot$ is nontrivial, and we conclude that there exists a non-zero $\underline{\alpha} \in \mathbb{Z}^3$ that is orthogonal to all the elements of $M$, hence in particular to all the elements of every $M_j$, for $j=1,\ldots,r$.
We can now write
\[
\begin{aligned}
f(xt^{\alpha_1}, yt^{\alpha_2}, zt^{\alpha_3}) 
& = \sum_{j=1}^r \sum_{\underline{i} \in I_j} p_{\underline{i}} \cdot (xt^{\alpha_1}, yt^{\alpha_2}, zt^{\alpha_3})^{\underline{i}} \\
& = \sum_{j=1}^r t^{\underline{\alpha} \cdot \underline{i}_j} \sum_{\underline{i} \in I_j} p_{\underline{i}} \cdot (x, y, z)^{\underline{i}} t^{\underline{\alpha} \cdot (\underline{i}-\underline{i}_j)}.
\end{aligned}
\]
By construction of $\underline{\alpha}$, every scalar product $\underline{\alpha} \cdot (\underline{i}-\underline{i}_j)$ vanishes (notice that $\underline{i}-\underline{i}_j$ lies in $M_j$), so the previous expression evaluates to
\[
\begin{aligned}
f(xt^{\alpha_1}, yt^{\alpha_2}, zt^{\alpha_3}) 
& = \sum_{j=1}^r t^{\underline{\alpha} \cdot \underline{i}_j}  \sum_{\underline{i} \in I_j} p_{\underline{i}} \cdot (x, y, z)^{\underline{i}} = 0,
\end{aligned}
\]
where we have used the fact that (again by construction)  $\sum_{\underline{i} \in I_j} p_{\underline{i}} \cdot (x, y, z)^{\underline{i}}$ vanishes for every $j$.

Suppose on the other hand that $\mathbb{Z}^3/M$ is finite. Multiplying the relation $\sum_{\underline{i} \in I_j} p_{\underline{i}} \cdot (x, y, z)^{\underline{i}}=0$ by the invertible element $(x,y,z)^{-\underline{i}_j}$ we obtain, for every $j$, the equality
\[
\sum_{\underline{i} \in I_j} p_{\underline{i}} \cdot (x, y, z)^{\underline{i}-\underline{i}_j} = 0.
\]
This is a vanishing sum of rational multiples of roots of unity, satisfying the following two properties:
\begin{enumerate}
\item No proper subsum vanishes: this follows from property (2) in the choice of the decomposition $\{0,1,2\}^3 = \coprod I_j$;
\item it is \textit{normalised}, in the sense that the root of unity 1 does appear in the sum (in the summand corresponding to $\underline{i}=\underline{i}_j$).
\end{enumerate}
By \cite{ConwayJones}*{Theorem 5}, we obtain that the common order of the roots of unity $(x,y,z)^{\underline{i}-\underline{i}_j}$ is a squarefree number $n_j$ such that $\sum_{p \mid n_j} (p-2) \leq |I_j|-2 \leq \ell-2$. Notice in particular that $n_j$ divides the $n$ in the statement.

By definition of exponent of an abelian group we know that the vector $(e,0,0)$ belongs to $M$, that is, we may write $(e,0,0)$ as a linear combination (with integer coefficients) of elements of the form $\underline{i}-\underline{i_j}$ for $\underline{i} \in I_j$. Raising $(x,y,z)$ to the power $(e,0,0)$, we obtain that $x^e$ may be written as a product of terms of the form $(x,y,z)^{\underline{i}-\underline{i}_j}$, each of which -- by what we already showed -- is a root of unity of order dividing $n$. Clearly, this implies that $x$ is a root of unity of order dividing $en$. The same argument, applied to the vectors $(0,e,0)$ e $(0,0,e)$, shows that the orders of $y$ and $z$ also divide $en$, which concludes the proof.
\end{proof}

\begin{definition}
We will say that a solution $(a,b,c,d,x,y,z)$ to Equation \eqref{eqn:main} \textit{\bif} if it satisfies the conclusion of case (1) of the previous proposition.
\end{definition}

For such solutions there is no absolute bound on the degree of the field $\Q(x,y,z)$. These solutions are classified in \cite{DvoVenZan}*{§ 6}, where it is proved that they can only arise from \emph{rectangular} and \emph{superrectangolar} spaces (cf.~the definitions in \cite{DvoVenZan}*{§ 7.1}).
Given the importance of this class of spaces, which enjoy additional symmetries, they will be reviewed and classified in Section~\ref{sec:Rettangolari}.

\section{Rectangular spaces and CM spaces}\label{sec:Rettangolari}
We recall here some classes of spaces which play a special role in the classification.
\subsection{Rectangular spaces}
We call \emph{rectangular} any space containing a rational angle of $\pi/2$.
Up to homothety, these can be characterized as those generated by $1,\tau$ with a purely imaginary $\tau$, or as those generated by $1,\tau$ with $\abs{\tau}=1$ (\cite{DvoVenZan}*{Lemma 4.1}).

Solutions to \eqref{eqn:main} belonging to unbounded families come from rectangular spaces, and indeed the presence of a right angle provides a geometric explanation (by reflection) for the existence of a third rational angle when a second one is given.

\def\Immtau{1.5}
\def\Imma{2.3}
\begin{center}
\begin{tikzpicture}[scale=1.5]
\coordinate (A) at (-2.5,0);
\coordinate (B) at (2.5,0);
\coordinate (C) at (0,-0.5);
\coordinate (D) at (0,2.5);
\coordinate (uno) at (1,0);
\node[anchor=90]at(uno){$1$};
\coordinate (tau) at (0,\Immtau);
\node[anchor=0]at(tau){$\tau$};
\coordinate (tau1) at ($(tau)+(1,0)$);
\node[anchor=270]at(tau1){$\tau+1$};
\coordinate (taua) at ($(tau)+(\Imma,0)$);
\node[anchor=270]at(taua){$\tau+a$};
\coordinate (tau-1) at ($(tau)-(1,0)$);
\node[anchor=270]at(tau-1){$\tau-1$};
\coordinate (tau-a) at ($(tau)-(\Imma,0)$);
\node[anchor=270]at(tau-a){$\tau-a$};
\coordinate (O) at (0,0);
\node[anchor=45]at(O){$O$};
\draw [->,thin, name path=ab] (A)--(B);
\draw [->,thin, name path=cd] (C)--(D);
\draw [->,thick,red, name path=ac] (O)--(tau);
\draw [->,thick,red, name path=ac] (O)--(uno);
\draw [->,thick,blue, name path=ac] (O)--(tau1);
\draw [->,thick,blue, name path=ac] (O)--(taua);
\draw [->,thick,green, name path=ac] (O)--(tau-1);
\draw [->,thick,green, name path=ac] (O)--(tau-a);
\end{tikzpicture}
\end{center}

Up to homothety, we can describe the configuration as follows: fix a rational number $a\neq 0,\pm 1$ and a root of unity $y\neq 1$; if the equation 
\[
 \tau^2+(a-1)\frac{y+1}{y-1}\tau-a=0
\]
has a purely imaginary solution $\tau$, then $1,\tau$ generate a rectangular space having a rational angle $(\tau+1,\tau+a)$ with squared amplitude $y$.

\subsection{Superrectangular spaces}

We call a rectangular space whose right angle belongs to a rational triple \emph{superrectangular}. It is easy to see that superrectangular spaces are homothetic to a space of the form $\generatedQ{1,\zeta}$ for some $\zeta\in\roots$, and that their rational triple always extends to a rational quadruple.

\def\Immtau{1.5}
\begin{center}
\begin{tikzpicture}[scale=1.5]
\coordinate (A) at (-2,0);
\coordinate (B) at (2,0);
\coordinate (C) at (0,-0.5);
\coordinate (D) at (0,2.5);
\coordinate (uno) at (1,0);
\node[anchor=90]at(uno){$1$};
\coordinate (tau) at (0,\Immtau);
\node[anchor=0]at(tau){$\tau$};
\coordinate (tau1) at ($(tau)+(1,0)$);
\node[anchor=270]at(tau1){$\tau+1$};
\coordinate (tau-1) at ($(tau)-(1,0)$);
\node[anchor=270]at(tau-1){$\tau-1$};
\coordinate (O) at (0,0);
\node[anchor=45]at(O){$O$};
\draw [->,thin, name path=ab] (A)--(B);
\draw [->,thin, name path=cd] (C)--(D);
\draw [->,thick,blue, name path=ac] (O)--(tau);
\draw [->,thick,blue, name path=ac] (O)--(uno);
\draw [->,thick,blue, name path=ac] (O)--(tau1);
\draw [->,thick,blue, name path=ac] (O)--(tau-1);
\end{tikzpicture}
\end{center}

Up to homothety, we can describe the configuration as follows: fix a root of unity $y\neq 1$ and set
\begin{align*}
\tau&=\frac{y+1}{y-1};
\end{align*}
then $1,\tau$ generate a superrectangular space with a rational quadruple $(1,\tau,\tau+1,\tau-1)$ such that the rational angle $(1,\tau+1)$ has squared amplitude equal to $y$.

It was proved in \cite{DvoVenZan} that every superrectangular space has exact type \circled{4} with only two exceptions (up to homothety): the fields $\Q(\sqrt{-1})$ and $\Q(\sqrt{-3})$. These two cases belong to another class of special rectangular spaces with additional structure.

\subsection{CM spaces}
 Every imaginary quadratic field is clearly a space in our sense; we call \textit{CM space} any space homothetic to an imaginary quadratic field. It is easy to see that distinct imaginary quadratic fields are not homothetic, and that each of them contains infinitely many rational angles. Furthermore, any space containing infinitely many rational angles must be a CM space \cite{DvoVenZan}*{Theorem 4.3}.
 
 One can compute the exact type of CM spaces and find that:
 \begin{itemize}
  \item $\Q(\sqrt{-3})$ contains infinitely many rational sextuples;
  \item $\Q(\sqrt{-1})$ contains infinitely many rational quadruples;
  \item any other CM space contains infinitely many rational right angles.
 \end{itemize}
See \cite{DvoVenZan}*{Theorem 4.2} for details.

We state here a simple lemma that cuts down the amount of trivial cases we will have to consider in our subsequent computations:
\begin{lemma}\label{lemma:2RightAnglesImplyCM}
 Any space which contains two distinct right angles is a CM space.
\end{lemma}
\begin{proof}
 Two distinct right angles cannot be adjacent (recall that, with our conventions, $(v_1,v_2)$ and $(v_1,-v_2)$ define the same angle). Therefore we can fix one of the angles as $(1,\tau)$, the second one as $(\tau+b_0,\tau+b_1)$, and \cite{DvoVenZan}*{Equation (3.11)} holds with $x=y=-1$.
 We thus have
 \[
  \tau^2=b_0 b_1,
 \]
 which is a (negative) rational number, and therefore $\generatedQ{1,\tau}$ is a CM space.
\end{proof}

\section{The Galois conjugates argument}\label{sect:GaloisConjugates}
We discuss here an argument involving cyclotomic extensions that is used to speed up the computation.

Let $(a,b,c,d,x,y,z)$ be a solution to Equation \eqref{eqn:main}, and let $n=n(x,y,z)$ be the common order of the roots of unity $x, y, z$.
Let $q$ be an integer dividing $n$, write $N:=n/q$, and assume $(N, q)=1$ (in this case we will say that $q$ divides $n$ \textit{exactly}).
We can write
\[
x = \xi_1\xp, \quad y = \xi_2 \yp, \quad z = \xi_3 \zp,
\]
where $\xi_i = \zeta_{q}^{e_i}$ (with each exponent $e_1, e_2, e_3$ lying in the interval $[0,q-1]$) and $\xp, \yp, \zp \in U$ have order prime to $q$. Notice in particular that there is only a finite number of possibilities for the triple $(\xi_1,\xi_2,\xi_3)$.

The extensions $\mathbb{Q}(\zeta_{q})/\mathbb{Q}$ and $\mathbb{Q}(\zeta_N)/\mathbb{Q}$ are linearly disjoint, hence we have a canonical identification $\Gal( \mathbb{Q}(\zeta_n) / \mathbb{Q}) \cong \Gal( \mathbb{Q}(\zeta_N) / \mathbb{Q}) \times \Gal( \mathbb{Q}(\zeta_{q}) / \mathbb{Q})$. We will use this identification to consider $\Gal( \mathbb{Q}(\zeta_{q}) / \mathbb{Q})$ as the subgroup of $\Gal( \mathbb{Q}(\zeta_n) / \mathbb{Q})$ fixing $\mathbb{Q}(\zeta_N)$. Observe now that by assumption we have
\[
P(a,b,c,d,\xi_1 \xp, \xi_2 \yp, \xi_3 \zp)=P(a,b,c,d,x,y,z)=0,
\]
and that the coefficients of $P$, the quantities $a,b,c,d$, and the roots of unity $\xp, \yp, \zp$ are all contained in $\mathbb{Q}(\zeta_N)$. As a consequence, for every $\sigma \in \Gal(\mathbb{Q}(\zeta_{q}) / \mathbb{Q})$ we have
\[
P(a,b,c,d, \sigma(\xi_1) \xp, \sigma(\xi_2) \yp, \sigma(\xi_3) \zp) = 0.
\]
For each fixed choice of $\xi_1=\zeta_{q}^{e_1}, \xi_2=\zeta_{q}^{e_2}, \xi_3=\zeta_{q}^{e_3}$ such that the least common multiple of the orders of $\xi_1, \xi_2, \xi_3$ is $q$ (equivalently, the greatest common divisor $(e_1,e_2,e_3,q)$ is 1), we let
$V_{\xi_1,\xi_2,\xi_3}$ be the variety over $\mathbb{Q}(\zeta_{q})$ defined in $\mathbb{A}^7$ by
\begin{equation}\label{eq: equations defining the Galois conjugates variety}
P(a,b,c,d, \sigma(\xi_1) \xp, \sigma(\xi_2) \yp, \sigma(\xi_3) \zp) = 0 \quad \forall \sigma \in \Gal(\mathbb{Q}(\zeta_{q})/\mathbb{Q});
\end{equation}
notice that here the 7 variables are $a,b,c,d,\xp,\yp,\zp$. By construction, $V_{\xi_1,\xi_2,\xi_3}$ is defined over $\mathbb{Q}(\zeta_{q})$. The ideal defining it is stable under $\Gal(\mathbb{Q}(\zeta_{q})/\mathbb{Q})$, so it defines in a natural way a corresponding variety over $\mathbb{Q}$, which by a slight abuse of notation we still denote by $V_{\xi_1,\xi_2,\xi_3}$ (or by $V_{\xi_1,\xi_2,\xi_3, \mathbb{Q}}$ when we want to stress the field of definition). Finally, we let $W_{\xi_1,\xi_2,\xi_3}$ be the projection of $V_{\xi_1,\xi_2,\xi_3}$ to $\mathbb{A}_{\mathbb{Q}}^4(a,b,c,d)$.

It is often convenient in our computations to consider the varieties $V_{\xi_1,\xi_2,\xi_3}$ and $W_{\xi_1,\xi_2,\xi_3}$. If $W_{\xi_1,\xi_2,\xi_3}$ does not have any rational points, or more generally if non of its rational points satisfies the non-degeneracy conditions \eqref{eq:Conditions}, then we conclude that for no solution to \eqref{eqn:main} the common order of $x, y, z$ is exactly divisible by $q$.
This technique allows us to further bound the possible common orders in a computationally efficient way. Notice that Equation \eqref{eq: equations defining the Galois conjugates variety} shows that $V_{\xi_1, \xi_2, \xi_3}$ is defined by $\varphi(q)$ equations, so, as $q$ grows, it becomes increasingly unlikely for $V_{\xi_1, \xi_2, \xi_3}$ to have (non-trivial) rational points. This heuristic expectation is fully confirmed by our computations.

\section{Spaces with a triple of adjacent rational angles}\label{sect:32}
In this section we classify the spaces of type \circled{3}+\circled{2}. This is a much simpler task than the classification of spaces of type 3\circled{2}, but of the same nature. Therefore, this section will also provide a blueprint for how the main part of the proof will be organised and what computations we need to perform.

Spaces of type \circled{3}+\circled{2} are described by the following equation, computed in \cite{DvoVenZan}*{§ 3.4}:

\begin{equation}\label{eq:3plus2}
\begin{aligned}
P_{\circled{\tiny{3}} +\circled{\tiny{2}}}(a,b,c,x,y,z)&=(2a^2 - a(b + c) + 2bc)xy  - abx^2y - acy\\
& -(2a^2 - a(b + c) + 2bc) xyz +abyz+acx^2yz \\
& + b(a - c)x^2 + c(a - b)y^2  - a(a - b)xy^2 - a(a - c)x \\
&  -b(a-c)y^2z - c(a-b)x^2z + a(a-b)xz + a(a-c)xy^2z =0.
\end{aligned}
\end{equation}

We seek solutions with $x,y,z\in\roots$ and $(a:b:c)\in\pro_2(\Q)$ satisfying the following non-degeneracy conditions:
\begin{equation}\label{eq:Conditions3plus2}
 x,y,z\neq 1,\qquad x\neq y,\qquad a,b,c\text{ distinct and non-zero.}
\end{equation} 
Given such a solution, setting
\begin{equation*}
 \tau=ax\frac{y-1}{x-y}
\end{equation*}
and $V=\generatedQ{1,\tau}$ gives a space with a rational triple $(1,\tau,\tau+a)$ and a rational angle $(\tau+b,\tau+c)$.

\begin{remark}
\leavevmode
    \begin{enumerate}
        \item In principle, there would be no need of treating this case separately, as it can be recovered from the general equation \eqref{eqn:main} when $|\{0,a,b,c,d\}|<5$.
        Specifically, Equation~\eqref{eq:3plus2} can be obtained from   \eqref{eqn:main} by setting $b=0$ and renaming the variables through the substitution $(a,c,d,x,y,z)\mapsto (a,b,c,x,x/y,z)$.
        By working with Equation~\eqref{eq:3plus2}, however, we are able to avoid a great amount of spurious degenerate solutions without geometric meaning, leading to cleaner and faster computations.
        \item Not every symmetry taken into account by Definition~\ref{def:se} is also valid for Equation \eqref{eq:3plus2}; in particular, the roles of the three variables $x, y, z$ are now genuinely different, so that we cannot freely permute them. We still have that, for every solution $(a,b,c,x,y,z)$ to \eqref{eq:3plus2}, the following are also solutions:
        \begin{align}\label{simmetrie:3plus2}
            \begin{split}
                &\left(a,c,b,x,y,z^{-1}\right)\\
                &\left(-a,-a+b,-a+c,y,x,z\right)\\
                &\left(\frac{1}{a},\frac{1}{b},\frac{1}{c},x^{-1},yx^{-1},z\right)
            \end{split}
        \end{align}
        and they lead to values of $\tau$ generating spaces homothetic to the original one.
        As far as the Galois action is concerned, on the other hand, we still have that for every solution $(a,b,c,x,y,z)$ to \eqref{eq:3plus2} and every $\sigma \in \operatorname{Gal}(\overline{\mathbb{Q}}/\mathbb{Q})$, the 6-uple $(a,b,c,\sigma(x),\sigma(y),\sigma(z))$ is again a solution to \eqref{eq:3plus2}.
        \item We observe that the space $V=\langle 1, \tau \rangle_{\mathbb{Q}}$ only depends on $x$ and $y$ (and not on $a,b,c$). As a consequence, we know \textit{a priori} that---outside of the unbounded families already mentioned---there is only a \textit{finite} list of spaces of type \circled{3} + \circled{2}, because the orders of $x$ and $y$ are bounded by a universal constant.
    \end{enumerate}
\end{remark}

We employ the strategy outlined in Section~\ref{sect:Subsums}, fixing a maximal decomposition of the 14 summands in \eqref{eq:3plus2} into vanishing subsums. 
\begin{proposition}\label{prop:LongSubsum3plus2}
With notation as in Proposition~\ref{prop:Subsums},
\begin{enumerate}
\item if $\ell \geq 9$ then $M=\mathbb{Z}^3$, $e=1$, and the common order $n$ of $x, y, z$ is squarefree and satisfies $\sum_{p \mid n} (p-2) \leq 14-2$;
\item in all cases, if $\mathbb{Z}^3/M$ is finite, then its exponent $e$ satisfies $e \leq 10$.
\end{enumerate}
\end{proposition}
\begin{proof}
The proof is computational. For part 1 we proceed as follows: we list all 9-elements subsets of the terms appearing in \eqref{eq:3plus2}, we think of such a subset as being the $I_j$ realising the maximal length, and we check that the corresponding module $M_j \subseteq \mathbb{Z}^3$ is all of $\mathbb{Z}^3$. Notice that it suffices to perform this operation on subsets with exactly nine elements: if $I_j, I_j'$ are two subsets with $I_j \subseteq I_j'$, the corresponding modules $M_j$ and $M_j'$ satisfy $M_j \subseteq M_j'$, hence the fact that the equality $M_j=\mathbb{Z}^3$ holds for all subsets $I_j$ with $|I_j|=9$ implies that the same statement holds for all $I_j$ with $|I_j| \geq 9$.
Finally, given a solution of \eqref{eq:3plus2}, the statement about $n$ follows upon applying \cite{ConwayJones}*{Theorem 5} to the vanishing sum indexed by the subset of maximal cardinality in the partition $\{I_j\}$.

For part 2, let $V$ be the set consisting of the 14 vectors giving the exponents of $(x,y,z)$ in the 14 summands of \eqref{eq:3plus2}, and let $W = \{v_1-v_2 : v_1, v_2 \in V\}$. By construction, given a solution of \eqref{eq:3plus2} and a maximal decomposition into vanishing subsums $\{I_j\}$, the module $M$ is generated by vectors in $W$.
We loop over all 3-elements subsets $\{w_1,w_2,w_3\}$ of $W$, and for each we check whether $\mathbb{Z}^3 / (\mathbb{Z}w_1 + \mathbb{Z}w_2 + \mathbb{Z} w_3)$ is finite. If it is, we compute its exponent, and find that it always bounded above by $10$.

Notice that, when $M$ has finite index, we can always find three vectors of the form $\underline{i}-\underline{i}_j$, with $\underline{i} \in I_j$, that generate a finite-index subgroup of $M$: this follows from basic linear algebra upon tensoring $M$ with $\mathbb{Q}$. In particular, given a solution to \eqref{eq:3plus2} that does not belong to an unbounded family, the corresponding module $M$ contains one of the modules with three generators that we considered above. This shows that the exponent of $\mathbb{Z}^3/M$ is bounded above by 10 in all cases when $\mathbb{Z}^3/M$ is finite.
\end{proof}

\begin{corollary}\label{cor:Bound32Step1}
Let $(a,b,c,x,y,z)$ be a solution of \eqref{eq:3plus2} not belonging to an unbounded family. The common order $n$ of $x,y,z$ divides $2^4 \cdot 3^3 \cdot 5^2 \cdot 7^2 \cdot 11 \cdot 13$ and is not divisible by $2^3 \cdot 3^2$.
\end{corollary}
\begin{proof}
Let $(a,b,c,x,y,z)$ be a solution to \eqref{eq:3plus2}. We continue with the notation of Proposition~\ref{prop:Subsums}. One of the following holds:
\begin{enumerate}
\item $\mathbb{Z}^3/M$ is infinite: then $(a,b,c,x,y,z)$ belongs to an unbounded family, which is not the case by assumption;
\item $\ell \geq 9$: then by Proposition~\ref{prop:LongSubsum3plus2} (1) we have $n \mid \prod_{p \leq 13} p$;
\item $\ell < 9$ and $\mathbb{Z}^3/M$ is finite: by Proposition~\ref{prop:Subsums} we obtain $n \mid e \prod_{p \leq \ell}p \mid e \prod_{p \leq 7} p$, where (by part (2) of Proposition~\ref{prop:LongSubsum3plus2}) $e$ is an integer not exceeding $10$. For any such integer $e$ the conclusion in the statement holds.
\end{enumerate}
\end{proof}

Consider now a solution $(a,b,c,x,y,z)$ to Equation \eqref{eq:3plus2}, where $x,y,z$ have common order $n$. As in Section~\ref{sect:GaloisConjugates}, we let $q$ be an integer such that $q \mid n$ and $(n,n/q)=1$. We also write $x=\xi_1 \xp, y=\xi_2 \yp, z=\xi_3 \zp$ with $\xi_i = \zeta_q^{e_i}$ and $e_i \in [0,q-1]$. Notice that since $q$ divides the common order of $x,y,z$ we have $(q,e_1,e_2,e_3)=1$.

We see then that $(a,b,c,x,y,z)$ gives rise to a $\overline{\mathbb{Q}}$-point on the variety $V_{\xi_1,\xi_2,\xi_3}$ defined in $\mathbb{A}^6$ by
\[
P_{\circled{\tiny{3}} + \circled{\tiny{2}}}(a,b,c, \sigma(\xi_1)\xp,\sigma(\xi_2) \yp, \sigma(\xi_3) \zp )=0 \quad \forall \sigma \in \Gal(\mathbb{Q}(\zeta_q)/\mathbb{Q}).
\]
We denote by $V_{\xi_1,\xi_2,\xi_3,\mathbb{Q}}$ the $\mathbb{Q}$-variety obtained from $V_{\xi_1,\xi_2,\xi_3}$ (see Section~\ref{sect:GaloisConjugates}). The solution $(a,b,c,x,y,z)$ also gives rise to a $\overline{\mathbb{Q}}$-point of $V_{\xi_1,\xi_2,\xi_3,\mathbb{Q}}$.
The following is proved by a direct computation.
\begin{proposition}\label{prop:RuleOutDivisors3Plus2}
Let $q \in \{3^3, 5^2, 7^2, 7, 11, 13, 3^2 \cdot 5, 3 \cdot 5\}$. Let $e_1, e_2, e_3 \in [0,q-1]$ be integers such that $(e_1,e_2,e_3,q)=1$ and let $\xi_i=\zeta_q^{e_i}$. Define a set $S \subseteq \{x'-1,y'-1,z'-1,x'-y'\}$ as follows:
\begin{enumerate}
\item $x'-1 \in S$ if and only if $e_1=0$ (respectively, $y'-1 \in S$ if and only if $e_2=0$, and $z'-1 \in S$ if and only if $e_3=0$);
\item $x'-y' \in S$ if and only if $e_1=e_2$.
\end{enumerate}
 
The intersection of $V_{\xi_1,\xi_2,\xi_3,\mathbb{Q}}$ with $\{ abc(a-b)(b-c)(c-a) \prod_{s \in S} s \neq 0 \}$ is empty.
\end{proposition}

\begin{corollary}\label{cor: sharp bound case 2+3}
With the same hypotheses as in Corollary~\ref{cor:Bound32Step1}, the common order $n$ of $x, y, z$ divides $2^4 \cdot 3^2 \cdot 5$ and is not divisible by $3 \cdot 5$.
\end{corollary}
\begin{proof}
By Corollary~\ref{cor:Bound32Step1} we already know that $n \mid 2^4 \cdot 3^3 \cdot 5^2 \cdot 7^2 \cdot 11 \cdot 13$ and that $2^3 \cdot 3^2 \nmid n$. Suppose now that $13 \mid n$: the previous divisibility implies in particular $(13,n/13)=1$, and therefore $(a,b,c,x,y,z)$ gives rise to a $\overline{\Q}$-point on $V_{\xi_1,\xi_2,\xi_3,\mathbb{Q}}$ for suitable 13-th roots of unity $\xi_1,\xi_2,\xi_3$. Proposition~\ref{prop:RuleOutDivisors3Plus2} shows that one of the following holds:
\begin{enumerate}
\item $abc(a-b)(b-c)(c-a)=0$, which contradicts the non-degeneracy conditions \eqref{eq:Conditions3plus2};
\item $e_1=0$ and $x'=1$, which implies $x=\zeta_q^{e_1} x'=1$, again contradicting \eqref{eq:Conditions3plus2}. An identical argument applies in the cases $y'=1$ and $z'=1$.
\item $e_1=e_2$ and $x'=y'$, which implies $x=\xi_1 x' = \xi_2 y' = y$, once again contradicting \eqref{eq:Conditions3plus2}.
\end{enumerate}
In all cases we obtain a contradiction, which shows that $13 \nmid n$. A similar argument also shows that $7^2 \nmid n, 11 \nmid n, 5^2 \nmid n$, and $3^3 \nmid n$. Furthermore, once we know that $3^3 \nmid n$, we can repeat the same argument again with $3^2 \cdot 5$ (since at this point we know that $3^2 \cdot 5 \mid n \Rightarrow (3^2 \cdot 5,n/(3^2 \cdot 5))=1$), and obtain that $45  \nmid n$. In a similar way we prove that $7$ and $15$ do not divide $n$, which concludes the proof.
\end{proof}

The bound on the common order of $(x,y,z)$ is now small enough that we can easily enumerate all solutions of Equation \eqref{eq:3plus2} outside of the unbounded families (recall that these correspond to rectangular spaces), and for each of them describe the rational angles contained in the corresponding space:%
\begin{theorem}\label{thm:3plus2}
Let $V$ be a space of type \circled{3}+\circled{2}. Then one of the following holds:
\begin{enumerate}
\item $V$ is a rectangular space;
\item $V$ is homothetic to $\langle 1, \zeta_8^2+\zeta_8+1 \rangle_{\mathbb{Q}}$; it is of exact type 2\circled{3}.
Setting $\tau=\zeta_8^2+\zeta_8+1$, the two rational triples are $(1,\tau,\tau-1)$ and $(\tau+1,\tau-1/2, \tau - 2)$; we have
\begin{align*}
 \arg(\tau)&=\frac{\pi}{4} , & \arg(\tau-1)&=\frac{3}{8}\pi, &
 \arg\left(\frac{\tau-1/2}{\tau+1}\right)&=\frac{\pi}{8} , & \arg\left(\frac{\tau-2}{\tau+1}\right)&=\frac{3}{8}\pi. 
\end{align*}

\item $V$ is homothetic to $  \langle \zeta_5^2 + 2\zeta_5+2,1 \rangle_{\mathbb{Q}}$; it is of exact type 2\circled{3}.
Setting $\tau=\zeta_5^2+2\zeta_5+2$, the two rational triples are $(1,\tau,\tau-1)$ and $(\tau-1/3,\tau-4,\tau-5/4)$; we have
\begin{align*}
 \arg(\tau)&=\frac{3}{10}\pi , & \arg(\tau-1)&=\frac{2}{5}\pi, &
 \arg\left(\frac{\tau-4}{\tau-1/3}\right)&=\frac{2}{5}\pi , & \arg\left(\frac{\tau-5/4}{\tau-1/3}\right)&=\frac{\pi}{10}. 
\end{align*}

\item $V $ is homothetic to $ \langle 2\zeta_5^3 +\zeta_5 + 2, 1 \rangle_{\mathbb{Q}}$;  it is of exact type 2\circled{3}.
Setting $\tau=2\zeta_5^3+\zeta_5+2$, the two rational triples are $(1,\tau,\tau-1)$ and $(\tau-1/3,\tau-4,\tau-5/4)$; we have
\begin{align*}
 \arg(\tau)&=-\frac{\pi}{10} , & \arg(\tau-1)&=-\frac{4}{5}\pi, &
 \arg\left(\frac{\tau-4}{\tau-1/3}\right)&=-\frac{4}{5}\pi , & \arg\left(\frac{\tau-5/4}{\tau-1/3}\right)&=-\frac{7}{10}\pi. 
\end{align*}

\item $V $ is homothetic to $ \langle 2\zeta_{12}^3+2\zeta_{12}^2-\zeta_{12}-1, 1 \rangle_{\mathbb{Q}}$;  it is of exact type 2\circled{3}.
Setting $\tau=2\zeta_{12}^3+2\zeta_{12}^2-\zeta_{12}-1$, the two rational triples are $(1,\tau,\tau-1)$ and $(\tau+3,\tau-2,\tau-3/4)$; we have
\begin{align*}
 \arg(\tau)&=\frac{7}{12}\pi , & \arg(\tau-1)&=\frac{2}{3}\pi, &
 \arg\left(\frac{\tau-2}{\tau+3}\right)&=\frac{5}{12}\pi , & \arg\left(\frac{\tau-3/4}{\tau+3}\right)&=\frac{\pi}{3}. 
\end{align*}

\end{enumerate}
The spaces in parts (2)--(5) are illustrated in Section~\ref{sec:figure:3+3}.
\end{theorem}

\section{Spaces with \texorpdfstring{$ab=cd$}{ab=cd}}\label{sect:abcd}
Under the assumption that $ab=cd$, equation \eqref{eqn:main} splits, and after removing multiple factors and trivial solutions one is left with 
\begin{equation}\label{eq:abcd}
f_{ab=cd} := (d-b)xyz + (b-c) xy + (a-d)xz + (c-a)yz + (c-a)x  + (a-d)y + (b-c)z + (d-b) = 0.
\end{equation}
As explained in Section~\ref{subsec:corresp}, in this case it is not possible to recover a single $\tau$ that generates the corresponding space. A solution to $\eqref{eq:abcd}$ need not correspond to a space, and when it does, $\tau$ has (in general) degree two over $\Q(x,y,z)$. One finds two (in principle non-homothetic) spaces corresponding to the same algebraic solution.
This peculiarity justifies a separate study of this case.

As in Proposition~\ref{prop:LongSubsum3plus2} one shows:
\begin{proposition}\label{prop:LongSubsumabcd}
With notation as in Proposition~\ref{prop:Subsums},
\begin{enumerate}
\item if $\ell \geq 5$ then $M=\mathbb{Z}^3$, $e=1$, and the (squarefree) common order $n$ of $x, y, z$ satisfies $\sum_{p \mid n} (p-2) \leq 8-2$;
\item in all cases, if $\mathbb{Z}^3/M$ is finite, then its exponent $e$ satisfies $e \leq 4$.
\end{enumerate}
\end{proposition}

In this case, the argument of Section~\ref{sect:GaloisConjugates} is not necessary: we immediately obtain the following corollary, which will be enough for our purposes.

\begin{corollary}
Let $(a,b,c,d,x,y,z)$ be a solution to Equation \eqref{eq:abcd} not belonging to an unbounded family and let $n$ be the least common multiple of the orders of $x, y, z$. Then $n$ divides one among $\{2 \cdot 3 \cdot 5, 2 \cdot 3 \cdot 7, 2^3 \cdot 3, 2 \cdot 3^2 \}$.
\end{corollary}
\begin{proof}
In case (1) of Proposition~\ref{prop:LongSubsumabcd} one checks immediately that $m$ divides $2 \cdot 3 \cdot 5$ or $2 \cdot 3 \cdot 7$. In case (2), we can assume $\ell \leq 4$, and Proposition~\ref{prop:Subsums} (2) yields that the common order of $x, y, z$ divides $6e$, where $e \in \{1,2,3,4\}$.
\end{proof}

The bound provided by the previous corollary is sharp enough that we can test all possible cases.
\begin{theorem}\label{thm:abcd}
Let $(a,b,c,d,x,y,z)$ be a solution to \eqref{eq:abcd}. Then one of the following holds:
\begin{enumerate}
\item the solution belongs to an unbounded family;
\item two of the variables $a, b, c, d$ are equal to zero;
\item $x=y=z=-1$;
\item $n=10$, and the solution is \se to $\left( 3, \frac{1}{2}, \frac{3}{2}, 1, \zeta_{10}, \zeta_{10}^2, \zeta_{10}^3 \right)$;
\item $n=12$, and the solution is \se to one of the two solutions 
\[
\left( 2, \frac{2}{3}, \frac{4}{3}, 1, \zeta_{12}, \zeta_{12}^3, \zeta_{12}^5 \right), \quad \left( \frac{4}{3}, \frac{2}{3}, \frac{8}{9}, 1, \zeta_{12}, \zeta_{12}^4, \zeta_{12}^5 \right);
\]

\item $n=6$, and the solution is \se to one of the following:
\begin{enumerate}
\item $\displaystyle \left(\frac{3cd}{c+2d} , \frac{c+2d}{3} , c, d, \zeta_6, \zeta_6, \zeta_6^2 \right)$ for some $(c,d) \in \mathbb{Q}^2$ with $c+2d \neq 0$;
\item $\displaystyle \left( 2d, \frac{1}{2}c, c,d ,\zeta_6, \zeta_6^2, \zeta_6^2\right)$ for some $(c,d) \in \mathbb{Q}^2$;
\item $\displaystyle \left( \frac{3 d^2-4 bd}{2d-3b},b,\frac{3 b d-4 b^2}{2d-3b},d,\zeta_6, \zeta_6^2, \zeta_6^3 \right)$ for some $(b,d) \in \mathbb{Q}^2$ with $2d-3b \neq 0$.
\end{enumerate}
\end{enumerate}
\end{theorem}

As mentioned, it is not necessarily the case that every solution to the algebraic equations actually leads to a valid geometric configuration. Concerning the geometric interpretation of the previous theorem, we have the following result: 
\begin{theorem}\label{thm:abcdGeometric}
\leavevmode
\begin{enumerate}
\setcounter{enumi}{0}
\item If the solution corresponds to a space, it is a rectangular space;
\item if the solution corresponds to a space, it is a superrectangular space;
\item the solution $(a,b,c,d,-1,-1,-1)$ with $ab=cd$ is geometrically acceptable if and only if $ab < 0$; in this case, the corresponding spaces are CM (Lemma~\ref{lemma:2RightAnglesImplyCM});
\item these solution corresponds to 2 geometric configurations, up to homothety. Setting $(a,b,c,d)=(3, \frac{1}{2}, \frac{3}{2}, 1)$ and
\[
\tau_1 = -\frac{1}{2} \sqrt{\frac{1}{2} \left(2 \sqrt{5}+\sqrt{52 \sqrt{5}+45}+13\right)} \cdot e^{\frac{3\pi i}{10}}, \quad \tau_2 = -\frac{1}{2} \sqrt{\frac{1}{2} \left(2 \sqrt{5}+\sqrt{52 \sqrt{5}+45}+13\right)} \cdot e^{\frac{3\pi i}{10}}
\]
gives the two non-equivalent configurations. For both values of $\tau$ one has
\[
\operatorname{arg} (\tau / \overline{\tau} ) = \frac{3}{10} \cdot 2 \pi, \quad 
\operatorname{arg} \left( \frac{\tau +a}{\tau + b} \bigm/ \frac{\overline{\tau}+a}{\overline{\tau}+b} \right) = \frac{2}{5} \cdot 2 \pi, \quad \operatorname{arg} \left( \frac{\tau +c}{\tau + d} \bigm/ \frac{\overline{\tau}+c}{\overline{\tau}+d} \right) = \frac{1}{10} \cdot 2\pi,
\]
and these are the only rational angles in $\langle 1, \tau \rangle$, which is in particular a space of exact type 3\circled{2};
\item none of these solutions gives rise to a geometric configuration;
\item the families (a) and (b) parametrise the same spaces. The parametrisation in (b) gives a geometrically valid solution if and only if $c^2-8cd+4d^2>0$. For such $c, d$ there are two corresponding values of $\tau$, and they generate homothetic spaces. Denote by $V_{[c: d]}$ the resulting space (up to homothety), then $V_{[c : d]}$ generically has 4 rational angles, of squared amplitudes $\zeta_6^{\pm 1}, \zeta_6^{\pm 2}$. It has less than $4$ rational angles if and only if $c+2d=0$, in which case $V_{[c:d]}$ is of exact type $3 \circled{2}$. The space $V_{[c:d]}$ has more than 4 rational angles if it is either the CM space $\Q(\zeta_3)$, or if it is not CM but has a right angle, in which case it is of exact type $5\circled{2}$. The latter possibility happens if and only if the space $V_{[c:d]}$ corresponds to a point on family (c), which is in fact a rational sub-family of (b).
\end{enumerate}

\end{theorem}

\begin{proof}
Parts (1)-(5) are straightforward computations, so we only prove (6).

The solutions in families (a) and (b)
 for the same pair $(c,d)$ and the same angle $z=\zeta_6^2$ lead to the same equations 
\eqref{eqn:SecondoAngolo}, \eqref{eqn:SecondoAngolo.r}, hence parametrise the same spaces, which we describe in greater detail below.

We show that the family (c) is a sub-family of (b).
Clearly, the family in (b) is equivalent to $\left(c, d, 2d, \frac{1}{2}c, \zeta_6, \zeta_6^2, \zeta_6^2 \right)$. Given a point
\[
\displaystyle \left(\frac{3 d_0^2-4 b_0d_0}{2d_0-3b_0},b_0,\frac{3 b_0 d_0-4 b_0^2}{2d_0-3b_0}, d_0, \zeta_6, \zeta_6^2, \zeta_6^3 \right)
\]
in family (c), we take
\[
c = d_0 \frac{3d_0-4b_0}{2d_0-3b_0}, \quad d=b_0:
\]
this gives a point on family (b), and the corresponding values of $\tau$ satisfy the same quadratic equation (specifically, that coming from eliminating $\overline{\tau}$ from the equations $\tau/\overline{\tau}=\zeta_6$ and $\frac{\tau+c}{\tau+d} \big/ \overline{ \left(\frac{\tau+c}{\tau+d}\right) }=\zeta_6^2$). Thus, (c) is a sub-family of (b). 

Let us show that it is precisely the subfamily such that the quantity $(c-d)(c-4d)$ is a square.
Indeed, on the one hand we have
\[
\begin{aligned}
(c-d)(c-4d) & = \left( d_0 \frac{3d_0-4b_0}{2d_0-3b_0} - b_0 \right)\left(d_0 \frac{3d_0-4b_0}{2d_0-3b_0}-4b_0 \right) \\
& = 3 \frac{d_0^2-2b_0d_0+b_0^2}{2d_0-3b_0} \cdot 3  \frac{d_0^2-4b_0d_0+4b_0^2}{2d_0-3b_0} \\
& = \left(3\frac{(d_0-b_0)(d_0-2b_0)}{2d_0-3b_0}\right)^2;
\end{aligned}
\]
conversely, let $P_{(b)}=(c_0,d_0,2d_0,\frac{1}{2}c_0,\zeta_6, \zeta_6^2, \zeta_6^2)$ be a point on family (b) and suppose that $(c_0-d_0)(c_0-4d_0)$ is a square. We take $b=d_0$ and $d$ such that
\[
c_0 = \frac{3d^2-4bd}{2d-3b} \Leftrightarrow c_0(2d-3d_0) =3d^2-4d_0d \Leftrightarrow 3d^2 - 2d (2d_0+c_0)+2c_0d_0=0.
\]
The (reduced) discriminant of this quadratic equation in $d$ is $(2d_0+c_0)^2-9c_0d_0=4d_0^2-5c_0d_0 + c_0^2=(c_0-d_0)(c_0-4d_0)$, which is a square by assumption. Thus, a point of family (b) for which $(c_0-d_0)(c_0-4d_0)$ is a square gives rise to the point $P_{(c)}=\left(\frac{3d^2-4bd}{2d-3b}, b, \frac{3bd-4b^2}{2d-3b}, d, \zeta_6, \zeta_6^2, \zeta_6^3\right)$ on family (c). Note that $P_{(c)}$ gives the same quadratic equation for $\tau$ as the original point $P_{(b)}$.
Finally, we should check that the point $P_{(c)}$ that we obtained does in fact satisfy $2d-3b \neq 0$. Assume the contrary. Replacing the values found above for $b, d$ we obtain
\[
\begin{aligned}
2d=3b & \Leftrightarrow 4d_0+2c_0 \pm 2 \sqrt{(c_0-d_0)(c_0-4d_0)} = 9d_0 \\
& \Leftrightarrow 4(c_0^2-5c_0d_0+4d_0^2) = (5d_0-2c_0)^2\\
& \Leftrightarrow 9d_0^2 = 0,
\end{aligned}
\]
which corresponds to a CM space.

Finally, we describe which solutions correspond to geometric configurations. This happens when  \eqref{eqn:SecondoAngolo.r} has a real solution, which happens if and only if $c^2-8cd+4d^2>0$.
Let us show that in this case to a point $[c,d]\in\mathbb{P}_1(\Q)$ corresponds a unique lattice up to homotethy.

Indeed, on the one hand it is clear that rescaling $(c, d)$ by a non-zero rational number gives a homothetic space. On the other hand, for fixed $(c, d)$, the relevant values of $\tau$ are obtained as the roots of a quadratic equation (so there are in principle two $\tau$ for any given $(c,d)$), but in fact the two lattices obtained in this way are always homothetic. To see this, notice that $\tau_1, \tau_2$ are the roots of the equation
\[
(-\zeta_6^2 + 2)\tau^2 + (2c + d)\tau + (\zeta_6 + 1)cd=0;
\]
it follows easily that $\tau_1\tau_2$ and $\tau_1+\tau_2$ both belong to the $2$-dimensional $\Q$-vector space $\Q(\zeta_6)$, hence the three numbers $1, \tau_1\tau_2$ and $\tau_1+\tau_2$ are $\Q$-linearly dependent.
It is easy to see that---in complete generality---the lattices $\langle 1, \tau_1 \rangle_{\Q}$ and $\langle 1, \tau_2 \rangle_{\Q}$ are homothetic if and only if the four numbers $1, \tau_1, \tau_2, \tau_1\tau_2$ are linearly dependent over $\Q$ (see \cite{DvoVenZan}*{Remark 2.1}), which proves the claim.

Let $[c,d]\in\mathbb{P}_1(\Q)$ be such that  $c^2-8cd+4d^2>0$. Fix
\begin{equation*}
 \tau=-\zeta_{12}\frac{\sqrt{3}}{6}(c+2d + \sqrt{c^2-8 c d +4 d^2})
\end{equation*}
and write $V_{[c:d]}=\generated{1,\tau}$.

Let $S=\{[c:d]\in\mathbb{P}_1(\Q)\mid c^2-8 c d +4 d^2\text{ is a rational square}\}\cup \{[2:-1]\}$.

The map
\[
\begin{array}{cccc}
f : & \mathbb{P}_1(\mathbb{Q}) & \to & \{ \text{homothety classes of spaces} \} \\
& [c : d] & \mapsto & \text{class of } V_{[c : d]}
\end{array}
\]
is 2-to-1 on $\mathbb{P}_1(\mathbb{Q})\setminus S$. The involution $\iota : [c : d] \mapsto [4d : c]$ preserves the fibres of $f$.

If $c^2-8cd+4d^2$ is a rational square, then $V_{[c:d]}=\Q(\zeta_3)$ is a CM space; this happens for infinitely many $[c:d]$.

If $c+2d=0$ we have that $V_{[2:-1]}$ has exact type  $3 \circled{2}$.

In all other cases $V_{[c:d]}$ contains 4 rational angles $(1,\tau), (\tau+\frac{3cd}{c+2d} , \tau+\frac{c+2d}{3}) , (\tau+2d, \tau+\frac{1}{2}c), (\tau+c,\tau+d)$, all distinct and not adjacent, with squared amplitudes $\zeta_6,\zeta_6,\zeta_6^2,\zeta_6^2$ respectively. 

The space $V_{[c:d]}$ contains additional rational angles if and only if $(c-d)(c-4d)$ is a square. In this case, $(\tau+e, \tau+f)$, with $e, f$  given by $\frac{(c+2d) \pm \sqrt{(c-d)(c-4d)}}{3}$ respectively, is a right angle and there are no other rational angles in this space. This rational sub-family is preserved by the involution $\iota$.

To see that no other angles can appear, observe that such a square amplitude must lie in the field $\Q(\tau,\zeta_6)=\Q(\zeta_6, \sqrt{c^2-8cd+4d^2})$, which is an extension of $\Q(\zeta_6)$ of degree al most 2. The only possibilities are $\zeta_6^1,\zeta_6^2,\zeta_6^3$ or a primitive 12th root of unity; in this latter case however we must have that  $\sqrt{c^2-8cd+4d^2}\in\Q(\zeta_{12})\cap \R = \Q(\sqrt{3})$, and this can only happen if $c^2-8cd+4d^2$ is a rational square because 3 is not a square modulo 4.

We finally remark that all four Galois conjugates of $V_{[c:d]}$ are homothetic to $V_{[c:d]}$. We have already showed it for the Galois morphism that acts trivially on $\zeta_6$. We now notice that $-\overline{\zeta_{12}}\frac{\sqrt{3}}{6}(c+2d - \sqrt{c^2-8 c d +4 d^2})$ times $\tau$ is a rational number, so this conjugate too generates a space homothetic to $V_{[c:d]}$.

As for $\overline{\tau}$, one can check directly that $\sqrt{c^2-8 c d +4 d^2}\in \Q(\tau+\overline{\tau})$ and from there that $\overline{V_{[c:d]}}$ is also homothetic to $V_{[c:d]}$.
\end{proof}

\section{The 2+2+2 case}\label{sect:222}
\subsection{Bounding the degree}
Again in analogy with Proposition~\ref{prop:LongSubsum3plus2}, one proves:
\begin{proposition}\label{prop:LongSubsum222}
With notation as in Proposition~\ref{prop:Subsums},
\begin{enumerate}
\item if $\ell \geq 10$ then $e \leq 2$, and the common order $n$ of $x, y, z$ divides $2 \prod_{p \leq \ell} p$ (hence divides $2 \prod_{p \leq 23} p$);
\item in all cases, if $\mathbb{Z}^3/M$ is finite, then its exponent $e$ satisfies $e \in \{1,\ldots,22\} \cup \{24,26,30\}$.
\end{enumerate}
\end{proposition}

\begin{remark}
A coarser, but still useful, bound as in (2) can be obtained by noticing that the exponent of $\mathbb{Z}^3/M$ is bounded by the non-zero determinant of a $3 \times 3$ matrix $A$ with columns given by three vectors of the form $\underline{i} - \underline{i}_j$ (notation as in the proof of Proposition~\ref{prop:LongSubsum3plus2}). Since each coefficient of $\underline{i} - \underline{i}_j$ is bounded in absolute value by $2$, one immediately obtains that $\det(A) \leq 3! \cdot 2^3=48$.
\end{remark}

The argument of Corollary~\ref{cor:Bound32Step1} now gives the following.

\begin{corollary}
Let $(a,b,c,d,x,y,z)$ be a solution to Equation \eqref{eqn:main}. If the solution does not belong to an unbounded family, the common order $n$ of $x,y,z$ satisfies at least one of the following:
\begin{enumerate}
\item $n$ is squarefree or twice a squarefree, and divisible only by primes $\leq 23$;
\item $n$ is of the form $em$, with $m$ squarefree and divisible only by primes $\leq 7$, and $e \in \{1,...,22\} \cup \{24,26,30\}$.
\end{enumerate}
In particular, $n$ divides $2^5 \cdot 3^3 \cdot 5^2 \cdot 7^2 \cdot 11 \cdot 13 \cdot 17 \cdot 19 \cdot 23$.
\end{corollary}

Reasoning as in Corollary~\ref{cor: sharp bound case 2+3}, one can then check the following.

\begin{proposition}\label{prop:RulingOutDivisors}
Let $(a,b,c,d,x,y,z)$ be a solution to Equation \eqref{eqn:main}, and let $n$ be the common order of $x,y,z$. The following implication holds.
suppose that $q$ divides $n$ exactly, where $q$ is one of the following:
\begin{itemize}
\item $11, 13, 17, 19, 23$;
\item $2^5, 3^3, 5^2, 7^2$; %
\item $5 \cdot 7, 3^2 \cdot 7, 2^4 \cdot 7, 2^3 \cdot 7$;
\item $3^2 \cdot 5, 2^4 \cdot 5$.
\end{itemize}
Then $[a : b : c : d] \in \mathbb{P}_{3}(\mathbb{Q})$ satisfies one of the following:
\begin{itemize}
\item $ab=cd$;
\item $abcd(a-b)(a-c)(a-d)(b-c)(b-d)(c-d)=0$.
\end{itemize}
\end{proposition}

\begin{corollary}
In the setting of Proposition~\ref{prop:RulingOutDivisors}, at least one of the following holds:
\begin{enumerate}
\item $ab=cd$;
\item $abcd(a-b)(a-c)(a-d)(b-c)(b-d)(c-d)=0$;
\item $(a,b,c,d,x,y,z)$ \bif;
\item $n \in \{ 1, 2, 3, 4, 5, 6, 7, 8, 9, 10, 12, 14, 15, 16, 18, 20, 21, 24, 28, 30, 36, 40, 42, 48, 60, 72, 84, 120, 144 \}$.
\end{enumerate}
\end{corollary}
Solutions as in (1) are classified in Section~\ref{sect:abcd}, while solutions falling in case (2) are either geometrically meaningless (if, for example, $a=b$) or correspond to spaces having at least a triple of adjacent rational angles. These have been classified in Section~\ref{sect:32}. The unbounded families are described in Section~\ref{sec:Rettangolari}. Thus, we just need to study Equation \eqref{eqn:main} for the finitely many values of $n$ listed in (4), and we may do so under the additional constraints $abcd(a-b)(a-c)(a-d)(b-c)(b-d)(c-d) \neq 0$, $ab \neq cd$.

\begin{remark}\label{rmk: intersecting conjugates of S}
For a given value of $n \geq 3$ and fixed $n$-th roots of unity $x,y,z$, equation~\eqref{eqn:main} becomes the equation of a surface $S$ which is in general defined over $\mathbb{Q}(\zeta_n)^+$ and not defined over $\Q$ unless $x,y,z$ are fourth or sixth roots of unity (see the appendix of \cite{DvoVenZan} for discussion of the reducibility and field of definition of $S$). Rational points $(a : b : c : d)$ on $S$ must therefore lie in the intersection of $\frac{1}{2}\varphi(n)$ conjugates of $S$. Heuristically, we should expect this intersection to consist of curves when $\varphi(n)=4$ (which happens precisely for $n=5, 8, 10, 12$), and at most of isolated points otherwise. This is precisely what happens, and the isolated rational points are typically highly singular points of $S$. These points, such as $(1 : 1 : 1 : 1)$ or $(1 : 0 : 0 : 0)$, often show up for many different values of $n$ and many choices of the roots of unity $x, y, z$. In this respect, it is maybe worth mentioning that, for $n=30$ and $(x,y,z)=(\zeta_{30}^{5}, \zeta_{30}^6, \zeta_{30}^{10})$, we find the `exceptional' rational point $(a : b : c : d)=(2 : 2 : 4 : 1)$, which is a solution only for this choice of $n$ and roots of unity (up to symmetry equivalence); however it does not satisfy the nondegeneracy conditions and does not correspond to any geometric configuration.
\end{remark}

The computation to be performed now amounts to finding the rational points on a certain number of algebraic varieties, most of which are empty or consist only of finitely many points; however curves of genus up to 5 are also encountered. We discuss in Proposition~\ref{prop:HigherGenusCurves} which curves we encounter, and we present in Section~\ref{sec:comput.details} the techniques used to study the set of their rational points. The outcome of the computation is the following theorem:

\begin{theorem}\label{thm:Enumerate222}
Let $(a,b,c,d,x,y,z)$ be a solution to Equation \eqref{eqn:main} satisfying the non-degeneracy conditions \eqref{eq:Conditions} and the additional constraints $abcd(a-b)(a-c)(a-d)(b-c)(b-d)(c-d) \neq 0$, $ab \neq cd$. Let $n$ be the common order of $x,y,z$. Then $n \in \{5,8,10,12\}$.
\end{theorem}

Building on Remark~\ref{rmk: intersecting conjugates of S}, we explicitly note that for $n \not \in \{5,8,10,12\}$ we only need to study the rational points of a zero-dimensional scheme, which of course does not pose any great difficulties. It would be possible to study directly the $1$-dimensional intersections that arise for $n \in \{5, 8, 10, 12\}$, but we find it easier to analyse these cases more directly and from a more geometric perspective, which we do in the following subsections. We note that the direct analysis, with much more tedious computations, leads to the same conclusions that we report in the rest of this section.

\subsection{General structure}
In listing all solutions for a given degree we will present them as a rationally parametrised family of spaces having a certain angle, with additional angles appearing when certain expressions are squares in $\Q$. We give here the general structure of these theorems. Note that, if the space $V$ contains an angle of squared amplitude $\zeta_n^k$, then (choosing generators that form an angle of that squared amplitude and rescaling suitably) we can write $V=\langle 1, \tau\rangle_{\mathbb{Q}}$, where $\tau$ is in $\mathbb{R} \zeta_{2n}^k$.

\begin{theorem}\label{thm: general description angles in rational families}
Fix $(n, k)$ in $\{(8,1), (10,1), (10,2), (12,1), (12,3) \}$. The $\mathbb{Q}$-vector subspace of $\mathbb{Q}(\zeta_n)$ given by $\mathbb{Q}(\zeta_n) \cap \mathbb{R}\zeta_{2n}^k$ has dimension 2. Let $\tau_{1,0}, \tau_{0,1} \in \mathbb{Q}(\zeta_n)$ be the basis of this space given by the following table
\[
\begin{array}{|c|c|c|}
\hline
 (n,k) & \tau_{1,0} & \tau_{0,1}\\
 \hline
 (8,1) & \zeta_8^2-\zeta_8^3 & 1+\zeta_8\\
 \hline
 (10,1)  & -2-2\zeta_{10}^2+\zeta_{10}^3 & 1+\zeta_{10}\\
 \hline
 (10,2)  & 1+\zeta_{10}^2 & \zeta_{10} \\
 \hline
 (12,1) & \zeta_{12}+ \zeta_{12}^2-\zeta_{12}^3 & 1+\zeta_{12} \\
 \hline
 (12,3) & \zeta_{12}+\zeta_{12}^2 & 1+\zeta_{12}^3\\
 \hline
\end{array}
\]
For $(\lambda, \mu) \in \mathbb{Q}^2 \setminus\{ (0,0) \}$, set
\[
\tau_{\lambda, \mu} = \lambda \tau_{1,0} + \mu \tau_{0,1}.
\]
The space $\langle 1, \tau_{\lambda, \mu}\rangle_{\mathbb{Q}}$ only depends on $[\lambda:\mu] \in \mathbb{P}_1(\mathbb{Q})$. We denote it by $V_{[\lambda : \mu]}$.

\begin{enumerate}
\item Fix $e \in \mathbb{N}$. There exists a homogeneous quartic form $f_{n, k, e}(\lambda, \mu) \in \mathbb{Q}[\lambda, \mu]$ such that the following holds: the space $V_{[\lambda : \mu]}$ contains an angle with squared amplitude $\zeta_n^e$ if and only if $f_{n, k, e}(\lambda, \mu)$ is the square of a rational number $u$. 
\item There exist four rational functions $a_{1}, a_{2}, b_{1}, b_{2} \in \mathbb{Q}(\lambda, \mu)$ such that, setting
\[
a_\pm (\lambda, \mu, u) = a_1(\lambda, \mu) \pm a_2(\lambda, \mu) u, \quad b_\pm (\lambda, \mu, u) = b_1(\lambda, \mu) \pm b_2(\lambda, \mu) u,
\]
the following holds. Suppose that $(\lambda, \mu, u) \in\ \mathbb{Q}^3$, with $(\lambda, \mu) \neq (0,0)$, is a solution of the equation $u^2=f_{n, k, e}(\lambda, \mu)$: then, the rational angles in the space $V_{[\lambda : \mu]}$ with squared amplitude $\zeta_n^e$ are given by
\[
(\tau_{\lambda, \mu} + a_+(\lambda, \mu, u), \tau_{\lambda, \mu} + b_+(\lambda, \mu, u)), \quad (\tau_{\lambda, \mu} + a_-(\lambda, \mu, u), \tau_{\lambda, \mu} + b_-(\lambda, \mu, u)).
\]
These two angles may coincide or have a side in common.
\item The equation $u^2= f_{n, k, e}(\lambda, \mu)$ defines a curve in the weighted projective space $\mathbb{P}_{\mathbb{Q}}(1,1,2)$ (the variables are taken in the order $\lambda, \mu, u$). This curve is either the union of two irreducible components of genus 0 or a curve of genus $1$.
\end{enumerate}

\end{theorem}

The theorem simply describes what happens carrying out the computation. For the angles considered, $\mathbb{Q}(\zeta_n) \cap \mathbb{R}\zeta_{2n}^k$ has dimension 2, and cannot have dimension greater than 2, as $\mathbb{Q}(\zeta_n) \cap \mathbb{R}$ has also dimension 2, and the whole $\mathbb{Q}(\zeta_n)$ has dimension 4.
The existence of angles of a fixed amplitude correspond to the existence of rational points on the conic given by \eqref{eqn:1angle:prima} or \eqref{eqn:SecondoAngolo}, which in this case is defined over a real quadratic field. A rational point must therefore lie in the intersection with the conjugate conic, and expressing $\tau_{[\lambda,\mu]}$ it terms of the basis leads to the forms $f_{n,k,e}$.

We now prove in general certain geometric properties of a class of curves that we encounter in the computations.
\begin{proposition}\label{prop:HigherGenusCurves}
Let $f_1(\lambda,\mu)$ and $f_2(\lambda,\mu)$ be two non-proportional homogeneous forms of degree 4. Let $C$ be the curve defined in $\mathbb{P}_\mathbb{Q}(1,1,2,2)$ (weighted projective space in the variables $\lambda, \mu, u, v$, of respective weights $1,1,2,2$) by
\[
C : \; \begin{cases}
u^2 = f_1(\lambda, \mu) \\
v^2 = f_2(\lambda, \mu)
\end{cases}
\]
The following hold:
\begin{enumerate}
\item $C$ is a Galois cover of $\mathbb{P}_{1,\mathbb{Q}}$ with Galois group $\left(\mathbb{Z}/2\mathbb{Z}\right)^2$;
\item if $f_1, f_2$ have no common factors, $C$ is smooth, has genus 5, and is an étale double cover of the hyperelliptic, genus 3 curve $C' : \, w^2=f_1(\lambda,\mu)f_2(\lambda,\mu)$ (sitting in $\mathbb{P}_\mathbb{Q}(1,1,4)$).
\item if the greatest common divisor of $f_1, f_2$ is a quadratic polynomial $q_1(\lambda_1,\lambda)$, then $C$ is not smooth, and its desingularisation $\tilde{C}$ has genus 3. Write $f_1=q_1(\lambda,\mu)q_2(\lambda,\mu)$, $f_2=q_1(\lambda,\mu)q_3(\lambda,\mu)$, where each $q_i(\lambda,\mu)$ has degree 2. The curve $C$ (hence also $\tilde{C}$) admits degree-2 maps to the three genus-1 curves $w^2=q_i(\lambda,\mu)q_j(\lambda,\mu)$ (sitting in $\mathbb{P}_{\mathbb{Q}}(1,1,2)$) for $i,j = 1, 2, 3$ with $i \neq j$.
\end{enumerate}
\end{proposition}
\begin{proof}
\begin{enumerate}
\item The automorphisms $\alpha : [\lambda : \mu : u : v] \mapsto [\lambda : \mu : -u : v]$ and $\beta : [\lambda : \mu : u : v] \mapsto[\lambda:\mu:u:-v]$ generate a subgroup $G$ of $\operatorname{Aut}_{\mathbb{Q}}(C)$ isomorphic to $\left( \mathbb{Z}/2\mathbb{Z} \right)^2$. The quotient map $C \to C/G$ can be identified with
\[
\begin{array}{cccc}
f : & C & \to & \mathbb{P}_{1,\mathbb{Q}} \\
& [\lambda : \mu : u : v] & \mapsto & [\lambda : \mu]:
\end{array}
\]
indeed, it is clear that the fibers of $f$ are invariant under the action of $G$, and that the degree of $f$ is $4=|G|$.
\item 
Smoothness of $C$ follows immediately from the fact that all derivatives should vanish at a non-smooth point, which would force $f_1, f_2$ to have a common factor. To compute the genus of $C$ we study the ramification of $f$. 
As $C$ is smooth, $f$ is the quotient map by a group action, and the group action is tame since we are working in characteristic 0, it suffices to understand the fixed points of each non-identity element in $G$. The fixed points of $\alpha$ (respectively $\beta$) can be identified with the points $[\lambda : \mu : v]$ that solve $f_1(\lambda,\mu)=0$ and $v^2=f_2(\lambda,\mu)$ (respectively, $f_2(\lambda,\mu)=0$ and $u^2=f_1(\lambda,\mu)$), hence -- using the assumption that $f_1,f_2$ have no common factors -- both $\alpha$ and $\beta$ have 8 fixed points. The automorphism $\alpha\beta$ has no fixed points, again because of the hypothesis that $f_1, f_2$ have no common factors.
The Riemann-Hurwitz formula then yields
\[
2g(C) - 2 = 4(2g(\mathbb{P}_1)-2) + \sum_{h \in G} (\operatorname{ord}(h)-1) \# \operatorname{Fix}(h) = -8 +  8 +8 \Longrightarrow g(C)=5.
\]
To prove the last statement of (2), notice that
\[
\begin{array}{ccc}
C & \to & C' \\
\left[\lambda : \mu : u : v \right] & \mapsto & [\lambda : \mu : uv]
\end{array}
\]
is well-defined, of degree 2, and étale (it is the quotient map for the action of $\alpha\beta$, which as already observed has no fixed points). Furthermore, $C'$ is hyperelliptic and defined by a separable polynomial of degree $8$, hence has genus $3 = \frac{8-2}{2}$.
\item Without loss of generality we may work with the base-change of $C$ to an algebraic closure of $\mathbb{Q}$. Write $D_i$ for the divisor cut out on $\mathbb{P}_1$ by $q_i(\lambda,\mu)=0$: it consists of 2 distinct points of multiplicity 1, and the supports of $D_1, D_2, D_3$ are pairwise disjoint. To see that $C$ is not smooth, one may check that the points $[\lambda:\mu:0:0]$, where $[\lambda:\mu]$ is in the support of $D_1$, are singular. The maps to the three genus-1 curves described in the statement can be taken to be
\[
\begin{array}{ccc}
C & \to & w^2 = q_1(\lambda,\mu)q_2(\lambda,\mu) \\
\left[ \lambda : \mu : u : v \right] & \mapsto & [\lambda : \mu : u]
\end{array}
\]
\[
\begin{array}{ccc}
C & \to & w^2 = q_1(\lambda,\mu)q_3(\lambda,\mu) \\
\left[ \lambda : \mu : u : v \right] & \mapsto & [\lambda : \mu : v]
\end{array}
\]
and
\[
\begin{array}{ccc}
C & \to & w^2 = q_2(\lambda,\mu)q_3(\lambda,\mu) \\
\left[ \lambda : \mu : u : v \right] & \mapsto & [\lambda : \mu : uv/q_1(\lambda,\mu) ].
\end{array}
\]
To compute the genus of $\tilde{C}$ we work at the level of function fields and analyse the ramification of the map $\tilde{f} : \tilde{C} \to \mathbb{P}_1$ induced by the quotient map $f : C \to \mathbb{P}_1$. It is clear that the only places of $\overline{\mathbb{Q}}(\mathbb{P}_1)$ (potentially) ramified in $\overline{\mathbb{Q}}(\tilde{C})$ are those corresponding to points in the support of $D_1+D_2+D_3$. We claim that each of these six places has ramification index 2; this implies that there are two places of $\overline{\mathbb{Q}}(\tilde{C})$ lying over it, and by Riemann-Hurwitz we get
\[
2g(\tilde{C}) - 2 = 4 (2g(\mathbb{P}_1) -2) + 6 \cdot 2 \cdot (2-1) \Longrightarrow g(\tilde{C})=3.
\]
To prove the statement about ramification, observe that
\[
\overline{\mathbb{Q}}(\tilde{C})=\overline{\mathbb{Q}}(\mathbb{P}_1)(\sqrt{q_1(\lambda/\mu,1) q_2(\lambda/\mu,1)}, \sqrt{q_1(\lambda/\mu,1) q_3(\lambda/\mu,1)}, \sqrt{q_2(\lambda/\mu,1) q_3(\lambda/\mu,1)})
\]
is a biquadratic extension, and therefore the role of the three divisors $D_1, D_2, D_3$ is symmetrical. Thus it suffices to prove the claim for a place $P$ of $\overline{\mathbb{Q}}(\mathbb{P}_1)$ corresponding to a point in the support of $D_3$. The place $P$ is unramified in the intermediate extension $\overline{\mathbb{Q}}(\sqrt{q_1(\lambda/\mu,1) q_2(\lambda/\mu,1)}) / \overline{\mathbb{Q}}(\mathbb{P}_1)$, while it is ramified in $\overline{\mathbb{Q}}(\sqrt{q_1(\lambda/\mu,1) q_3(\lambda/\mu,1)}) / \overline{\mathbb{Q}}(\mathbb{P}_1)$, hence also in $\overline{\mathbb{Q}}(\tilde{C}) / \overline{\mathbb{Q}}(\mathbb{P}_1)$. This immediately implies our claim.
\end{enumerate}
\end{proof}

\begin{remark}
Proposition~\ref{prop:HigherGenusCurves} will be applied to the following situation. By Theorem~\ref{thm: general description angles in rational families} (1), the existence of an angle of squared amplitude $\zeta_n^e$ in the space $V_{[\lambda:\mu]}$ is equivalent to the existence of a rational point on a curve of the form $u^2=f_{n,k,e}(\lambda, \mu)$. Thus, the simultaneous existence of angles of squared amplitudes $\zeta_n^{e_1}, \zeta_n^{e_2}$ is equivalent to the existence of a rational point on the fibre product
\[
C_{n,k,e_1,e_2} : \; \begin{cases}
u^2=f_{n,k,e_1}(\lambda, \mu) \\
v^2 = f_{n, k, e_2}(\lambda, \mu),
\end{cases}
\]
which is precisely the kind of curve studied in Proposition~\ref{prop:HigherGenusCurves}.

The curves $C_{n,k,e_1,e_2}$ were originally found in a completely different way, namely, as (desingularisations of) irreducible components of the intersection of the two conjugates of a surface $S$ defined over $\mathbb{Q}(\zeta_n)^+$, see Remark~\ref{rmk: intersecting conjugates of S}. The techniques of \cites{MR3882288, MR3904148, MR4280568} revealed that their Jacobians are all split, which led us to look for the more conceptual interpretation of Proposition~\ref{prop:HigherGenusCurves}.
\end{remark}

\subsection{\texorpdfstring{$n=8$}{n=8}}\label{subsect:8}

\begin{theorem}\label{thm:8}
For $(\lambda,\mu) \in \mathbb{Q}^2 \setminus \{(0,0)\}$, let
\[
\tau_{\lambda, \mu} := \lambda(\zeta_8^2-\zeta_8^3) + \mu(1+\zeta_8) = \lambda \cdot \sqrt{2-\sqrt{2}} \zeta_{16} + \mu \cdot \sqrt{2+\sqrt{2}} \zeta_{16}.
\]
The following hold:
\begin{enumerate}
\item The space $\langle 1, \tau_{\lambda, \mu}\rangle_{\mathbb{Q}}$ only depends on $[\lambda : \mu] \in \mathbb{P}_1(\mathbb{Q})$. We denote it by $V_{[\lambda : \mu]}$.
\item Let $\tau \in \mathbb{Q}(\zeta_8)$ generate a space $V$ such that the squared angle between $1$ and $\tau$ is $\zeta_8$. There exists $(\lambda,\mu) \in \mathbb{Q}^2 \setminus \{(0,0)\}$ such that $\tau = \tau_{\lambda, \mu}$.
\item Conversely, for every $(\lambda,\mu) \in \mathbb{Q}^2 \setminus \{(0,0)\}$, the squared angle between $1$ and $\tau$ is $\zeta_8$.
\item For every $[\lambda : \mu] \in \mathbb{P}_1(\mathbb{Q})$, the complex conjugate $\overline{V_{[\lambda : \mu]}}$ is homothetic to $V_{[\lambda : \mu]}$.
\item The map
\[
\begin{array}{cccc}
f : & \mathbb{P}_1(\mathbb{Q}) & \to & \{ \text{homothety classes of spaces} \} \\
& [\lambda : \mu] & \mapsto & \text{class of } V_{[\lambda : \mu]}
\end{array}
\]
is 2-to-1. The involution $\iota : [\lambda : \mu] \mapsto [\lambda + \mu : \lambda - \mu]$ preserves the fibres of $f$. %

\item Let $S= \{ [0:1],[1:0],[1:1],[-1 : 1] \}$. Notice that $\iota([0:1]) =[1 : -1]$ and $\iota([1:0]) = [1:1]$.
The spaces corresponding to $[\lambda : \mu] \in S$ have the following structure:
\begin{itemize}
\item  $[\lambda : \mu] = [0 : 1]$ or $[-1 : 1]$. The space is of type $\circled{4}$. A rational quadruple of $V_{[0:1]}$ is given by $\{1, \tau_{0,1}, \tau_{0,1}-1, \tau_{0,1}-2\}$. The various angles in this quadruple realise all the squared amplitudes $\zeta_8, \zeta_8^2, \zeta_8^3$ and $\zeta_8^4$; this is the configuration formed by the sides and diagonals of a rhombus with angles $\pi/4$ and $3\pi/4$.
\item $[\lambda : \mu] = [1:0]$ or $[1:1]$. The space is of type $2 \cdot \circled{3}$. The two rational triples of $V_{[1:0]}$ are given by $\{\tau_{1,0} + 1, \tau_{1,0} -2, \tau_{1,0} - \frac{1}{2}\}$ and $\{ 1, \tau_{1,0}, \tau_{1,0}-1\}$. This is the space of Theorem~\ref{thm:3plus2}~(2).

\end{itemize}

\item If $[\lambda:\mu]\not\in S$, the space $V_{[\lambda : \mu]}$ contains precisely two angles with squared amplitude $\zeta_8$ and precisely two angles with squared amplitude $\zeta_8^3$. These are given respectively by
\[
\zeta_8 : \quad (1,\tau) \text{ and } (\tau+a,\tau+b),
\]
where
\[
a=\frac{\lambda^2-2 \lambda \mu-\mu^2}{\lambda+\mu}, \quad b=-\frac{2 \left(\lambda^3+2 \lambda^2 \mu-\lambda \mu^2\right)}{\lambda^2-2 \lambda \mu-\mu^2}
\]
and by
\[
\zeta_8^3 : \quad (\tau+c,\tau+d) \text{ and } (\tau+c', \tau+d')
\]
where
\[
c = \frac{-\lambda^2 - 2 \lambda \mu + \mu^2}{2 \lambda}, \quad d= -2 (\lambda + \mu)
\]
\[
c' = -2 \mu, \quad d'=\frac{\lambda^2 + 2 \lambda \mu - \mu^2}{\mu-\lambda}.
\]
\item\label{grado8:fam.ell.zeta^2} Let $[\lambda : \mu] \not \in S$. The space $V_{[\lambda : \mu]}$ contains an angle with squared amplitude $\zeta_8^2$ if and only if $\lambda(\lambda^3+\lambda^2 \mu + \lambda \mu^2 + \mu^3)$ is the square of a rational number $u$. In this case, setting
\[
e_{\pm} = \frac{ \lambda^2 - \mu^2 \mp u}{\mu}, \quad f_{\pm} = \frac{-2\lambda \mu \pm u}{\lambda},
\]
the space $V_{[\lambda : \mu]}$ contains precisely two angles with squared amplitude $\zeta_8^2$, given by $(\tau+e_+,\tau+f_+)$ and $(\tau+e_-,\tau+f_-)$. The curve in weighted projective space
\[
u^2 = \lambda(\lambda^3+\lambda^2 \mu + \lambda \mu^2 + \mu^3)
\]
has genus 1 and a rational point, so it defines an elliptic curve. It is in particular isomorphic over $\mathbb{Q}$ to the elliptic curve with Weierstrass equation $y^2=x^3+x^2+x+1$, whose Mordell-Weil group over $\mathbb{Q}$ is $\mathbb{Z} \oplus \mathbb{Z}/2\mathbb{Z}$. A generator for the free part of the Mordell-Weil group is given by $(0,1)$. The set of $[\lambda : \mu]$ such that $V_{[\lambda : \mu]}$ contains an angle of squared amplitude $\zeta_8^2$ is stable under $\iota$, and the map $[\lambda : \mu : u] \mapsto [\lambda+\mu : \lambda-\mu : 2u]$ is a lift of $\iota$ to an involution of the elliptic curve. It corresponds to translation by the non-trivial 2-torsion point $[1 : -1 : 0]$.

\item\label{grado8:fam.ell.zeta^4}  Let $[\lambda : \mu] \not \in S$. The space $V_{[\lambda : \mu]}$ contains an angle with squared amplitude $\zeta_8^4$ if and only if $3 \lambda^4 + 2 \lambda^3 \mu + 4 \lambda^2 \mu^2 + 2 \lambda \mu^3 + \mu^4$ is the square of a rational number $v$. In this case, setting
\[
g = \frac{ \lambda^2 - 2 \lambda \mu -\mu^2 \mp v}{\lambda+\mu}, \quad h = \frac{\lambda^2-2\lambda \mu-\mu^2 \pm v}{\lambda + \mu},
\]
the space $V_{[\lambda : \mu]}$ contains precisely two angles with squared amplitude $\zeta_8^4$ are precisely $(\tau+g_+,\tau+h_+)$ and $(\tau+g_-,\tau+h_-)$. The curve in weighted projective space
\[
v^2 = 3 \lambda^4 + 2 \lambda^3 \mu + 4 \lambda^2 \mu^2 + 2 \lambda \mu^3 + \mu^4
\]
has genus 1 and a rational point, so it defines an elliptic curve. It is in particular isomorphic over $\mathbb{Q}$ to the elliptic curve with Weierstrass equation $y^2 + 2xy - 2y = x^3 - 6x^2 - 3x$, whose Mordell-Weil group over $\mathbb{Q}$ is $\mathbb{Z} \oplus \mathbb{Z}/2\mathbb{Z}$. A generator for the free part of the Mordell-Weil group is given by $(0,2)$. Notice that $g_+g_- = ab$, so that this case falls within the scope of Theorem~\ref{thm:abcdGeometric}.

\item Let $[\lambda : \mu] \not \in S$. The space $V_{[\lambda : \mu]}$ does not have both an angle with squared amplitude $\zeta_8^2$ and an angle with squared amplitude $\zeta_8^4$.
\end{enumerate}
\end{theorem}

\begin{remark} 
The reader may wonder why the curve appearing in part \eqref{grado8:fam.ell.zeta^4} of the above theorem has genus 1, while all the curves appearing in the classification of spaces with $ab=cd$ (see Theorem~\ref{thm:abcd}) have genus 0. The reason is that in Theorem~\ref{thm:abcdGeometric} we do not impose any restrictions on the field generated by $\tau$, which is (usually) a quadratic extension of the field $\mathbb{Q}(x,y,z)$, while in Theorem~\ref{thm:8} we require that $\tau$ belongs to the field $\mathbb{Q}(x,y,z)=\mathbb{Q}(\zeta_8)$. Thus we find, with a different geometric structure, a subfamily of the rectangular spaces having additional angles.

\end{remark}

\begin{remark}
Part (7) of Theorem~\ref{thm:8} shows that the angles with amplitude $\zeta_8^3$ can be rationally parametrised. An avatar of the same fact appears in the computations that verify Theorem~\ref{thm:Enumerate222}, see the output file \texttt{222.out}. Indeed, one finds that (for $x=y=\zeta_8$ and $z=\zeta_8^3$) the intersection of the conjugates of the surface $S$ described in Remark~\ref{rmk: intersecting conjugates of S} contains a rational curve. Similarly, a (highly singular) model of the curve of genus $1$ that appears in part (8) can also be obtained as the intersection of two conjugates of the surface $S$ for $x=y=\zeta_8$, $z=\zeta_8^2$. Similar comments apply to the theorems below: the various curves of genera $1$, $3$ and $5$ that we describe could also be realised (though with a very singular model) as intersections of conjugates of the surface $S$, for suitable values of the parameters $x, y, z$.
\end{remark}

\subsection{\texorpdfstring{$n=5$}{n=5} and \texorpdfstring{$n=10$}{n=10}}\label{subsect:10}
Consider a solution $(a,b,c,d,x,y,z)$ to Equation \eqref{eqn:main} with $ab \neq cd$ and $x,y,z$ having common order $n=10$. In this case, Equation \eqref{eq:tauabcd} shows that $\tau$ belongs to $\mathbb{Q}(\zeta_{10})$. One of the following necessarily holds: one among $x,y,z$ is a primitive $10$th root of unity (in which case, up to the Galois action and permutations, we may assume that $x=\zeta_{10}$), or two of the three roots of unity have order $5$ and the remaining one has order $2$ (other combinations are ruled out by Lemma~\ref{lemma:2RightAnglesImplyCM}), and as above we may assume $x=\zeta_{10}^2$.
We study the two cases in Theorem~\ref{thm:101} and~\ref{thm:102} respectively. Notice that a solution with common order $n=5$ falls within the scope of Theorem~\ref{thm:102}, by simply requiring that the only rational angles that appear in the lattice are given by $\zeta_{10}^{\pm 2}$ and $\zeta_{10}^{\pm 4}$, see Remark~\ref{rmk:nequals5}.

\begin{theorem}\label{thm:101}
For $(\lambda,\mu)\in\Q^2\setminus\{(0,0)\}$ let
\[
\tau_{\lambda,\mu} = \lambda (-2-2\zeta_{10}^2+\zeta_{10}^3) + \mu(1+\zeta_{10}).
\]
The following hold:

\begin{enumerate}
\item The space $V=\langle 1,\tau \rangle_{\mathbb{Q}}$ only depends on $[\lambda : \mu] \in \mathbb{P}_1(\mathbb{Q})$. We denote it by $V_{[\lambda,\mu]}$.

\item Let $\tau \in \mathbb{Q}(\zeta_{10})$ generate a space $V$ such that the squared angle between $1$ and $\tau$ is $\zeta_{10}$. Then there exist $(\lambda,\mu)\in\Q^2\setminus\{(0,0\}$ such that $\tau=\tau_{\lambda,\mu}$.

\item Conversely, for every $(\lambda,\mu)\in\Q^2\setminus\{(0,0\}$, the squared angle between 1 and $\tau$ is $\zeta_{10}$.

\item For every $[\lambda:\mu]\in\mathbb{P}_1(\Q)$, the complex conjugate $\overline{V_{[\lambda,\mu]}}$ is homothetic to $V_{[\lambda,\mu]}$.

\item The map
\[
\begin{array}{cccc}
f : & \mathbb{P}_1(\mathbb{Q}) & \to & \{ \text{homothety classes of spaces} \} \\
& [\lambda : \mu] & \mapsto & \text{class of } V_{[\lambda : \mu]}
\end{array}
\]
is 2-to-1. The involution $\iota : [\lambda : \mu] \mapsto [3\lambda - \mu : 4\lambda - 3\mu]$ preserves the fibres of $f$. %

\item Let $S= \{ [0:1],[1:3] \}$. Notice that $\iota([0:1]) =[1 : 3]$ and $\iota([1:3]) = [0:1]$.
The space corresponding to $[\lambda : \mu] \in S$ is of type $\circled{4}$. A rational quadruple of $V_{[0:1]}$ is given by $\{1, \tau_{0,1}, \tau_{0,1}-1, \tau_{0,1}-2\}$. The various angles in this quadruple realise all the squared amplitudes $\zeta_{10}, \zeta_{10}^2, \zeta_{10}^4$ and $\zeta_{10}^5$.

\item If $[\lambda:\mu]\not\in S$, the space $V_{[\lambda : \mu]}$ contains precisely two angles with squared amplitude $\zeta_{10}$. These are given by
\[
\zeta_{10} : \quad (1,\tau) \text{ and } (\tau+a,\tau+b),
\]
where
\[
a=\frac{4 \lambda^2 - 6 \lambda \mu + \mu^2}{3 \lambda - \mu}, \quad b = -\frac{5 \lambda (\lambda^2 + \lambda \mu - \mu^2)}{4 \lambda^2 - 6 \lambda \mu + \mu^2}
\]

\item Let $[\lambda : \mu] \not \in S$. The space $V_{[\lambda : \mu]}$ contains an angle with squared amplitude $\zeta_{10}^2$ if and only if $\lambda (3 \lambda-\mu) (11 \lambda^2-9 \lambda \mu+4 \mu^2)$ is the square of a rational number $u$. In this case, setting
\[
c_{\pm} = \frac{-2\mu^2+7\lambda\mu-3\lambda^2 \mp u}{2\mu-4\lambda} , \quad d_{\pm} = \frac{-5\lambda\mu+5\lambda^2 \mp u}{2\lambda}
\]
the rational angles in $V$ with squared amplitude $\zeta_{10}^2$ are precisely $(\tau+c_+,\tau+d_+)$ and $(\tau+c_-,\tau+d_-)$. The curve in weighted projective space
\[
u^2 = \lambda (3 \lambda-\mu) \left(11 \lambda^2-9 \lambda \mu+4 \mu^2\right)
\]
has genus 1 and a rational point $[0:1:0]$, so it defines an elliptic curve. It is in particular isomorphic over $\mathbb{Q}$ to the elliptic curve with Weierstrass equation $y^2 + xy + y = x^3 - x^2 + 2$, whose Mordell-Weil group over $\mathbb{Q}$ is $\mathbb{Z} \oplus \mathbb{Z}/2\mathbb{Z}$. A generator for the free part of the Mordell-Weil group is given by $(4,5)$.

\item Let $[\lambda : \mu] \not \in S$. The space $V_{[\lambda : \mu]}$ contains an angle with squared amplitude $\zeta_{10}^3$ if and only if $41 \lambda^4 - 38 \lambda^3 \mu + 9 \lambda^2 \mu^2 + 8 \lambda \mu^3 - 
 4 \mu^4$ is the square of a rational number $u$. In this case, setting
\[
e_\pm = -\frac{\lambda^2+11 \lambda \mu-6 \mu^2 + u}{2 \lambda-4 \mu},
\quad 
f_\pm = \frac{9 \lambda^2-\lambda \mu-4 \mu^2+u}{4 \lambda+2 \mu}
\]
the rational angles in $V_{[\lambda : \mu]}$ with squared amplitude $\zeta_{10}^3$ are precisely $(\tau+e_+,\tau+f_+)$ and $(\tau+e_-,\tau+f_-)$. The curve in weighted projective space
\[
u^2 = 41 \lambda^4 - 38 \lambda^3 \mu + 9 \lambda^2 \mu^2 + 8 \lambda \mu^3 - 4 \mu^4
\]
has genus 1 and a rational point $[1:1:4]$, so it defines an elliptic curve. It is in particular isomorphic over $\mathbb{Q}$ to the elliptic curve with Weierstrass equation $y^2 + xy + y = x^3 - x^2 + 20x + 22$, whose Mordell-Weil group over $\mathbb{Q}$ is $\mathbb{Z} \oplus \mathbb{Z}/4\mathbb{Z}$. A generator for the free part of the Mordell-Weil group is given by $(79,660)$. A torsion point of order 4 is $(4,10)$.

\item Let $[\lambda : \mu] \not \in S$. The space $V_{[\lambda : \mu]}$ contains an angle with squared amplitude $\zeta_{10}^4$ if and only if $(14 \lambda^2-6 \lambda \mu+\mu^2) (2 \lambda^2-2 \lambda \mu+\mu^2)$ is the square of a rational number $u$. In this case, setting
\[
g_\pm = \frac{6 \lambda^2-4 \lambda \mu-\mu^2 \pm u}{4 \lambda }, \quad h_\pm = \frac{-2 \lambda^2+8 \lambda \mu-3 \mu^2 \pm u}{2 (\mu-\lambda)}
\]
the rational angles in $V_{[\lambda : \mu]}$ with squared amplitude $\zeta_{10}^4$ are precisely $(\tau+g_+,\tau+h_+)$ and $(\tau+g_-,\tau+h_-)$. The curve in weighted projective space
\[
u^2 = (14 \lambda^2-6 \lambda \mu+\mu^2) (2 \lambda^2-2 \lambda \mu+\mu^2)
\]
has genus 1 and a rational point $[0:1:1]$, so it defines an elliptic curve. It is in particular isomorphic over $\mathbb{Q}$ to the elliptic curve with Weierstrass equation $y^2=x^3 + x^2 - 3x - 2$, whose Mordell-Weil group over $\mathbb{Q}$ is $\mathbb{Z} \oplus \mathbb{Z}/2\mathbb{Z}$. A generator for the free part of the Mordell-Weil group is given by $(3,-5)$.

\item Let $[\lambda : \mu] \not \in S$. The space $V_{[\lambda : \mu]}$ contains an angle with squared amplitude $\zeta_{10}^5$ if and only if $31 \lambda^4 - 38 \lambda^3 \mu + 24 \lambda^2 \mu^2 - 7 \lambda \mu^3 + \mu^4$ is the square of a rational number $u$. In this case, setting
\[
i_\pm = \frac{4 \lambda^2-6 \lambda \mu+\mu^2 \mp u}{3 \lambda-\mu}, \quad j_\pm = \frac{4 \lambda^2-6 \lambda \mu+\mu^2 \pm u}{3 \lambda-\mu}
\]
the rational angles in $V_{[\lambda : \mu]}$ with squared amplitude $\zeta_{10}^5$ are precisely $(\tau+i_+,\tau+j_+)$ and $(\tau+i_-,\tau+j_-)$. We have $i_+j_+ = i_-j_- = ab$ (where $a,b$ are as in (4)), so that this case falls into the classification of Theorem~\ref{thm:abcdGeometric}.

\item If such a space has, in addition to the angles of squared amplitude $\zeta_{10}$, rational angles of two different amplitudes, then (up to a nonzero rational number) $\tau$ is one of the following:
\begin{enumerate}
\item $\tau_1 = \frac{1}{2} (\zeta_{10}^3 - 2\zeta_{10}^2 + 2\zeta_{10})$. The corresponding space is of type \circled{3} + \circled{3}, with the two rational triples given by $(1,\tau,\tau-\frac{1}{2})$ (angles of squared amplitudes $\zeta_{10}^{\pm 1}, \zeta_{10}^{\pm 2}, \zeta_{10}^{\pm 3}$) and $(\tau-\frac{5}{8}, \tau-2, \tau - \frac{1}{6})$ (angles of squared amplitudes $\zeta_{10}^{\pm 1}, \zeta_{10}^{\pm 2}, \zeta_{10}^{\pm 3}$).

\item $\tau_2 = \frac{1}{2} (-\zeta_{10}^3 + 2\zeta_{10}^2 + 2\zeta_{10} + 4)$. The corresponding space is of type \circled{3} + \circled{3}, with the two rational triples given by $(1,\tau,\tau-\frac{5}{2})$ (angles of squared amplitudes $\zeta_{10}^{\pm 1}, \zeta_{10}^{\pm 2}, \zeta_{10}^{\pm 3}$) and $(\tau-\frac{5}{8}, \tau-\frac{15}{2}, \tau - 2)$ (angles of squared amplitudes $\zeta_{10}^{\pm 1}, \zeta_{10}^{\pm 2}, \zeta_{10}^{\pm 3}$).

\item $\tau_3 = \frac{1}{3} (\zeta_{10}^3 - 2\zeta_{10}^2 + 3\zeta_{10} + 1)$. The corresponding space is of type \circled{4}, with a rational 4-tuple given by $(1,\tau,\tau-\frac{5}{3}, \tau-\frac{5}{6})$ (angles of squared amplitudes $\zeta_{10}^{\pm 1}, \zeta_{10}^{\pm 2}, \zeta_{10}^{\pm 4}, \zeta_{10}^5$).

\item $\tau_4 = \zeta_{10}+1$. The corresponding space is of type \circled{4}, with a rational 4-tuple given by $(1,\tau,\tau-1, \tau-2)$ (angles of squared amplitudes $\zeta_{10}^{\pm 1}, \zeta_{10}^{\pm 2}, \zeta_{10}^{\pm 4}, \zeta_{10}^5$).

\item $\tau_5 = \zeta_{10}^3 - 2\zeta_{10}^2 + \zeta_{10} - 1$. The corresponding space is of type \circled{3} + \circled{3}, with the two rational triples given by $(1,\tau,\tau+1)$ (angles of squared amplitudes $\zeta_{10}^{\pm 1}, \zeta_{10}^{\pm 3}, \zeta_{10}^{\pm 4}$) and $(\tau-\frac{1}{2}, \tau+\frac{4}{3}, \tau +5)$ (angles of squared amplitudes $\zeta_{10}^{\pm 1}, \zeta_{10}^{\pm 3}, \zeta_{10}^{\pm 4}$).

\item $\tau_6 = 2\zeta_{10}^3 - 4\zeta_{10}^2 + \zeta_{10} - 3$. The corresponding space is of type \circled{3} + \circled{3}, with the two rational triples given by $(1,\tau,\tau+5)$ (angles of squared amplitudes $\zeta_{10}^{\pm 1}, \zeta_{10}^{\pm 3}, \zeta_{10}^{\pm 4}$) and $(\tau-10, \tau+1, \tau + \frac{15}{4})$ (angles of squared amplitudes $\zeta_{10}^{\pm 1}, \zeta_{10}^{\pm 3}, \zeta_{10}^{\pm 4}$).

\end{enumerate}
In addition, we have $\langle 1, \tau_1 \rangle_\mathbb{Q} \sim \langle 1,\tau_2\rangle_\mathbb{Q}$, $\langle 1, \tau_3 \rangle_\mathbb{Q} \sim \langle 1,\tau_4\rangle_\mathbb{Q}$, $\langle 1, \tau_5 \rangle_\mathbb{Q} \sim \langle 1,\tau_6\rangle_\mathbb{Q}$, and these are all the pairs of homothetic spaces in this list. Finally, concerning the comparison with Theorem~\ref{thm:3plus2}, we have 
\[
\langle 1, \tau_1 \rangle_\mathbb{Q} \sim \langle 1,\tau_2\rangle_\mathbb{Q} \sim \langle 1, 2\zeta_5^3 + \zeta_5+2 \rangle_{\mathbb{Q}} \text{ and } \langle 1, \tau_5 \rangle_\mathbb{Q} \sim \langle 1,\tau_6\rangle_\mathbb{Q} \sim \langle 1, \zeta_5^2 + 2\zeta_5+2 \rangle_{\mathbb{Q}}
\]

\end{enumerate}
\end{theorem}

\begin{theorem}\label{thm:102}

For $(\lambda,\mu)\in\Q^2\setminus\{(0,0)\}$ let
\[
\tau_{\lambda,\mu} =\lambda(1+\zeta_{10}^2) +\mu \zeta_{10}.
\]
The following hold:

\begin{enumerate}
\item The space $V=\langle 1,\tau \rangle_{\mathbb{Q}}$ only depends on $[\lambda : \mu] \in \mathbb{P}_1(\mathbb{Q})$. We denote it by $V_{[\lambda,\mu]}$.

\item Let $\tau \in \mathbb{Q}(\zeta_{10})$ generate a space $V$ such that the squared angle between $1$ and $\tau$ is $\zeta_{10}^2$. Then there exist $(\lambda,\mu)\in\Q^2\setminus\{(0,0\}$ such that $\tau=\tau_{\lambda,\mu}$.

\item Conversely, for every $(\lambda,\mu)\in\Q^2\setminus\{(0,0\}$, the squared angle between 1 and $\tau$ is $\zeta_{10}^2$.

\item For every $[\lambda:\mu]\in\mathbb{P}_1(\Q)$, the complex conjugate $\overline{V_{[\lambda,\mu]}}$ is homothetic to $V_{[\lambda,\mu]}$.

\item The map
\[
\begin{array}{cccc}
f : & \mathbb{P}_1(\mathbb{Q}) & \to & \{ \text{homothety classes of spaces} \} \\
& [\lambda : \mu] & \mapsto & \text{class of } V_{[\lambda : \mu]}
\end{array}
\]
is 2-to-1. The involution $\iota : [\lambda : \mu] \mapsto [-\lambda : \lambda +\mu]$ preserves the fibres of $f$. %

\item Let $S= \{ [0:1],[1:0],[1:-1],[2:-1] \}$. Notice that $\iota([0:1]) =[0 : 1]$, $\iota([2:-1])=[2:-1]$ and $\iota([1:0]) = [1:-1]$.
The spaces corresponding to $[\lambda : \mu] \in S$ have the following structure:

\begin{itemize}
\item  $[\lambda : \mu] = [0 : 1]$. The space is of type $\circled{4}$. A rational quadruple of $V_{[0:1]}$ is given by $\{1, \tau_{0,1}, \tau_{0,1}-1, \tau_{0,1}+1\}$. The various angles in this quadruple realise all the squared amplitudes $\zeta_{10}, \zeta_{10}^2, \zeta_{10}^4$ and $\zeta_{10}^5$.
\item $[\lambda : \mu] = [1:0]$ or $[1:-1]$. The space is of type $\circled{4}$. A rational quadruple of $V_{[1:0]}$ is given by $\{1, \tau_{1,0}, \tau_{1,0}-1, \tau_{1,0}-2\}$. The various angles in this quadruple realise all the squared amplitudes $\zeta_{10}^2, \zeta_{10}^3, \zeta_{10}^4$ and $\zeta_{10}^5$.
\item  $[\lambda : \mu] = [2 : -1]$. The space is of type $\circled{2}$, with only one rational angle $(1,\tau_{2,-1})$ of squared amplitude $\zeta_{10}^2$.
\end{itemize}

\item If $[\lambda:\mu]\not\in S$, the space $V_{[\lambda : \mu]}$ contains precisely two angles with squared amplitude $\zeta_{10}^2$.
These are given by
\[
(1,\tau) \text{ and } (\tau+a,\tau+b),
\]
where
\[
a = \frac{-2 \lambda \mu - \mu^2}{\lambda + \mu}, \quad b = \frac{\lambda(\lambda^2 + \lambda \mu - \mu^2)}{\mu (2 \lambda + \mu)}.
\]

\item Let $[\lambda : \mu] \not \in S$. The space $V_{[\lambda : \mu]}$ contains an angle with squared amplitude $\zeta_{10}^4$ if and only if $\lambda( \lambda^3 + 2 \lambda^2 \mu + 5 \lambda \mu^2 + 4 \mu^3)$ is the square of a rational number $u$. 
The curve in weighted projective space
\[
u^2 = \lambda( \lambda^3 + 2 \lambda^2 \mu + 5 \lambda \mu^2 + 4 \mu^3)
\]
has genus 1 and a rational point $(0:1:0)$, so it defines an elliptic curve. It is in particular isomorphic over $\mathbb{Q}$ to the elliptic curve with Weierstrass equation $y^2 + xy + y = x^3 + x^2$, whose Mordell-Weil group over $\mathbb{Q}$ is cyclic of order 4, generated by $(0,0)$. The only four rational points are given by $[\lambda : \mu :u ] = [0 : 1 : 0], [1 : 0 : \pm 1], [1 : -1 : 0]$. The spaces corresponding to these rational points are all of type \circled{4}.

\item Let $[\lambda : \mu] \not \in S$. The space $V_{[\lambda : \mu]}$ contains an angle with squared amplitude $\zeta_{10}^5$ if and only if $\lambda^4 + 2 \lambda^3 \mu + 4 \lambda^2 \mu^2 + 3 \lambda \mu^3 + \mu^4$ is the square of a rational number $u$. In this case, setting
\[
e_\pm = \frac{ -2 \lambda \mu-\mu^2 \mp u}{\lambda+\mu}, \quad f_\pm =  \frac{-2 \lambda \mu-\mu^2 \pm u}{\lambda+\mu}
\]
the rational angles in $V_{[\lambda : \mu]}$ with squared amplitude $\zeta_{10}^5$ are precisely $(\tau+e_+,\tau+f_+)$ and $(\tau+e_-,\tau+f_-)$. We have $e_+f_+=e_-f_-=ab$ (where $a,b$ are as in (4)), so that this case falls into the classification of Theorem~\ref{thm:abcdGeometric}.

\item No such space has both an angle with squared amplitude $\zeta_{10}^4$ and an angle with squared amplitude $\zeta_{10}^5$.
\end{enumerate}
\end{theorem}

\begin{remark}\label{rmk:nequals5}
If $(a, b, c, d, x, y, z)$ is a solution to Equation \eqref{eqn:main} with $ab \neq cd$ and $x, y, z$ having common order $5$, then (up to symmetry equivalence) we may assume $x = \zeta_5 = \zeta_{10}^2$, and therefore the corresponding value of $\tau \in \mathbb{Q}(\zeta_5)=\mathbb{Q}(\zeta_{10})$ generates a space of the type considered in Theorem~\ref{thm:102}. In order for this space to have other rational angles besides those of squared amplitude $\zeta_{10}^2$, it needs to contain an angle with squared amplitude $\zeta_{10}^4$. The theorem implies that the corresponding space is of type \circled{4}.
\end{remark}

\subsection{\texorpdfstring{$n=12$}{n=12}}\label{subsect:12}
Consider a solution $(a,b,c,d,x,y,z)$ to Equation \eqref{eqn:main} with $ab \neq cd$ and $x,y,z$ having common order $n=12$. In this case, Equation \eqref{eq:tauabcd} shows that $\tau$ belongs to $\mathbb{Q}(\zeta_{12})$. One of the following necessarily holds: one among $x,y,z$ is a primitive $12$th root of unity (in which case, up to the Galois action and permutations, we may assume that $x=\zeta_{12}$), or one of the three roots of unity has order 4 (in which case, again up to the Galois action and permutations, we may assume $x=\zeta_{12}^3$). Notice that, if neither of these conditions holds, then all the three roots of unity $x,y,z$ have order dividing 6, contradicting the fact that their common order is $12$. Finally, in the case when $x$ is $\zeta_{12}^3$ and none among $x,y,z$ is a primitive $12$th root of unity, then at least one between $y$ and $z$ needs to be $\zeta_{12}^{\pm 2}$ or $\zeta_{12}^{\pm 4}$ in order to ensure that the common order of $x,y,z$ is exactly $12$.

We study the case $x=\zeta_{12}$ in Theorem~\ref{thm:121} and the case $x=\zeta_{12}^3$ (with neither $y$ nor $z$ primitive) in Theorem~\ref{thm:123}.

\begin{theorem}\label{thm:121}

For $(\lambda,\mu) \in \mathbb{Q}^2 \setminus \{(0,0)\}$, let
\[
\tau_{\lambda, \mu} :=\lambda(\zeta_{12}+ \zeta_{12}^2-\zeta_{12}^3) + \mu(1+\zeta_{12}).
\]
The following hold:
\begin{enumerate}
\item The space $\langle 1, \tau_{\lambda, \mu}\rangle_{\mathbb{Q}}$ only depends on $[\lambda : \mu] \in \mathbb{P}_1(\mathbb{Q})$. We denote it by $V_{[\lambda : \mu]}$.
\item Let $\tau \in \mathbb{Q}(\zeta_{12})$ generate a space $V$ such that the squared angle between $1$ and $\tau$ is $\zeta_{12}$. There exists $(\lambda,\mu) \in \mathbb{Q}^2 \setminus \{(0,0)\}$ such that $\tau = \tau_{\lambda, \mu}$.
\item Conversely, for every $(\lambda,\mu) \in \mathbb{Q}^2 \setminus \{(0,0)\}$, the squared angle between $1$ and $\tau$ is $\zeta_{12}$.
\item For every $[\lambda : \mu] \in \mathbb{P}_1(\mathbb{Q})$, the complex conjugate $\overline{V_{[\lambda : \mu]}}$ is homothetic to $V_{[\lambda : \mu]}$.
\item The map
\[
\begin{array}{cccc}
f : & \mathbb{P}_1(\mathbb{Q}) & \to & \{ \text{homothety classes of spaces} \} \\
& [\lambda : \mu] & \mapsto & \text{class of } V_{[\lambda : \mu]}
\end{array}
\]
is 2-to-1. The involution $\iota : [\lambda : \mu] \mapsto [\lambda + \mu :- \mu]$ preserves the fibres of $f$. %

\item Let $S= \{ [0:1],[1:0],[1:1],[-1 : 1],[-1,2],[-2,1] \}$. Notice that $\iota([0:1]) =[-1 : 1]$, $\iota([1:0]) =[1 : 0]$,$\iota([1:1]) =[-2 : 1]$, and $\iota([-1:2]) = [-1:2]$.
The spaces corresponding to $[\lambda : \mu] \in S$ have the following structure:
\begin{itemize}
\item  $[\lambda : \mu] = [0 : 1]$ or $[-1 : 1]$. The space is of type $\circled{4}$. A rational quadruple of $V_{[0:1]}$ is given by $\{1, \tau_{0,1}, \tau_{0,1}-1, \tau_{0,1}-2\}$. The various angles in this quadruple realise all the squared amplitudes $\zeta_{12}, \zeta_{12}^2, \zeta_{12}^5$ and $\zeta_{12}^6$

\item  $[\lambda : \mu] = [1:0]$. The space is of type $\circled{4}$. A rational quadruple of $V_{[1:0]}$ is given by $\{1, \tau_{1,0}, \tau_{1,0}-1, \tau_{1,0}-2\}$. The various angles in this quadruple realise all the squared amplitudes $\zeta_{12}, \zeta_{12}^2, \zeta_{12}^3$ and $\zeta_{12}^5$

\item  $[\lambda : \mu] = [1 : 1]$ or $[2 : -1]$. The space is of type $2 \cdot \circled{3}$. Two rational triples of $V_{[1:1]}$ are given by $\{1, \tau_{1,1}, \tau_{1,1}-3,\}$ and $\{\tau_{1,1}+1, \tau_{1,1}-4, \tau_{1,1}-3/2,\}$. In both triples the various angles realise all the squared amplitudes $\zeta_{12}, \zeta_{12}^4, \zeta_{12}^5$

\item  $[\lambda : \mu] = [-1 : 2]$. The space is of type $2 \cdot \circled{3}$. Two rational triples of $V_{[-1:2]}$ are given by $\{1, \tau_{-1,2}, \tau_{-1,2}+1,\}$ and $\{\tau_{-1,2}-3, \tau_{-1,2}+3/4, \tau_{-1,2}+2,\}$. In both triples the various angles realise all the squared amplitudes $\zeta_{12}, \zeta_{12}^4, \zeta_{12}^5$.

\end{itemize}

\item If $[\lambda:\mu]\not\in S$, the space $V_{[\lambda : \mu]}$ contains precisely two angles with squared amplitude $\zeta_{12}$
 and two angles with squared amplitude $\zeta_{12}^5$. 
These are given by %
\[
\zeta_{12} : \quad (1,\tau) \text{ and } (\tau+a,\tau+b),
\]
where
\[
a = - \mu \frac{2 \lambda +\mu}{\lambda+\mu}, \quad b=\lambda  \frac{2 \lambda^2+2 \lambda \mu- \mu^2}{\mu (2 \lambda+\mu)}
\]
and by
\[
\zeta_{12}^5 : \quad (\tau+c,\tau+d) \text{ and } (\tau+c', \tau+d')
\]
where
\[
c =-\lambda-2\mu, \quad d = \frac{-2 \lambda^2-2 \lambda \mu+\mu^2}{\lambda-\mu}
\]
\[
c' = \frac{-2\lambda^2-2\lambda \mu+\mu^2}{2 \lambda}, \quad d' = -2 (\lambda+\mu)
\]

\item Let $[\lambda : \mu] \not \in S$. The space $V_{[\lambda : \mu]}$ contains an angle with squared amplitude $\zeta_{12}^2$ if and only if $\lambda (\lambda+\mu) \left(\lambda^2+\lambda \mu+\mu^2\right)$ is the square of a rational number $u$. The curve in weighted projective space
\[
u^2 = \lambda (\lambda+\mu) \left(\lambda^2+\lambda \mu+\mu^2\right)
\]
has genus 1 and a rational point $[0:1:0]$, so it defines an elliptic curve. It is in particular isomorphic over $\mathbb{Q}$ to the elliptic curve with Weierstrass equation $y^2=1+2x+2x^2+x^3$, whose Mordell-Weil group over $\mathbb{Q}$ is $\mathbb{Z}/4\mathbb{Z}$. The only four rational points are given by $[\lambda : \mu : u] = [1 : 0 : \pm 1], [0 : 1 : 0], [-1 : 1 : 0]$ and they correspond to the three lattices of type \circled{4} shown below at part (9).

\item Let $[\lambda : \mu] \not \in S$. The space $V_{[\lambda : \mu]}$ contains an angle with squared amplitude $\zeta_{12}^3$ if and only if $(2 \lambda^2+2 \lambda \mu+\mu^2) (2 \lambda^2+2 \lambda \mu+5 \mu^2)$ is the square of a rational number $u$. The curve in weighted projective space
\[
u^2 = (2 \lambda^2+2 \lambda \mu+\mu^2) (2 \lambda^2+2 \lambda \mu+5 \mu^2)
\]
has genus 1 and a rational point $[1:0:2]$, so it defines an elliptic curve. It is in particular isomorphic over $\mathbb{Q}$ to the elliptic curve with Weierstrass equation $y^2 + 2xy = x^3 - 6x^2 + 4x$, whose Mordell-Weil group over $\mathbb{Q}$ is $(\mathbb{Z}/2\mathbb{Z})^2$. The only four rational points are given by $[\lambda : \mu : u] = [1 : 0 : \pm 2], [1 : -2 : \pm 6]$ 

\item Let $[\lambda : \mu] \not \in S$. The space $V_{[\lambda : \mu]}$ contains an angle with squared amplitude $\zeta_{12}^4$ if and only if $(\lambda^2+\lambda \mu+\mu^2)(5 \lambda^2+5 \lambda \mu+2 \mu^2)$ is the square of a rational number $u$. In this case, setting
\[
e_\pm =\frac{\lambda^2-2 \lambda \mu-2 \mu^2 \pm u}{\lambda+2 \mu}, \quad f_\pm = -\frac{\lambda^2+4 \lambda \mu+\mu^2 \pm u}{2 \lambda+\mu}
\]
the rational angles in $V_{[\lambda : \mu]}$ with squared amplitude $\zeta_{12}^6$ are precisely $(\tau+e_+,\tau+f_+)$ and $(\tau+e_-,\tau+f_-)$.
The curve in weighted projective space
\[
u^2 = \left(\lambda^2+\lambda \mu+\mu^2\right) \left(5 \lambda^2+5 \lambda \mu+2 \mu^2\right)
\]
has genus 1 and a rational point $[1:1:6]$, so it defines an elliptic curve. It is in particular isomorphic over $\mathbb{Q}$ to the elliptic curve with Weierstrass equation $y^2 = x^3 - 18x - 27$, whose Mordell-Weil group over $\mathbb{Q}$ is $\mathbb{Z}/2\mathbb{Z} \oplus \mathbb{Z}$. A generator for the free part of the Mordell-Weil group is given by $(-2,1)$.

\item Let $[\lambda : \mu] \not \in S$. The space $V_{[\lambda : \mu]}$ contains an angle with squared amplitude $\zeta_{12}^6$ if and only if $(\lambda^2 + \lambda \mu + \mu^2) (2 \lambda^2 + 2 \lambda \mu + \mu^2)$ is the square of a rational number $u$. In this case, setting
\[
g_\pm =-\frac{ \mu(2 \lambda +\mu) \pm u}{\lambda+\mu}, \quad h_\pm = \frac{-\mu (2 \lambda+\mu) \pm u}{\lambda+\mu}
\]
the rational angles in $V$ with squared amplitude $\zeta_{12}^6$ are precisely $(\tau+g_+,\tau+h_+)$ and $(\tau+g_-,\tau+h_-)$.
The curve in weighted projective space
\[
u^2 =(\lambda^2 + \lambda \mu + \mu^2) (2 \lambda^2 + 2 \lambda \mu + \mu^2)
\]
has genus 1 and a rational point $[1:-1:1]$, so it defines an elliptic curve. It is in particular isomorphic over $\mathbb{Q}$ to the elliptic curve with Weierstrass equation $y^2 = x^3 - x^2 - 4x - 2$, whose Mordell-Weil group over $\mathbb{Q}$ is $\mathbb{Z}/2\mathbb{Z} \oplus \mathbb{Z}$. A generator for the free part of the Mordell-Weil group is given by $\left(-\frac{3}{4},\frac{1}{8} \right)$.

\item If such a space has, in addition to the angles of squared amplitude $\zeta_{12}$ and $\zeta_{12}^5$, rational angles of two different amplitudes, then (up to a nonzero rational number) $\tau$ is one of the following:
\begin{enumerate}
\item $\tau_1 = \zeta_{12}^3-\zeta_{12}^2+1$. The corresponding space is of type \circled{4}, with a rational 4-tuple given by $(1,\tau,\tau-1,\tau-\frac{1}{2})$ (angles of squared amplitudes $\zeta_{12}^{\pm 1}, \zeta_{12}^{\pm 2}, \zeta_{12}^{\pm 5}, \zeta_{12}^6$).

\item $\tau_2 = \zeta_{12}+1$. The corresponding space is of type \circled{4}, with a rational 4-tuple given by $(1,\tau,\tau-1,\tau-2)$ (angles of squared amplitudes $\zeta_{12}^{\pm 1}, \zeta_{12}^{\pm 2}, \zeta_{12}^{\pm 5}, \zeta_{12}^6$).

\item $\tau_3 = \frac{1}{2}(\zeta_{12}^3-\zeta_{12}^2+\zeta_{12}+2)$. The corresponding space is of type \circled{4}, with a rational 4-tuple given by $(1,\tau,\tau-1,\tau-\frac{3}{2})$ (angles of squared amplitudes $\zeta_{12}^{\pm 1}, \zeta_{12}^{\pm 3}, \zeta_{12}^{\pm 4}, \zeta_{12}^{\pm 5}, \zeta_{12}^6$).

\end{enumerate}
In addition, we have $\langle 1, \tau_1 \rangle_\mathbb{Q} \sim \langle 1,\tau_2\rangle_\mathbb{Q}$, and this is the only pair of homothetic spaces in this list.

\end{enumerate}
\end{theorem}

\begin{theorem}\label{thm:123}

For $(\lambda,\mu) \in \mathbb{Q}^2 \setminus \{(0,0)\}$, let
\[
\tau_{\lambda, \mu} :=  \lambda(\zeta_{12}+\zeta_{12}^2) + \mu(1+\zeta_{12}^3).
\]
The following hold:
\begin{enumerate}
\item The space $\langle 1, \tau_{\lambda, \mu}\rangle_{\mathbb{Q}}$ only depends on $[\lambda : \mu] \in \mathbb{P}_1(\mathbb{Q})$. We denote it by $V_{[\lambda : \mu]}$.
\item Let $\tau \in \mathbb{Q}(\zeta_{12})$ generate a space $V$ such that the squared angle between $1$ and $\tau$ is $\zeta_{12}^3$. There exists $(\lambda,\mu) \in \mathbb{Q}^2 \setminus \{(0,0)\}$ such that $\tau = \tau_{\lambda, \mu}$.
\item Conversely, for every $(\lambda,\mu) \in \mathbb{Q}^2 \setminus \{(0,0)\}$, the squared angle between $1$ and $\tau$ is $\zeta_{12}^3$.
\item For every $[\lambda : \mu] \in \mathbb{P}_1(\mathbb{Q})$, the complex conjugate $\overline{V_{[\lambda : \mu]}}$ is homothetic to $V_{[\lambda : \mu]}$.
\item The map
\[
\begin{array}{cccc}
f : & \mathbb{P}_1(\mathbb{Q}) & \to & \{ \text{homothety classes of spaces} \} \\
& [\lambda : \mu] & \mapsto & \text{class of } V_{[\lambda : \mu]}
\end{array}
\]
is 2-to-1. The involution $\iota : [\lambda : \mu] \mapsto [-\lambda : \lambda+ \mu]$ preserves the fibres of $f$. %

\item Let $S= \{ [-2,1] \}$. Notice that $\iota([-2:1]) =[-2 : 1]$.
The spaces corresponding to $[\lambda : \mu] \in S$ has type $\circled{2}$ with a single rational angle given by $(1,\tau_{-2,1})$ with squared amplitude $\zeta_{12}^3$

\item If $[\lambda:\mu]\not\in S$, the space $V_{[\lambda : \mu]}$  contains precisely two angles with squared amplitude $\zeta_{12}^3$.
These are given by
\[
(1,\tau) \text{ and } (\tau+a,\tau+b),
\]
where
\[
a = -\lambda - 2 \mu, \quad b = \frac{\lambda^2 - 2 \lambda \mu - 2 \mu^2}{\lambda + 2 \mu}.
\]

\item Let $[\lambda : \mu] \not \in S$. The space $V_{[\lambda : \mu]}$ contains an angle with squared amplitude $\zeta_{12}^2$ if and only if $(\lambda^2 - \lambda \mu - \mu^2) (\lambda^2 + \lambda \mu + \mu^2)$ is the square of a rational number $u$. In this case, setting
\[
c_\pm = -\frac{2 \lambda \mu+\mu^2 \pm u}{\lambda+\mu}, \quad d_\pm= \frac{\lambda^2-\mu^2 \mp u}{\mu}
\]
the rational angles in $V$ with squared amplitude $\zeta_{12}^2$ are precisely $(\tau+c_+,\tau+d_+)$ and $(\tau+c_-,\tau+d_-)$.
The curve in weighted projective space
\[
u^2 = (\lambda^2 - \lambda \mu - \mu^2) (\lambda^2 + \lambda \mu + \mu^2)
\]
has genus 1 and a rational point $(1:0:1)$, so it defines an elliptic curve. It is in particular isomorphic over $\mathbb{Q}$ to the elliptic curve with Weierstrass equation $y^2 = x^3 - x^2 + 4x$, whose Mordell-Weil group over $\mathbb{Q}$ is $\mathbb{Z}/2\mathbb{Z} \oplus \mathbb{Z}$. A generator for the free part of the Mordell-Weil group is given by $(1,2)$.

\item Let $[\lambda : \mu] \not \in S$. The space $V_{[\lambda : \mu]}$ contains an angle with squared amplitude $\zeta_{12}^4$ if and only if $(5 \lambda^2 - \lambda \mu - \mu^2) (\lambda^2 + \lambda \mu + \mu^2)$ is the square of a rational number $u$. In this case, setting
\[
e_\pm = -\frac{\lambda^2+4 \lambda \mu+\mu^2 \pm u}{2 \lambda+\mu}, \quad f_\pm = \frac{-2 \lambda^2-2 \lambda \mu+\mu^2 \pm u}{\lambda-\mu}
\]
the rational angles in $V_{[\lambda : \mu]}$ with squared amplitude $\zeta_{12}^4$ are precisely $(\tau+e_+,\tau+f_+)$ and $(\tau+e_-,\tau+f_-)$.
The curve in weighted projective space
\[
u^2 = (5 \lambda^2 - \lambda \mu - \mu^2) (\lambda^2 + \lambda \mu + \mu^2)
\]
has genus 1 and a rational point $(1:1:3)$, so it defines an elliptic curve. It is in particular isomorphic over $\mathbb{Q}$ to the elliptic curve with Weierstrass equation $y^2 = x^3 + 9x + 54$, whose Mordell-Weil group over $\mathbb{Q}$ is $\mathbb{Z}/2\mathbb{Z} \oplus \mathbb{Z}$. A generator for the free part of the Mordell-Weil group is given by $(6,18)$.

\item The space $V_{[\lambda : \mu]}$ does not contain any angles with squared amplitude $\zeta_{12}^6$.
Indeed, this happens if and only if $2(\lambda^2+\lambda\mu+\mu^2)$ is the square of a rational number $u$, but %
the genus-0 curve $u^2=2(\lambda^2+\lambda\mu+\mu^2)$ in $\mathbb{P}_{2,\mathbb{Q}}$ has no rational points (since it has no $\mathbb{Q}_2$- and no $\mathbb{Q}_3$-points).

\item No such space has both an angle with squared amplitude $\zeta_{12}^2$ and an angle with squared amplitude $\zeta_{12}^4$.

\item If such a space has, in addition to the angles of squared amplitude $\zeta_{12}^3$, rational angles of two different amplitudes, then:
\begin{enumerate}
\item $[\lambda:\mu]=[1,0]$ or $[1:-1]$, which correspond to a space of type \circled{4}, with a rational 4-tuple given by $(1,\tau_{1,0},\tau_{1,0}-1,\tau_{1,0}+1)$ (angles of squared amplitudes $\zeta_{12}^{\pm 1}, \zeta_{12}^{\pm 2}, \zeta_{12}^{\pm 3}, \zeta_{12}^{\pm 5}$).

\item $[\lambda:\mu]=[1,1]$ or $[-1:2]$, which correspond to a space of type \circled{4}, with a rational 4-tuple given by $(1,\tau_{1,1},\tau_{1,1}-1,\tau_{1,1}-3)$ (angles of squared amplitudes $\zeta_{12}^{\pm 1}, \zeta_{12}^{\pm 3}, \zeta_{12}^{\pm 4}, \zeta_{12}^{\pm 5}$).

\end{enumerate}

\end{enumerate}
\end{theorem}

\begin{remark}
If $E \to \mathbb{P}^1$ is one of the elliptic families, and if $\iota : \mathbb{P}^1 \to \mathbb{P}^1$ is the involution (coming from complex conjugation) which exchanges homothetic spaces, a geometric interpretation can be given to the action of $\iota$ on $E$:
\begin{enumerate}
\item $\iota$ lifts to an involution of $E$; %
\item every involution of an elliptic curve is either of the form $x \mapsto Q-x$ or of the form $x \mapsto x+T_2$ with $T_2$ a 2-torsion point;
\item the lift must preserve the branch locus of $E \to \mathbb{P}^1$, hence it sends the $0$ of $E$ to a ramification point of $E \to \mathbb{P}^1$, that is, a 2-torsion point;
\item hence, if it is of the form $x \mapsto Q-x$, then $Q$ is a 2-torsion point;
\item given a lift $\varphi$ of $\iota$, also $-\varphi$ is a lift, so there is always a lift of the form $x \mapsto x+T_2$;
\item in particular, if $E$ admits a unique non-trivial 2-torsion point, then translation by that point induces $\iota$ on $\mathbb{P}^1$.
\end{enumerate}
\end{remark}

\section{Computational details}\label{sec:comput.details}

\subsection{Enumerating triples of roots of unity}
In carrying out several of the computational arguments outlined above we are faced with the problem of enumerating triples $(x,y,z)$ of roots of unity with a given common order, up to the equivalence relation given by Definition~\ref{def:se}. To make this process more efficient, we rely on the following observations:
\begin{lemma}\label{lemma: reducing triples}
Let $(x,y,z)$ be a triple with common order $n$. Write $x=\zeta_n^{e_1}$, $y=\zeta_n^{e_2}$ and $z=\zeta_n^{e_3}$ with $e_1, e_2, e_3 \in [1,n]$, and identify $(x,y,z)$ with $(e_1,e_2,e_3)$.
\begin{enumerate}
\item Up to the Galois action, $(e_1,e_2,e_3)$ is equivalent to a triple with $e_1 \mid n$;
\item Up to the action of the group $G \times \operatorname{Gal}(\overline{\mathbb{Q}}/\mathbb{Q})$ of Definition~\ref{def:se}, $(e_1,e_2,e_3)$ is equivalent to a triple with $e_1 \mid n$, $(e_2,n) \geq e_1$, $(e_3,n) \geq e_1$, $e_2, e_3 \leq \frac{n}{2}$, and $e_3 \geq e_2$. 
\end{enumerate}
\end{lemma}
\begin{proof}
Recall that the Galois group of $\mathbb{Q}(\zeta_n)/\mathbb{Q}$ can be canonically identified with $\left( \mathbb{Z}/n\mathbb{Z} \right)^\times$: we denote by $\sigma_i$ the Galois automorphism sending $\zeta_n$ to $\zeta_n^i$. 
Under the identification $(x,y,z) \leftrightarrow (e_1,e_2,e_3)$, the Galois action translates to $\sigma_i(e_1,e_2,e_3)=(ie_1,ie_2,ie_3)$. It is an elementary arithmetic fact that any two elements $a,b$ in $\mathbb{Z}/n\mathbb{Z}$ such that $(n,a)=(n,b)$ differ by multiplication by an element in $\left( \mathbb{Z}/n\mathbb{Z} \right)^\times$. This shows (1). As for (2), observe that up to the action $(e_1,e_2,e_3) \mapsto (\pm e_1 \bmod n, \pm e_2 \bmod n, \pm e_3 \bmod n)$ we can assume that each $e_i$ lies in the interval $[0,\frac{n}{2}]$. 
Using the permutation action we can reorder $(e_1,e_2,e_3)$ so that $(n,e_1) \geq (n, e_2), (n, e_3)$. The Galois action can then be used as in (1) to get an equivalent triple in which $e_1 \mid n$. Finally, reordering again if necessary we can assume $e_3 \geq e_2$.
\end{proof}

\begin{remark}\label{rem:roots_of_unity.degree_at_most_2}
A further reduction in the number of triples to be considered comes from the following observation. Fix a root of unity among $x, y, z$, say $x$ for simplicity. We can consider the polynomial $P$ of equation \eqref{eqn:main} as a polynomial of degree $2$ in $x$. Writing $P=x^2 P_2(a,b,c,d,y,z) + x P_1(a,b,c,d,y,z) + P_0(a,b,c,d,y,z)$, we see that $x$ lies in an at most quadratic extension of $\Q(y,z)$, unless we have $P_2(a,b,c,d,y,z)=P_1(a,b,c,d,y,z)=P_0(a,b,c,d,y,z)=0$. One checks that these equations can only be satisfied by a space of type \circled{4}, hence we can further assume that $\Q(x,y,z)$ has degree at most $2$ over $\Q(y,z)$. The same holds for the degree of $\Q(x,y,z)$ over its subfields $\Q(x,z)$ and $\Q(x,y)$.
\end{remark}

\begin{remark}
Working with the reduced set of triples suggested by Lemma \ref{lemma: reducing triples} and Remark~\ref{rem:roots_of_unity.degree_at_most_2} leads to a tremendous speed-up in practice: for $n=144$ (the largest value we need to consider in the computational proof of Theorem~\ref{thm:Enumerate222}), the reduced set of triples has size 1869, %
less than 0.1\% of the $2,515,968$ triples that satisfy $(e_1,e_2,e_3,144)=1$.
\end{remark}

\subsection{Finding defining equations over \texorpdfstring{$\Q$}{Q}}

Let $K$ be a finite extension of $\mathbb{Q}$ and let $p_1, \ldots, p_r \in K[x_1,\ldots,x_n]$ be polynomials with coefficients in $K$. Suppose we are interested in finding all solutions to the equations
\begin{equation}\label{eq:SchemeOverK}
p_1(x_1,\ldots,x_n) = \cdots = p_r(x_1,\ldots,x_n) = 0
\end{equation}
where $x_1,\ldots,x_n$ are \textit{rational} numbers. Fix a basis $\beta_1,\ldots,\beta_d$ of $K$ as $\mathbb{Q}$-vector space, where $d=[K:\mathbb{Q}]$. We can write each $p_i$ as $p_i(x_1,\ldots,x_n) = \sum_{j=1}^d \beta_j p_{ij}(x_1,\ldots,x_n)$, where each $p_{ij}$ now has rational coefficients. Any solution to \eqref{eq:SchemeOverK} with $x_1,\ldots,x_n \in \mathbb{Q}$ then corresponds to a rational solution of the system of equations
\begin{equation}\label{eq:SchemeOverQ}
p_{ij}(x_1,\ldots,x_n)=0 \quad \forall i=1,\ldots,r, \; j=1,\ldots,d,
\end{equation}
which is now given by equations with rational coefficients. In order to efficiently convert equations as in \eqref{eq:SchemeOverK} to equations as in \eqref{eq:SchemeOverQ}, we use the following simple observation:
\begin{lemma}
With the previous notation, the ideal of $\mathbb{Q}[x_1,\ldots,x_n]$ generated by the polynomials $p_{ij}$ is also generated by the polynomials $\operatorname{tr}_{K/\mathbb{Q}}(p_i \beta_j)$ for $i=1,\ldots,r$ and $j=1,\ldots,d$. Here the \textit{trace} of a polynomial is taken coefficient-wise.
\end{lemma}
\begin{proof}
As $K/\mathbb{Q}$ is a separable extension, it is well-known that the trace form $K \times K \to \mathbb{Q}$ is nondegenerate. Let $\beta_i^\vee$ be the dual basis of $\beta_i$ with respect to the trace form. By definition, $p_{ij}$ is then $\operatorname{tr}_{K/\mathbb{Q}}(p_i \beta_j^\vee)$. On the other hand, as $\{\beta_i\}$ and $\{\beta_i^\vee\}$ are two bases for the $\mathbb{Q}$-vector space $K$, there is an invertible matrix $M_{jk}$ with rational coefficients such that $\sum_k M_{jk}\beta_k^\vee = \beta_j$. This easily implies that the $\mathbb{Q}$-span of the polynomials $\operatorname{tr}_{K/\mathbb{Q}}(p_i \beta_j^\vee)$ coincides with the $\mathbb{Q}$-span of the polynomials $\operatorname{tr}_{K/\mathbb{Q}}(p_i \beta_j)$, which in turn implies the lemma.
\end{proof}

\subsection{Rational points on curves}

The higher-genus curves we encounter in Section~\ref{sect:222} all have genus 3 or 5, as proved in Proposition~\ref{prop:HigherGenusCurves}. In order to determine their rational points, however, we often reduce to studying certain auxiliary curves of genus 2, so we begin with this case. Given a (smooth) projective curve $C/\mathbb{Q}$ of genus 2 with Jacobian $J/\mathbb{Q}$, in order to determine $C(\mathbb{Q})$ we proceed as follows:
\begin{enumerate}
\item Since every such curve is hyperelliptic over $\mathbb{Q}$, we fix an equation for $C$ of the form $y^2=f(x)$ with $f(x) \in \mathbb{Q}[x]$ of degree 5 or 6 and having no repeated factors (by a slight abuse of notation we identify this affine curve with the unique smooth projective curve having the same function field). Up to rescaling, we can in fact assume that $f(x)$ has integral coefficients.
\item We preliminary run a quick test for local solubility: we loop over primes $3 \leq p \leq 13$ that do not divide the discriminant of $f(x)$ and check whether the (projective) curve $C_p$ obtained from $C$ by reduction modulo $p$ has any $\mathbb{F}_p$-rational points. Since we have a well-defined reduction map $C(\mathbb{Q}) \to C_p(\mathbb{F}_p)$, the fact that $C_p(\mathbb{F}_p)=\emptyset$ implies $C(\mathbb{Q})=\emptyset$. To explain why we stop at $p=13$, notice that for $p>13$ (and not dividing the discriminant of $f(x)$) the Weil bounds imply $\#C_p(\mathbb{F}_p) \geq p+1-4\sqrt{p} > 0$, that is, the reduced curve automatically has $\mathbb{F}_p$-points.
\item If $C$ passes this preliminary test, we use standard descent techniques to prove that the Mordell-Weil group $J(\mathbb{Q})$ has rank at most 1. If this is the case, then Chabauty's method (which is fully implemented in MAGMA for genus 2 curves with $\operatorname{rk} J(\mathbb{Q}) \leq 1$) can be used to provably determine $C(\mathbb{Q})$.
\end{enumerate}
All the genus-2 curves we meet can be handled by this combination of techniques.

\begin{remark}[`Modular' interpretation of the automorphisms of $C$]
The curves $C$ we studied in Section~\ref{sect:222} all possess three independent involutions. Two are given by $u \mapsto -u$ and $v \mapsto -v$; the third one is different in each case, but always of the form $[\lambda : \mu : u : v] \mapsto [ \alpha \lambda + \beta \mu, \gamma \lambda + \delta \mu, h u, ku ]$ for a certain $\begin{pmatrix}
\alpha & \beta \\
\gamma & \delta
\end{pmatrix}$ in $\operatorname{GL}_2(\mathbb{Q})$ and $h,k \in \mathbb{Q}^\times$. This fact is not an accident, and is implicitly used in the determination of the rational points of these curves (indeed, this extra involution induces an involution on the quotients of $C$, which is then exploited to obtain $C(\mathbb{Q})$), so we take some time to explain its geometric origin.

To simplify the notation, we consider the special case $n=10$, where we look at angles of squared amplitude $\zeta_{10}^2$ and $\zeta_{10}^3$ (but the argument applies unchanged to all other cases). Consider, therefore, the curve
\[
\begin{cases}
u^2 = f_2 (\lambda, \mu) = \lambda (3\lambda-\mu)(11\lambda^2-9\lambda\mu+4\mu^2) \\
v^2 = f_3(\lambda, \mu) = 41\lambda^4 - 38\lambda^3\mu + 9\lambda^2\mu^2 + 8\lambda\mu^3 - 4\mu^4.
\end{cases}
\]
A rational point $[\lambda : \mu : u : v ]$ of $C$ has the following `modular' interpretation: it gives rise to $\tau =\lambda(-2-2\zeta_{10}^2 + \zeta_{10}^3) + \mu(1+\zeta_{10})$, and to rational numbers
\[
a = \frac{4\lambda^2-6\lambda \mu+\mu^2}{3\lambda-\mu}, \quad b = -\frac{5\lambda(\lambda^2+\lambda\mu-\mu^2)}{4\lambda^2-6\lambda \mu + \mu^2},
\]
\[
a'_{\pm} = \frac{-3\lambda^2+7\lambda \mu - 2\mu^2 \mp u}{-4\lambda + 2\mu}, \quad b'_{\pm} = \frac{5\lambda^2 -5\lambda \mu \mp u}{2\lambda}
\]
and
\[
a''_{\pm} = \frac{\lambda^2+11\lambda\mu-6\mu^2 \mp u}{-2\lambda + 4\mu}, \quad b''_{\pm} = \frac{9\lambda^2 -\lambda \mu -4\mu^2 \mp u}{4\lambda+2\mu},
\]
such that in the space $\langle 1,\tau \rangle_{\mathbb{Q}}$ there are six distinguished pairs of rational angles, namely those given by the rays
\[
(1,\tau), (\tau+a, \tau+b); \quad (\tau + a'_{\pm}, \tau+b'_{\pm}); \quad (\tau+ a''_{\pm}, \tau+b''_{\pm}).
\]
The automorphism $u \mapsto -u$ exchanges $a'_+$ with $a'_-$ and $b'_+$ with $b'_-$, thus exchanging the two angles in the above list with squared amplitude $\zeta_{10}^2$: the `modular' interpretation of this automorphism is the fact that there is no preferred way to order these two angles, so the labels $\pm$ are arbitrary, and can be exchanged freely. Similarly, $v \mapsto -v$ simply exchanges $a''_+ \leftrightarrow a''_-$ and $b''_+ \leftrightarrow b''_-$, and therefore swaps the two angles with amplitude $\zeta_{10}^3$.

Finally, there is a \textit{third} geometric construction giving rise to an automorphism of $C$. Indeed, if $V$ is a space having angles of squared amplitudes $\zeta_{10}, \zeta_{10}^2, \zeta_{10}^3$, the same is true for the complex conjugate $\overline{V}$. If we let $\sigma_i$ (for $(i,10)=1$) be the automorphism of $\mathbb{Q}(\zeta_{10})$ sending $\zeta_{10} \mapsto \zeta_{10}^i$, then we have 
\[
\overline{V} = \langle 1, \overline{\tau} \rangle_\mathbb{Q} \sim \langle \frac{1}{\overline{\tau}}, 1 \rangle_{\mathbb{Q}} = \left\langle \frac{\sigma_1(\tau)\sigma_5(\tau)\sigma_7(\tau)}{\sigma_1(\tau)\sigma_5(\tau)\sigma_7(\tau)\sigma_{11}(\tau)}, 1 \right\rangle_{\mathbb{Q}}= \langle \sigma_1(\tau)\sigma_5(\tau)\sigma_7(\tau), 1 \rangle_{\mathbb{Q}},
\]
where the last equality follows from the fact that $\sigma_1(\tau)\sigma_5(\tau)\sigma_7(\tau)\sigma_{11}(\tau) = N_{\mathbb{Q}(\zeta_{10})/\mathbb{Q}}(\tau)$ is a nonzero rational number. Notice that $1/\overline{\tau}$ is in $\mathbb{Q}(\zeta_{10})$ and forms with the positive real axis the same angle that $\tau$ forms with $1$, so $1/\overline{\tau}$ is again of the form considered in Theorem~\ref{thm:101}. In particular, it is a linear combination with rational coefficients of $-2 -2\zeta_{10}^2+\zeta_{10}^3$ and $1+\zeta_{10}$.
A short calculation shows that
\[
\sigma_1(\tau)\sigma_5(\tau)\sigma_7(\tau) = (\lambda^2 + \lambda \mu - \mu^2) \cdot \left( (4\lambda - 3 \mu)(1+\zeta_{10}) + ( -\mu+3\lambda )(-2 -2\zeta_{10}^2+\zeta_{10}^3) \right),
\]
so that $\overline{V} \sim \langle 1, ( -\mu+3\lambda )(-2 -2\zeta_{10}^2+\zeta_{10}^3) + (4\lambda - 3 \mu)(1+\zeta_{10}) \rangle_{\mathbb{Q}}$. Of course one can also apply the geometric transformations performed above (the symmetry corresponding to complex conjugation and the various homotheties) to the lines in $V$ forming rational angles; in this way, from the original rational point $[\lambda : \mu : u : v]$ we obtain a new point, which can be seen to be $[ 3\lambda -\mu : 4\lambda - 3 \mu : 5u : 5v]$. We note that the signs of $5u, 5v$ in this formula are somewhat arbitrary, since they depend on the precise labelling of the angles in the transformed lattice, but any choice of signs would work, since $u \mapsto -u$ and $v \mapsto -v$ are also automorphisms of $C$.

Although this argument formally applies only to the \textit{rational points} of $C$, all we have done is algebraic manipulation, so it is not hard to see that this is indeed an automorphism of $C$ (of course, this can also be checked directly on the equations). Finally, the geometric interpretation given above (or the explicit formula) makes it clear that the automorphism constructed in this way commutes with $u \mapsto -u$ and with $v \mapsto -v$. This proves that all the curves $C$ we have to consider admit an action of $(\mathbb{Z}/2\mathbb{Z})^3$.
\end{remark}

All the higher-genus curves we encounter in Section~\ref{sect:222} are of one of the two types considered in Proposition~\ref{prop:HigherGenusCurves}, hence we only explain how we find all the rational points on such curves. Let $C$ be such a curve, and -- in case it is singular of genus 3 -- let $f : \tilde{C} \to C$ be its desingularisation. It is clear that we have $C(\mathbb{Q})=f(\tilde{C}(\mathbb{Q})) \cup C_{\operatorname{sing}}(\mathbb{Q})$, where $C_{\operatorname{sing}}$ is the 0-dimensional singular locus of $C$. The determination of $C_{\operatorname{sing}}(\mathbb{Q})$ is straightforward, so we can assume to be working with the smooth curve $\tilde{C}$. Furthermore, if $g(C)=5$, Proposition~\ref{prop:HigherGenusCurves} yields the existence of an explicit map $C \to C'$ with $g(C')=3$, and -- if we can determine $C'(\mathbb{Q})$ -- it is a simple matter to obtain from this also the set $C(\mathbb{Q})$. Thus we just need to explain how we determine $C(\mathbb{Q})$ when $C$ is a smooth curve of genus 3 over $\mathbb{Q}$. It should be stressed that there is no known algorithm capable of achieving this for \textit{all} smooth curves of genus 3 over $\mathbb{Q}$: we simply describe a procedure that \textit{happens} to apply to all the curves that we consider in this work.

\begin{enumerate}
\item We determine whether $C$ is hyperelliptic over $\overline{\mathbb{Q}}$ (in other words, we check whether the canonical map $f$ of $C$ is \textit{not} an embedding).
\item Suppose that $f$ is not an embedding. In this case, $f$ realises $C$ as a double cover of a smooth genus-0 curve $\mathcal{L}$. Using standard techniques, we determine whether $\mathcal{L}$ has $\mathbb{Q}$-rational points: if $\mathcal{L}(\mathbb{Q})=\emptyset$, then clearly also $C(\mathbb{Q})=\emptyset$. Otherwise, using a rational point on $\mathcal{L}$ we can identify $\mathcal{L}$ with $\mathbb{P}_{1,\mathbb{Q}}$, and from this obtain a hyperelliptic model $C_h : y^2=f(x)$ for $C$ over $\mathbb{Q}$ (where, again, with slight abuse of notation we identify this affine curve with the corresponding smooth projective curve).
\item If, on the other hand, the canonical map of $C$ is an embedding, then its image is a plane quartic $Q \subseteq \mathbb{P}_{2,\mathbb{Q}}$.
\item We compute the automorphism group $G$ over $\mathbb{Q}$ of $C_h$ or of $Q$, respectively:
\begin{enumerate}
\item in the hyperelliptic case this is not hard, since every such automorphism is of the form $(x,y) \mapsto ( \frac{ax+b}{cx+d}, \frac{ey}{(cx+d)^3})$ for some $a,b,c,d,e \in \mathbb{Q}$;
\item in the plane quartic case, we exploit the fact that -- since we are working with the canonical embedding -- every automorphism of $Q$ lifts to an automorphism of the ambient $\mathbb{P}_{2,\mathbb{Q}}$ (this follows from the fact that every automorphism of $C$ acts on the three-dimensional vector space of regular differentials by linear automorphisms). Thus we simply need to consider linear transformations of $\mathbb{P}_{2,\mathbb{Q}}$ that preserve -- up to nonzero scalars -- the single quartic equation defining $Q$.
\end{enumerate}
\item For every $\alpha$ in the automorphism group $G$, we consider the quotient map $g_\alpha : C \to C/\langle \alpha \rangle =:C_\alpha$ under the action of the finite group $\langle \alpha \rangle$. This has strictly lower genus than $C$.
\item For every $C_\alpha$ of genus 1, using MAGMA's machinery for elliptic curves
we check whether $C_\alpha(\mathbb{Q})$ is a finite set. If this is the case, pulling back the rational points via $g_\alpha$ we determine $C(\mathbb{Q})$.
\item For every $C_\alpha$ of genus 2, we use the techniques explained above to try and determine $C_\alpha(\mathbb{Q})$. If they succeed, we again obtain the set $C(\mathbb{Q})$ by pulling back $C_\alpha(\mathbb{Q})$ along $g_\alpha$.
\end{enumerate}

It is a (lucky) accident of nature that, for all the curves $C$ we consider in Section~\ref{sect:222}, the above procedure always finds a curve $C_\alpha$ for which the set $C_\alpha(\mathbb{Q})$ is finite and can be determined, and therefore $C(\mathbb{Q})$ can also be determined.

\newpage

\renewcommand{\arraystretch}{3}
\begin{landscape}
\section{Tables}\label{sect:Tables}

\begin{table}
    \begin{tabular}{|c|c|c|c|c|c|c|}
    \hline
    Angle & $u^2=f(\lambda,\mu)$ & $a_\pm, b_\pm$ & $P$ & $E$ & $E(\mathbb{Q})$ & Generators \\ \hline
$\zeta_8^2$ & $\lambda(\lambda+\mu)(\lambda^2+\mu^2)$ & $\displaystyle \frac{\lambda^2-\mu^2 \pm u}{\mu}$, $\displaystyle \frac{-2\lambda\mu \mp u}{\lambda}$ & $[0:1:0]$ & $y^2=(x+1)(x^2+1)$ & $\mathbb{Z}/2\mathbb{Z} \oplus \mathbb{Z}$ & $(-1,0), (1,2)$ \\ 
\hline
$\zeta_8^4$ & $(\lambda^2+\mu^2)(3\lambda^2+2\lambda\mu+\mu^2)$ & $\displaystyle \frac{\lambda^2-2\lambda\mu-\mu^2 \pm u}{\lambda+\mu}$, $\displaystyle \frac{\lambda^2-2\lambda\mu-\mu^2 \mp u}{\lambda + \mu}$ & $[0: 1: 1]$ & $y^2=(x+3)(x^2-2x-7)$ & $\mathbb{Z}/2\mathbb{Z} \oplus \mathbb{Z}$ & $(-3,0), (5,8)$ \\ \hline
    \end{tabular}
    
    \caption{Rational angles in lattices with squared angle $(1,\tau)=\zeta_8$}\label{table:8}
    \end{table}

\begin{table}
    \begin{tabular}{|c|c|c|c|c|c|c|}
    \hline
    Angle & $u^2=f(\lambda,\mu)$ & $a_\pm, b_\pm$ & $P$ & $E$ & $E(\mathbb{Q})$ & Generators \\ \hline
$\zeta_{10} ^ 4$ &
$\mu(\lambda+\mu)(4\lambda^2+\lambda \mu + \mu^2)$ &
$\displaystyle \frac{-2\lambda^2 - \lambda\mu + \mu^2 \pm u}{2\lambda},
\frac{-3\lambda\mu - \mu^2 \mp u}{2\mu}$ &
$[ 0 : 1 : 1 ]$ &
$y^2 + xy + y = x^3 + x^2$ &
$\mathbb{Z}/4\mathbb{Z}$ &
$(0, 0)$ \\ \hline
$\zeta_{10}^5$ &
$\lambda^4 + 3\lambda^3\mu + 4\lambda^2\mu^2 + 2\lambda\mu^3 + \mu^4$ &
$\displaystyle \frac{-\lambda^2 - 2\lambda\mu \pm u}{\lambda + \mu}, 
\frac{-\lambda^2 - 2\lambda\mu \mp u}{\lambda + \mu}$ &
$[ 0 : 1 : 1 ]$ &
$y^2 = (x+2)(x^2-x-1)$ &
$\mathbb{Z}/2\mathbb{Z} \oplus \mathbb{Z}$ &
$(-2, 0), (3, -5)$ \\ \hline
    \end{tabular}
    
    \caption{Rational angles in lattices with squared angle $(1,\tau)=\zeta_{10}^2$}
    \end{table}

\begin{table}
    \begin{tabular}{|c|c|c|c|c|c|c|}
    \hline
    Angle & $u^2=f(\lambda,\mu)$ & $a_\pm, b_\pm$ & $P$ & $E$ & $E(\mathbb{Q})$ & Generators \\ \hline
$\zeta_{12} ^ 2$ &
$(\lambda^2 - \lambda\mu - \mu^2)(\lambda^2 + \lambda\mu + \mu^2)$ &
$\displaystyle \frac{-2\lambda\mu - \mu^2 \mp u}{\lambda + \mu},
\frac{\lambda^2 - \mu^2 \mp u}{\mu}$ &
$[ 1 : -1 : 1 ]$ &
$y^2 = x(x^2 - x + 4)$ &
$\mathbb{Z}/2\mathbb{Z} \oplus \mathbb{Z}$ &
$(0, 0), (4, 8)$ \\ \hline

$\zeta_{12} ^ 4$ &
$(\lambda^2 + \lambda\mu + \mu^2)(5\lambda^2-\lambda\mu-\mu^2)$ &
$\displaystyle \frac{-\lambda^2 - 4\lambda\mu - \mu^2 \mp u}{2\lambda + \mu},
\frac{2\lambda^2 + 2\lambda\mu - \mu^2 \mp u}{-\lambda + \mu}$ &
$[ 1 : 1 : 3 ]$ &
$y^2 = (x+3)(x^2-3x+18)$ &
$\mathbb{Z}/2\mathbb{Z} \oplus \mathbb{Z}$ &
$(-3, 0), (6, 18)$ \\ \hline
    \end{tabular}
    
    \caption{Rational angles in lattices with squared angle $(1,\tau)=\zeta_{12}^3$}
    \end{table}

\begin{table}
    \begin{tabular}{|c|c|c|c|c|c|c|}
    \hline
    Angle & $u^2=f(\lambda,\mu)$ & $a_\pm, b_\pm$ & $P$ & $E$ & $E(\mathbb{Q})$ & Generators \\ \hline
$\zeta_{10}^2$ &
$\lambda (3\lambda-\mu)(11\lambda^2-9\lambda\mu+4\mu^2)$ &
$\begin{array}{c} \displaystyle \frac{-3\lambda^2 + 7\lambda\mu - 2\mu^2 \mp u}{-4\lambda + 2\mu} \\
\displaystyle \frac{5\lambda^2 - 5\lambda\mu \mp u}{2\lambda} \end{array}$ &
$[ 0 : 1 : 0 ]$ &
$y^2 + xy + y = x^3 - x^2 + 2$ &
$\mathbb{Z}/2\mathbb{Z} \oplus \mathbb{Z}$ &
$(-1, 0), (4, 5)$ \\ \hline

$\zeta_{10}^3$ &
$41\lambda^4 - 38\lambda^3\mu + 9\lambda^2\mu^2 + 8\lambda\mu^3 - 4\mu^4$ &
$\begin{array}{c} \displaystyle \frac{\lambda^2 + 11\lambda\mu - 6\mu^2 \mp u}{-2\lambda + 4\mu} \\
\displaystyle \frac{9\lambda^2 - \lambda\mu - 4\mu^2 \mp u}{4\lambda + 2\mu}\end{array}$ &
$[ 1 : 1 : 4 ]$ &
$y^2 + xy + y = x^3 - x^2 + 20x + 22$ &
$\mathbb{Z}/4\mathbb{Z} \oplus \mathbb{Z}$ &
$(4, 10), (79, -740)$ \\ \hline

$\zeta_{10}^4$ &
$(2\lambda^2-2\lambda\mu+\mu^2)(14\lambda^2-6\lambda\mu+\mu^2)$ & 
$\begin{array}{c} \displaystyle \frac{6\lambda^2 - 4\lambda\mu - \mu^2 \mp u}{4\lambda} \\
\displaystyle \frac{-2\lambda^2 + 8\lambda\mu - 3\mu^2 \mp u}{-2\lambda + 2\mu}
\end{array}$ &
$[ 0 : 1 : 1 ]$ &
$y^2 = (x+2)(x^2-x-1)$ &
$\mathbb{Z}/2\mathbb{Z} \oplus \mathbb{Z}$ &
$(-2, 0), (3, -5)$ \\ \hline

$\zeta_{10}^5$ &
$31\lambda^4 - 38\lambda^3\mu + 24\lambda^2\mu^2 - 7\lambda\mu^3 + \mu^4$ &
$\begin{array}{c}\displaystyle \frac{-4\lambda^2 + 6\lambda\mu - \mu^2 \pm u}{-3\lambda + \mu} \\
\displaystyle \frac{-4\lambda^2 + 6\lambda\mu - \mu^2 \mp u}{-3\lambda + \mu}
\end{array}
$ &
$[ 0 : 1 : 1 ]$ &
$y^2 = (x+5)(x^2-5x-25)$ &
$\mathbb{Z}/2\mathbb{Z} \oplus \mathbb{Z}$ &
$(-5, 0), (-\frac{15}{4}, \frac{25}{8})$ \\ \hline
    \end{tabular}
    
    \caption{Rational angles in lattices with squared angle $(1,\tau)=\zeta_{10}$}
    \end{table}

\begin{table}
    \begin{tabular}{|c|c|c|c|c|c|c|}
    \hline
    Angle & $u^2=f(\lambda,\mu)$ & $a_\pm, b_\pm$ & $P$ & $E$ & $E(\mathbb{Q})$ & Generators \\ \hline
$\zeta_{12} ^ 2$ & 
$\lambda(\lambda+\mu)(\lambda^2+\lambda\mu+\mu^2)$ &
$\begin{array}{c} \displaystyle \frac{\lambda^2 - \mu^2 \pm u}{\mu} \\
\displaystyle \frac{-\lambda^2 - 2\lambda\mu \mp u}{\lambda}\end{array}$ &
$[ 0 : 1 : 0 ]$ &
$y^2 = x(x^2-x+1)$ &
$\mathbb{Z}/4\mathbb{Z}$ &
$(1, 1)$ \\ \hline

$\zeta_{12} ^ 3$ &
$(2\lambda^2+2\lambda\mu+\mu^2)(2\lambda^2+2\lambda\mu+5\mu^2)$ &
$\begin{array}{c}\displaystyle \frac{-2\lambda^2 - 6\lambda\mu - \mu^2 \pm u}{4\lambda + 2\mu} \\
\displaystyle \frac{2\lambda^2 - 2\lambda\mu - 3\mu^2 \mp u}{2\mu} \end{array}$ &
$[ 1 : 0 : 2 ]$ &
$y^2 = (x-2)(x+1)(x+2)$ &
$\mathbb{Z}/2\mathbb{Z} \oplus \mathbb{Z}/2\mathbb{Z}$ &
$(2, 0), (-1, 0)$ \\ \hline

$\zeta_{12} ^ 4$ &
$(\lambda^2+\lambda\mu+\mu^2)(5\lambda^2+5\lambda\mu+2\mu^2)$ &
$\begin{array}{c}\displaystyle \frac{\lambda^2 - 2\lambda\mu - 2\mu^2 \pm u}{\lambda + 2\mu} \\
\displaystyle \frac{-\lambda^2 - 4\lambda\mu - \mu^2 \mp u}{2\lambda + \mu}\end{array}$ &
$[ 1 : 1 : 6 ]$ &
$y^2 = (x+3)(x^2-3x-9)$ &
$\mathbb{Z}/2\mathbb{Z} \oplus \mathbb{Z}$ &
$(-3, 0), (6, 9)$ \\ \hline

$\zeta_{12} ^ 6$ &
$(\lambda^2+\lambda\mu+\mu^2)(2\lambda^2+2\lambda\mu+\mu^2)$ &
$\begin{array}{c}\displaystyle \frac{-2\lambda\mu - \mu^2 \pm u}{\lambda + \mu}\\
\displaystyle \frac{-2\lambda\mu - \mu^2 \mp u}{\lambda + \mu} \end{array}$ &
$[ 0 : 1 : 1 ]$ &
$y^2 = (x+1)(x^2-2x-2)$ &
$\mathbb{Z}/2\mathbb{Z} \oplus \mathbb{Z}$ &
$(-1, 0), (-\frac{3}{4}, \frac{1}{8})$ \\ \hline
    \end{tabular}
    
    \caption{Rational angles in lattices with squared angle $(1,\tau)=\zeta_{12}$}
    \end{table}

\end{landscape}

\newpage

\section{Pictures of the four spaces of type 3+3}\label{sec:figure:3+3}

\begin{figure}[H]
\begin{center}
\scalebox{0.9}{
\begin{tikzpicture}[scale=4]
\coordinate (O) at (0,0);

\def\numerolati{8}
\node (pol) [draw, thick, blue!90!black,rotate=90,minimum size=9.5cm,regular polygon, regular polygon sides=\numerolati, rotate=180+180/\numerolati] at (0,0) {}; 

\foreach \n [count=\nu from 0, remember=\n as \lastn, evaluate={\nu+\lastn}] in {1,2,...,\numerolati} 
\node[anchor=(\n-1)*(360/\numerolati)+180]at(pol.corner \n){$\zeta^{\nu}$};

\node (t) [circle, fill=red, inner sep=0pt, minimum size=4pt, label=90:$\tau$]  at ($(pol.corner 3)-(pol.corner 2) + (pol.corner 1)$) {};

\coordinate (A) at ($(pol.corner 4)+sqrt(2)*(pol.corner 4) - sqrt(2)*(pol.corner 3)$);
\coordinate (B) at ($(pol.corner 3)+sqrt(2)*(pol.corner 3) - sqrt(2)*(pol.corner 2)$) {};
\coordinate (C)  at ($(pol.corner 2)+sqrt(2)*(pol.corner 2) - sqrt(2)*(pol.corner 1)$) {};
\coordinate (D) at ($(pol.corner 1)+sqrt(2)*(pol.corner 1) - sqrt(2)*(pol.corner 8)$) {};
\coordinate (E) at ($(pol.corner 8)+sqrt(2)*(pol.corner 8) - sqrt(2)*(pol.corner 7)$) {};
\coordinate (F) at ($(pol.corner 7)+sqrt(2)*(pol.corner 7) - sqrt(2)*(pol.corner 6)$) {};
\coordinate (G) at ($(pol.corner 6)+sqrt(2)*(pol.corner 6) - sqrt(2)*(pol.corner 5)$) {};
\coordinate (H) at ($(pol.corner 5)+sqrt(2)*(pol.corner 5) - sqrt(2)*(pol.corner 4)$) {};
\draw [thick, blue!90!black] (B)--(C);
\draw [thick, blue!90!black] (C)--(D);
\draw [thick, blue!90!black] (D)--(E);
\draw [thick, blue!90!black] (E)--(F);
\draw [thick, blue!90!black] (F)--(G);
\draw [thick, blue!90!black] (G)--(H);
\draw [thick, blue!90!black] (H)--(A);
\draw [thick, blue!90!black] (A)--(B);

\draw [thin] (C)--(E);
\draw [thin] (D)--(F);

\coordinate (tmeno1) at ($(t)-(pol.corner 1)$);
\coordinate (tpiu1) at ($(t)+(pol.corner 1)$);
\node [circle, fill=red, inner sep=0pt, minimum size=4pt, label=0:$\tau+1$]  at (tpiu1) {};
\node (tmeno2) [circle, fill=red, inner sep=0pt, minimum size=4pt, label=180:$\tau-2$]  at ($(t)-2*(pol.corner 1)$) {};
\node (tmeno1/2) [circle, fill=red, inner sep=0pt, minimum size=4pt, label=90:$\tau-1/2$]  at ($(t)-1/2*(pol.corner 1)$) {};

\draw [thin] (B)--(tmeno1/2);

\draw [thin] (O)--(pol.corner 2);
\draw [thin, name path=dO3] (pol.corner 3)--(pol.corner 5);
\draw [thin, name path=dO3] (pol.corner 4)--(pol.corner 6);
\draw [thin, name path=dO3] (pol.corner 1)--(pol.corner 4);
\draw [thin, name path=dO3] (pol.corner 3)--(pol.corner 8);   
\draw [thin, name path=dO3] (tpiu1)--(pol.corner 7);
\draw [thin, name path=dO3] (tmeno2)--(pol.corner 8);        
\draw [thin, name path=dO3] (tmeno2)--(pol.corner 4);   
             
 \node [circle, fill=red, inner sep=0pt, minimum size=4pt, label=90:$\tau-1$] at (tmeno1) {};             
             
\draw [->,thick,red, name path=ac] (O)--(tmeno1);
\draw [->,thick,red, name path=ac] (O)--(t); 
\draw [->,thick,red, name path=ac] (O)--(pol.corner 1);
\draw [->,thick,green, name path=ac] (O)--(tmeno2);
\draw [->,thick,green, name path=ac] (O)--(tmeno1/2);
\draw [->,thick,green, name path=ac] (O)--(tpiu1);

 \node [circle, fill=red, inner sep=0pt, minimum size=4pt, label=-90:$O$] at (O) {};

\end{tikzpicture}
}
\end{center}
\caption{$\tau=\zeta_8^2-\zeta_8+1$}\label{fig:ottagono:1}
\end{figure}

\newpage

\begin{figure}[H]
\begin{center}
\scalebox{0.9}{
\begin{tikzpicture}[scale=4]
\coordinate (O) at (0,0);
 \node [circle, fill=red, inner sep=0pt, minimum size=4pt, label=-90:$O$] at (O) {};
 
 \def\numerolati{10}
\node (pol) [draw, thick, blue!90!black,rotate=90,minimum size=10cm,regular polygon, regular polygon sides=\numerolati, rotate=198] at (0,0) {}; 

\foreach \n [count=\nu from 0, remember=\n as \lastn, evaluate={\nu+\lastn}] in {1,2,...,10} 
\node[anchor=(\n-1)*(360/\numerolati)+180]at(pol.corner \n){$\zeta^{\nu}$};

\node (t) [circle, fill=red, inner sep=0pt, minimum size=4pt, label=90:$\tau$]  at ($(pol.corner 5) + 2*(pol.corner 3) + 2*(pol.corner 1)$) {};

\node (tmeno1) [circle, fill=red, inner sep=0pt, minimum size=4pt, label=90:$\tau-1$]  at ($(t)-(pol.corner 1)$) {};

\node (tmeno1/3) [circle, fill=red, inner sep=0pt, minimum size=4pt, label=90:$\tau-1/3$]  at ($(t)-1/3*(pol.corner 1)$) {};
\node (t/3meno4/3) [circle, fill=red, inner sep=0pt, minimum size=4pt, label=90:$\frac{\tau}{3}-\frac{4}{3}$]  at ($1/3*(t)-4/3*(pol.corner 1)$) {};

\node (4/5tmeno1) [circle, fill=red, inner sep=0pt, minimum size=4pt, label=90:$\frac{4}{5}\tau-1$]  at ($4/5*(t)-1*(pol.corner 1)$) {};

\draw [->,thick,red, name path=ac] (O)--(tmeno1);
\draw [->,thick,red, name path=ac] (O)--(t); 
\draw [->,thick,red, name path=ac] (O)--(pol.corner 1);
\draw [->,thick,green, name path=ac] (O)--(tmeno1/3);
\draw [->,thick,green, name path=ac] (O)--(t/3meno4/3);
\draw [->,thick,green, name path=ac] (O)--(4/5tmeno1);

\draw [thin, name path=d18] (pol.corner 10)--(t);
\draw [thin, name path=d18] (pol.corner 5)--(t);
\draw [thin, name path=d18] (pol.corner 10)--(tmeno1);
\draw [thin, name path=d18] (pol.corner 6)--(tmeno1);

\draw [thin] (pol.corner 6)--(t/3meno4/3);

 \coordinate (B) at ($2/3*(pol.corner 6) + sqrt(5)/3*(pol.corner 6) + 1/3*(pol.corner 7) -sqrt(5)/3*(pol.corner 7)$);
  \coordinate (C) at ($2/3*(pol.corner 7) + sqrt(5)/3*(pol.corner 7) + 1/3*(pol.corner 8) -sqrt(5)/3*(pol.corner 8)$);
   \coordinate (D) at ($2/3*(pol.corner 8) + sqrt(5)/3*(pol.corner 8) + 1/3*(pol.corner 9) -sqrt(5)/3*(pol.corner 9)$);
   \coordinate (E) at ($2/3*(pol.corner 9) + sqrt(5)/3*(pol.corner 9) + 1/3*(pol.corner 10) -sqrt(5)/3*(pol.corner 10)$);
   \coordinate (F) at ($2/3*(pol.corner 10) + sqrt(5)/3*(pol.corner 10) + 1/3*(pol.corner 1) -sqrt(5)/3*(pol.corner 1)$);
   \coordinate (G) at ($2/3*(pol.corner 1) + sqrt(5)/3*(pol.corner 1) + 1/3*(pol.corner 2) -sqrt(5)/3*(pol.corner 2)$);
   \coordinate (H) at ($2/3*(pol.corner 2) + sqrt(5)/3*(pol.corner 2) + 1/3*(pol.corner 3) -sqrt(5)/3*(pol.corner 3)$);
   \coordinate (I) at ($2/3*(pol.corner 3) + sqrt(5)/3*(pol.corner 3) + 1/3*(pol.corner 4) -sqrt(5)/3*(pol.corner 4)$);
   \coordinate (J) at ($2/3*(pol.corner 4) + sqrt(5)/3*(pol.corner 4) + 1/3*(pol.corner 5) -sqrt(5)/3*(pol.corner 5)$);

 \draw [thick, blue!90!black] (t/3meno4/3)--(B);
\draw [thick, blue!90!black] (B)--(C);
\draw [thick, blue!90!black] (C)--(D);
\draw [thick, blue!90!black] (D)--(E);
\draw [thick, blue!90!black] (E)--(F);
\draw [thick, blue!90!black] (F)--(G);
\draw [thick, blue!90!black] (G)--(H);
\draw [thick, blue!90!black] (H)--(I);
\draw [thick, blue!90!black] (I)--(J);
 \draw [thick, blue!90!black] (J)--(t/3meno4/3);

  \draw [thin, name path=d18] (F)--(J);
   \draw [thin, name path=d18] (C)--(I);

\end{tikzpicture}
}

\end{center}
\caption{$\tau=\zeta_{10}^4+2\zeta_{10}+2$}\label{fig:pentagono:1}
\end{figure}

\newpage

\begin{figure}[H]
\begin{center}
\scalebox{0.9}{
\begin{tikzpicture}[scale=2]
\coordinate (O) at (0,0);
 \node [circle, fill=red, inner sep=0pt, minimum size=4pt, label=-90:$O$] at (O) {};
 
 \def\numerolati{10}
\node (pol) [draw, thick, blue!90!black,rotate=90,minimum size=6cm,regular polygon, regular polygon sides=\numerolati, rotate=198] at (0,0) {}; 

\foreach \n [count=\nu from 0, remember=\n as \lastn, evaluate={\nu+\lastn}] in {1,2,...,10} 
\node[anchor=(\n-1)*(360/\numerolati)+180]at(pol.corner \n){$\zeta^{\nu}$};

\node (t) [circle, fill=red, inner sep=0pt, minimum size=4pt, label=90:$\tau$]  at ($(pol.corner 3) - 2*(pol.corner 2) + 2*(pol.corner 1)$) {};

\node (tmeno1) [circle, fill=red, inner sep=0pt, minimum size=4pt, label=90:$\tau-1$]  at ($(t)-(pol.corner 1)$) {};

\node (tmeno4) [circle, fill=red, inner sep=0pt, minimum size=4pt, label=90:$\tau-4$]  at ($(t)-4*(pol.corner 1)$) {};

\node (3tmeno1) [circle, fill=red, inner sep=0pt, minimum size=4pt, label=90:$3\tau-1$]  at ($3*(t)-(pol.corner 1)$) {};
\node (4tmeno5) [circle, fill=red, inner sep=0pt, minimum size=4pt, label=90:$4\tau-5$]  at ($4*(t)-5*(pol.corner 1)$) {};

\coordinate (B) at ($8*(pol.corner 3) - 4*sqrt(5)*(pol.corner 3) + 4*sqrt(5)*(pol.corner 6) -7*(pol.corner 6)$);
\coordinate (B) at ($8*(pol.corner 4) - 4*sqrt(5)*(pol.corner 4) + 4*sqrt(5)*(pol.corner 7) -7*(pol.corner 7)$);
\coordinate (C) at ($8*(pol.corner 5) - 4*sqrt(5)*(pol.corner 5) + 4*sqrt(5)*(pol.corner 8) -7*(pol.corner 8)$);
\coordinate (D) at ($8*(pol.corner 6) - 4*sqrt(5)*(pol.corner 6) + 4*sqrt(5)*(pol.corner 9) -7*(pol.corner 9)$);
\coordinate (E) at ($8*(pol.corner 7) - 4*sqrt(5)*(pol.corner 7) + 4*sqrt(5)*(pol.corner 10) -7*(pol.corner 10)$);
\coordinate (F) at ($8*(pol.corner 8) - 4*sqrt(5)*(pol.corner 8) + 4*sqrt(5)*(pol.corner 1) -7*(pol.corner 1)$);
\coordinate (G) at ($8*(pol.corner 9) - 4*sqrt(5)*(pol.corner 9) + 4*sqrt(5)*(pol.corner 2) -7*(pol.corner 2)$);
\coordinate (H) at ($8*(pol.corner 10) - 4*sqrt(5)*(pol.corner 10) + 4*sqrt(5)*(pol.corner 3) -7*(pol.corner 3)$);
\coordinate (I) at ($8*(pol.corner 1) - 4*sqrt(5)*(pol.corner 1) + 4*sqrt(5)*(pol.corner 4) -7*(pol.corner 4)$);
\coordinate (J) at ($8*(pol.corner 2) - 4*sqrt(5)*(pol.corner 2) + 4*sqrt(5)*(pol.corner 5) -7*(pol.corner 5)$);
\coordinate (K) at ($8*(pol.corner 3) - 4*sqrt(5)*(pol.corner 3) + 4*sqrt(5)*(pol.corner 6) -7*(pol.corner 6)$);

\draw [thick, blue!90!black] (B)--(C);
\draw [thick, blue!90!black] (C)--(D);
\draw [thick, blue!90!black] (D)--(E);
\draw [thick, blue!90!black] (E)--(F);
\draw [thick, blue!90!black] (F)--(G);
\draw [thick, blue!90!black] (G)--(H);
\draw [thick, blue!90!black] (H)--(I);
\draw [thick, blue!90!black] (I)--(J);
\draw [thick, blue!90!black] (J)--(K);
\draw [thick, blue!90!black] (K)--(B);

\draw [thin, name path=d18] (pol.corner 1)--(pol.corner 8);
\draw [thin, name path=d48] (pol.corner 4)--(pol.corner 8);
\draw [thin, name path=d610] (pol.corner 6)--(3tmeno1);
\draw [thin, name path=d07] (O)--(pol.corner 7);
\draw [thin, name path=d310] (pol.corner 3)--(pol.corner 10);

\draw [thin, name path=d610] (pol.corner 3)--(tmeno4);

\draw [thin, name path=d610] (tmeno4)--(t);

\draw [thin, name path=d610] (pol.corner 3)--(K);
\draw [thin, name path=d610] (E)--(K);
\draw [thin, name path=d610] (D)--(H);
\draw [thin, name path=d610] (tmeno4)--(C);
\draw [thin, name path=d610] (tmeno4)--(H);

\draw [->,thick,red, name path=ac] (O)--(tmeno1);
\draw [->,thick,red, name path=ac] (O)--(t); 
\draw [->,thick,red, name path=ac] (O)--(pol.corner 1);
\draw [->,thick,green, name path=ac] (O)--(3tmeno1);
\draw [->,thick,green, name path=ac] (O)--(tmeno4);
\draw [->,thick,green, name path=ac] (O)--(4tmeno5);

\end{tikzpicture}
}
\end{center}
\caption{$\tau=\zeta_{10}^2-2\zeta_{10}+2$}\label{fig:pentagono:2}
\end{figure}

\newpage

\begin{figure}[H]
\begin{center}
\scalebox{0.9}{
\begin{tikzpicture}[scale=2]
\coordinate (O) at (0,0);
 \node [circle, fill=red, inner sep=0pt, minimum size=4pt, label=-90:$O$] at (O) {};
 
 \def\numerolati{12}
\node (pol) [draw, thick, blue!90!black,rotate=90,minimum size=12cm,regular polygon, regular polygon sides=\numerolati, rotate=180+180/\numerolati] at (0,0) {}; 

\foreach \n [count=\nu from 0, remember=\n as \lastn, evaluate={\nu+\lastn}] in {1,2,...,\numerolati} 
\node[anchor=(\n-1)*(360/\numerolati)+180]at(pol.corner \n){$\zeta^{\nu}$};

\node (t) [circle, fill=red, inner sep=0pt, minimum size=4pt, label=90:$\tau$]  at ($-2*(pol.corner 4) + 2*(pol.corner 3) + (pol.corner 2) - (pol.corner 1)$) {};

\node (tmeno1) [circle, fill=red, inner sep=0pt, minimum size=4pt, label=90:$\tau-1$]  at ($(t)-(pol.corner 1)$) {};

\node (t/3piu1) [circle, fill=red, inner sep=0pt, minimum size=4pt, label=90:$\frac{\tau}{3}+1$]  at ($1/3*(t)+1*(pol.corner 1)$) {};

\node (tmeno2) [circle, fill=red, inner sep=0pt, minimum size=4pt, label=90:$\tau-2$]  at ($(t)-2*(pol.corner 1)$) {};
\node (tmeno3/4) [circle, fill=red, inner sep=0pt, minimum size=4pt, label=90:$\tau-3/4$]  at ($(t)-3/4*(pol.corner 1)$) {};

\draw [thin, name path=d15] (t)--(tmeno2);

 \draw [thin] (pol.corner 2)--(pol.corner 12);
  \draw [thin] (pol.corner 5)--(O);
    \draw [thin] (pol.corner 3)--(pol.corner 8); 
        \draw [thin] (pol.corner 7)--(pol.corner 2); 
           \draw [thin] (pol.corner 9)--(tmeno2);   
             \draw [thin] (pol.corner 5)--(tmeno2);          

\draw [thin] (pol.corner 8)--(t/3piu1);       
 \draw [thin] (pol.corner 3)--(t/3piu1);

\draw [->,thick,red, name path=ac] (O)--(tmeno1);
\draw [->,thick,red, name path=ac] (O)--(t); 
\draw [->,thick,red, name path=ac] (O)--(pol.corner 1);
\draw [->,thick,green, name path=ac] (O)--(t/3piu1);
\draw [->,thick,green, name path=ac] (O)--(tmeno2);
\draw [->,thick,green, name path=ac] (O)--(tmeno3/4);

\coordinate (B) at ($(pol.corner 3) + 2*sqrt(3)/3*(pol.corner 3) -2*sqrt(3)/3*(pol.corner 4)$);
\coordinate (C) at ($(pol.corner 4) + 2*sqrt(3)/3*(pol.corner 4) -2*sqrt(3)/3*(pol.corner 5)$);
\coordinate (D) at ($(pol.corner 5) + 2*sqrt(3)/3*(pol.corner 5) -2*sqrt(3)/3*(pol.corner 6)$);
\coordinate (E) at ($(pol.corner 6) + 2*sqrt(3)/3*(pol.corner 6) -2*sqrt(3)/3*(pol.corner 7)$);
\coordinate (F) at ($(pol.corner 7) + 2*sqrt(3)/3*(pol.corner 7) -2*sqrt(3)/3*(pol.corner 8)$);
\coordinate (G) at ($(pol.corner 8) + 2*sqrt(3)/3*(pol.corner 8) -2*sqrt(3)/3*(pol.corner 9)$);
\coordinate (H) at ($(pol.corner 9) + 2*sqrt(3)/3*(pol.corner 9) -2*sqrt(3)/3*(pol.corner 10)$);
\coordinate (I) at ($(pol.corner 10) + 2*sqrt(3)/3*(pol.corner 10) -2*sqrt(3)/3*(pol.corner 11)$);
\coordinate (J) at ($(pol.corner 11) + 2*sqrt(3)/3*(pol.corner 11) -2*sqrt(3)/3*(pol.corner 12)$);
\coordinate (K) at ($(pol.corner 12) + 2*sqrt(3)/3*(pol.corner 12) -2*sqrt(3)/3*(pol.corner 1)$);
\coordinate (L) at ($(pol.corner 1) + 2*sqrt(3)/3*(pol.corner 1) -2*sqrt(3)/3*(pol.corner 2)$);

\draw [thick, blue!90!black] (t/3piu1)--(B);
\draw [thick, blue!90!black] (B)--(C);
\draw [thick, blue!90!black] (C)--(D);
\draw [thick, blue!90!black] (D)--(E);
\draw [thick, blue!90!black] (E)--(F);
\draw [thick, blue!90!black] (F)--(G);
\draw [thick, blue!90!black] (G)--(H);
\draw [thick, blue!90!black] (H)--(I);
\draw [thick, blue!90!black] (I)--(J);
\draw [thick, blue!90!black] (J)--(K);
\draw [thick, blue!90!black] (K)--(L);
\draw [thick, blue!90!black] (L)--(t/3piu1);
\draw [thin] (O)--(C);
\draw [thin] (E)--(G);
\draw [thin] (F)--(H);

\end{tikzpicture}
}
\end{center}
\caption{$\tau=-2\zeta_{12}^3+2\zeta_{12}^2+\zeta_{12}-1$}\label{fig:dodecagono:1}
\end{figure}
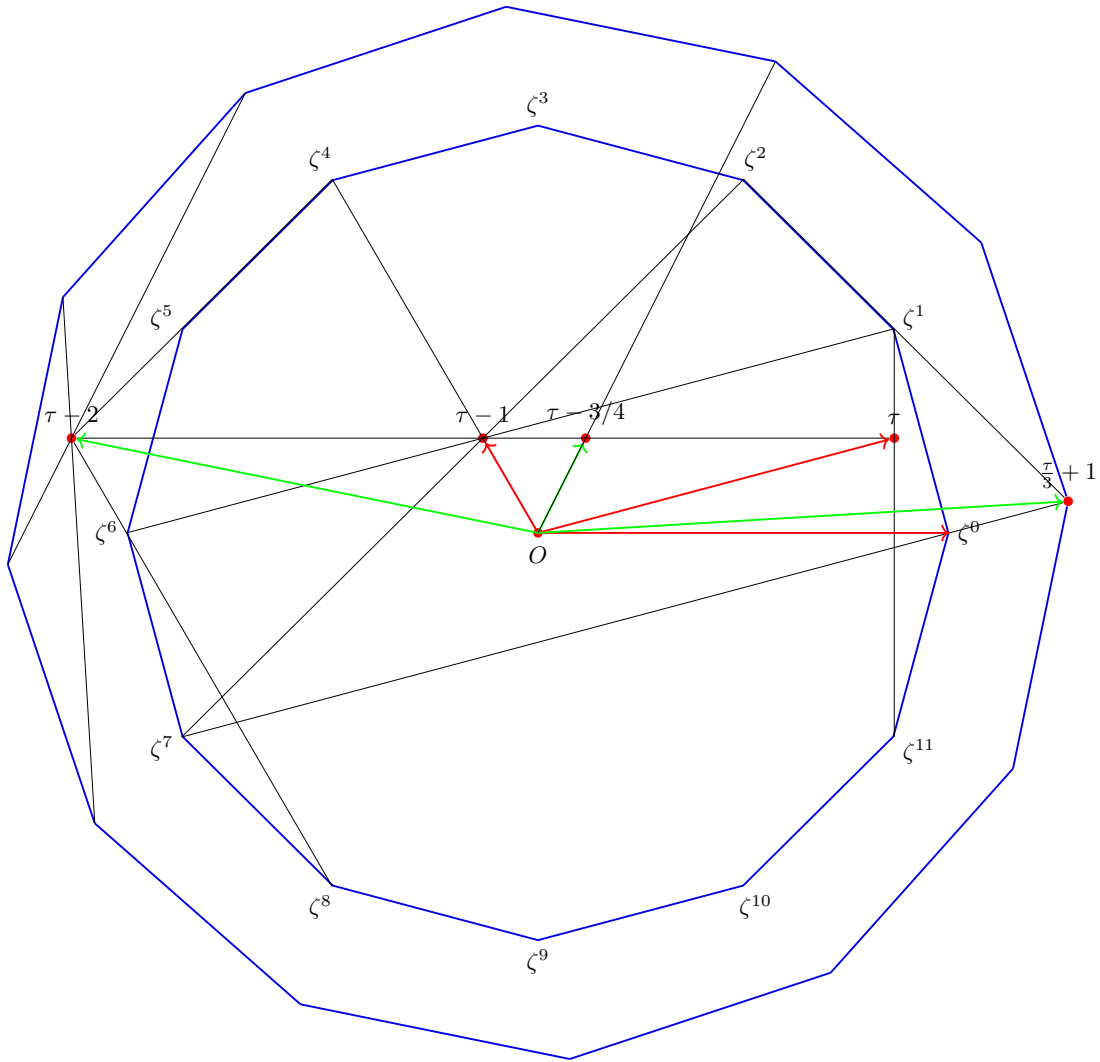

\newpage

\appendix
\section{A euclidean proof}

A euclidean explanation for the family $(\zeta_6, \zeta_6^2, \zeta_6^2, 2d, \frac{1}{2}c, c, d)$ from Theorem~\ref{thm:abcdGeometric} (6), at least in the case $cd < 0$, is provided by the following geometric construction, in which the points $D', D, T, M_C, C, O$ correspond respectively to $\tau + 2d, \tau +d, \tau, \tau + \frac{1}{2}c, \tau+c, 0$.
\begin{proposition}
Consider 5 points on a line $r$, labelled $D', D, T, M_C, C$ in this order, and let $O$ be a point that does not lie on $r$. Assume that:
\begin{enumerate}
\item the angle $\angle OTD$ is $30^\circ$;
\item the angle $\angle{COD}$ is $60^\circ$;
\item $M_C$ is the midpoint of $TC$;
\item the angle $\angle{M_C O D'}$ is $120^\circ$. %
\end{enumerate}
Then $D$ is the midpoint of the segment $D'T$.%
\end{proposition}
\begin{proof}\phantom{}

\begin{center}
\includegraphics[scale=1.2]{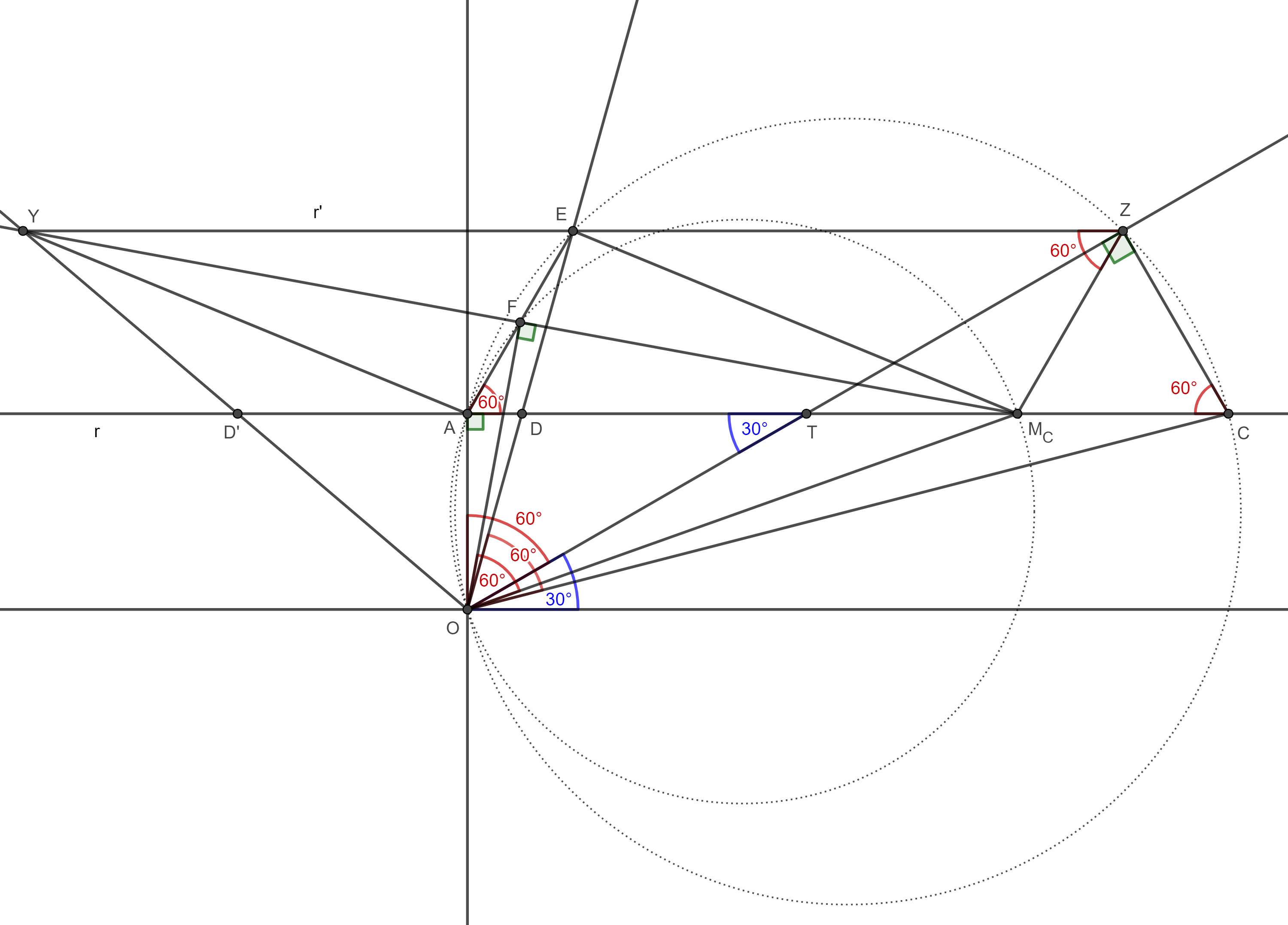}
\end{center}

Letting $A$ be the projection of $O$ on $r$, define a point $E$ by requiring that the triangles $EAO$ and $CTO$ are similar. In other words, $EAO$ is the image of $CTO$ under the unique spiral similarity of centre $O$ that carries $T$ to $A$. Under this spiral similarity, $M_C$ is sent to the midpoint $F$ of $AE$. The similarity implies $\angle{OAE} = \angle{OTC} = \angle{OAT} + \angle{TOA} = \angle{OAT} + (90^\circ-30^\circ)$, and on the other hand $\angle{OAE} = \angle{OAT} + \angle{TAE}$, whence $\angle{TAE} = 60^\circ$. We also have $\angle{M_COF} = (\angle{COD}-\angle{COM_C}) + \angle{DOF}$, and the spiral similarity gives $\angle{DOF} = \angle{COM_C}$, so $\angle{M_COF} = \angle{COD} = 60^\circ$.
Thus, in the quadrilateral $FAOM_C$ we have $\angle{M_COF} = 60^\circ = \angle{M_CAF}$, so the points $F, A, O, M_C$ all lie on the same circle. In particular, $\angle{OFM_C} = \angle{OAM_C}$ is a right angle.
We further have $\angle{COE} = \angle{COD} = 60^\circ = \angle{CAE}$, so the quadrilateral $COAE$ is also cyclic.

Let $Y$ be the intersection of the ray $OD'$ with the ray $M_CF$. The angle $\angle{M_COY} = \angle{M_COD'}$ is equal to $120^\circ$ by assumption. Since $FAOM_C$ is cyclic, the angle $\angle{YM_CO} = \angle{FM_CO}$ is supplementary to $\angle{FAO}=\angle{FAM_C} + \angle{M_CAO} = 150^\circ$, so $\angle{YM_CO}=30^\circ$. Thus, looking at the triangle $M_COY$ we obtain $\angle{OYM_C} = 180^\circ - \angle{M_COY}-\angle{YM_CO}=30^\circ$, hence the triangle $M_COY$ is isosceles. We already know that $OF$ is the altitude relative to its base, hence we obtain that $OF$ is also its median: $M_CF=FY$. As already noticed, we also have $EF=FA$. Thus, the diagonals of the quadrilateral $AM_CEY$ bisect each other, and therefore $AM_CEY$ is a parallelogram. In particular, the line $r'$ through $E$ and $Y$ is parallel to $r$. Let $Z$ be the point of intersection between $r'$ and the ray $OT$. We now show that $ZEAM_C$ is a parallelogram.

Note first that since $r, r'$ are parallel we have $\angle{TZE}=\angle{OTD}=30^\circ=180^\circ-\angle{OAE}$, so $Z$ lies on the same circle as $E, A, O$ (which we know to also pass through $C$). Hence, $\angle{CZO}=\angle{CAO}=90^\circ$. In particular, the triangle $TZC$ has $\angle{TZC} = 90^\circ$ and $\angle CTZ = \angle{OTA}=30^\circ$, so $\angle{ZCT}=60^\circ$. In this triangle, the segment $ZM_C$ is the median relative to the hypotenuse, so $ZM_C=\frac{1}{2}TC=M_CC$ and $CZ = \frac{1}{2}CT=M_CC$: the triangle $ZCM_C$ is therefore equilateral, and $M_CZC=60^\circ$. Looking at the circle through $Z, E, A, O, C$ we obtain $\angle{EZC} = 180^\circ-\angle{COE} = 180^\circ-\angle{COD}=120^\circ$, and by difference with $\angle{M_CZC}=60^\circ$ we get $\angle{M_CZE}=60^\circ$. Thus, $M_CZ$ is parallel to $EA$, and since $ZE$ is parallel to $AM_C$ we obtain as desired that $ZEAM_C$ is a parallelogram. 

Finally, the triangles $TOD'$ and $ZOY$ are similar (their sides being parallel in pairs), and the homotethy of centre $O$ that takes $T$ to $Z$ and $D'$ to $Y$ also takes $D$ to $E$. Hence, $D$ is the midpoint of $D'T$ if and only if $E$ is the midpoint of $YZ$. We have already shown that $EYAM_C$ and $ZEAM_C$ are parallelograms, so we have $YE=AM_C=EZ$, and the claim follows.
\end{proof}

\bibliography{biblio}
\end{document}